\documentclass[12pt]{amsart}
\usepackage{lineno,hyperref}
\usepackage{amsmath,amsfonts,amsthm,mathrsfs,mathtools}
\usepackage{amssymb,latexsym}
\usepackage{cases}
\usepackage[numbers]{natbib}
\usepackage{booktabs}
\usepackage{bm} 
\usepackage{cleveref}
\usepackage{paralist}
\usepackage{multirow}
\usepackage{enumerate}
\usepackage{adjustbox}
\usepackage[shortlabels]{enumitem}
\setlist[enumerate]{nosep}
\usepackage{array}

\usepackage[figuresright]{rotating}

\usepackage[section]{placeins}

\renewcommand\appendix{\setcounter{secnumdepth}{-2}}

\usepackage{url,cancel}
\usepackage{float}

\usepackage{caption,enumitem}
\newtheorem{thm}{Theorem}[section]

\newtheorem{lem}[thm]{Lemma}
\newtheorem{prop}[thm]{Proposition}
\newtheorem{cor}[thm]{Corollary}

\theoremstyle{remark}
\newtheorem{re}[thm]{Remark}

\theoremstyle{definition}
\newtheorem{ex}[thm]{Example}
\newtheorem{defn}[thm]{Definition}%[section]

\newcommand{\lhs}{\mbox{LHS}}
\newcommand{\rhs}{\mbox{RHS}}
\newcommand{\cls}{\mbox{cls}}
\newcommand{\spn}{\mbox{spn}}
\newcommand{\gen}{\mbox{gen}}
\newcommand{\ord}{\mathrm{ord}}

\newcommand{\co}{\mbox{Co}}
\newcommand{\exc}{\mbox{Ex}}

\newcommand{\rank}{\mbox{rank}}

\newcommand{\gk}{\mbox{GK}}
\newcommand{\egk}{\mbox{EGK}}
\newcommand{\jor}{\mbox{Jor}}
\newcommand{\tr}{\mbox{Tr}}

\newcommand{\chartic}{\mbox{char}}

\newcommand{\pr}{\mathrm{pr}}

\def\rep{{\to\!\!\! -}}

\makeatletter

\@addtoreset{equation}{section}
\makeatother
\numberwithin{equation}{section}

\begin{document}
\author{Zilong He}
\address{Department of Mathematics,	Dongguan University of Technology, Dongguan 523808, China}
\email{zilonghe@connect.hku.hk}
	
\title[]{Local information of ADC quadratic lattices over algebraic number fields}
\thanks{ }
\date{\today}
\begin{abstract}
 	In the paper, we mainly determine the structures, counting formulas, and density sets of representations for binary and ternary ADC quadratic lattices over arbitrary non-archimedean local fields.
 		
 	In the binary case, we show that under certain conditions, there are finitely many primitive positive definite ADC lattices and infinitely many non-primitive ones. We also provide concise formulas for local densities and masses using invariants from BONGs theory, and show that these invariants completely determine the local densities over arbitrary non-archimedean local fields. Moreover, we compute the corresponding local quantities for ADC lattices over algebraic number fields.
 		 		
 	In the ternary case, we characterize the codeterminant set of spinor exceptions and integral spinor norm groups for ADC lattices over arbitrary non-archimedean local fields. Based on these results, we further establish some sufficient conditions on indefinite ADC lattices over algebraic number fields.
\end{abstract}
\maketitle
\section{Introduction}	
	Motivated by Clark’s concept of ADC quadratic forms \cite{clark_ADC-I-2012}, we introduced the $n$-ADC property for higher-dimensional representations by quadratic forms, and studied its classification over non-archimedean local fields and algebraic number fields for $n\ge 2$ in \cite{He25}. When $n=1$, this notion agrees with Clark's definition. Precisely, let $Q$ be an $m$-ary positive definite quadratic form over $\mathbb{Z}$. Then $Q$ is called \textit{ADC} if it represents every positive integer $a$ over $\mathbb{Z}$ for which $Q$ represents $a$ over $\mathbb{Q}$. This concept originally stems from a result obtained independently by Aubry, Davenport, and Cassels, which can be restated as follows: if the quadratic form $Q$ with an integral Gram matrix satisfies the Euclidean property, meaning that for all $x\in \mathbb{Q}^{m}$, there exists $y\in \mathbb{Z}^{m}$ such that $Q(x-y)<1$, then it is ADC. The ADC property is closely related to two other well-known and widely studied topics: universal forms and regular forms, coined by Dickson \cite{dickson_ternary_1927,dickson_universal_1929}. Recall that a positive definite quadratic form over $\mathbb{Z}$ is called \textit{universal} if it represents all positive integers, and \textit{regular} if it represents every positive integer that is locally represented. In fact, a universal form is ADC, and an ADC form is regular (cf. \eqref{equiv:uni-adc-regular}). Thus, the ADC property can be viewed as a transitional concept between universality and regularity. In a later paper, Clark and Jagy \cite{clark_ADC-II-2014} further explored classification problems on ADC quadratic forms over various fields, particularly non-dyadic and $2$-adic local fields, as well as the field of rational numbers. More historical background and details can be found in \cite{clark_ADC-I-2012,clark_ADC-II-2014,He25}.
	
	In this paper, we mainly study the classification problem of ADC quadratic lattices over non-archimedean local fields and algebraic number fields that have not been covered by previous work. To be precise, we characterize ADC lattices over arbitrary non-archimedean local fields,  especially on the unary, binary and ternary cases. In  non-dyadic cases, we present a necessary and sufficient condition, different from that in \cite{clark_ADC-II-2014}, to emphasize the connection between ADC lattices and maximal lattices (Theorems \ref{thm:locallyn-ADC-binary}(i), (ii) and \ref{thm:locallyn-ADC-ternary}(i)). In fact, all  maximal lattices have been enumerated in \cite{HeHu2,hhx_indefinite_2021}; see Section \ref{sec:maximal-lattices} for further discussion. In dyadic cases, we formulate the classification using classical Jordan splittings (Theorems \ref{thm:locallyn-ADC-binary}(i), (iii) and \ref{thm:locallyn-ADC-ternary}(ii)), while adopting Beli's BONGs theory that have been effectively applied to representation problems of quadratic lattices in \cite{beli_universal_2020,He22,He25,HeHu2}. However, unlike the higher-dimensional representation cases studied in \cite{He25}, there appear to be more types of non-maximal ADC lattices. To describe this, we also provide formulas for counting these ADC lattices (Theorem \ref{thm:count-sol-23}). In  \cite{clark-density-2024}, Clark, Pollack, Rouse, and Thompson studied the densities of representation sets of quadratic forms. Using their results, we provide explicit formulas of such representation densities for ADC lattices over arbitrary non-archimedean local fields (Theorem \ref{thm:density}). 
	
	In the global situation, we establish the equivalent conditions for unary and binary ADC lattices over algebraic number fields (Theorem \ref{thm:globally ADC-rank2}), and some finiteness and infiniteness results for binary ADC lattices (Theorem \ref{thm:scale-ADC}, Corollaries \ref{cor:primitive-finite} and \ref{cor:imprimitive-infinite}). 
	
	In Section \ref{sec:mass}, we study the local density and mass formula to analyze the class number one condition for binary ADC lattices, given in Theorem \ref{thm:globally ADC-rank2}(iii). In non-dyadic fields, Ikeda and Katsurada \cite{ikeda-katsurada_formual-Siegel_2022} showed that the local density is completely determined by the GK-invariant introduced by Gross and Keating \cite{gk_inter_1993}. In \cite{cho_localdensity-2020}, Cho proved that for unramified dyadic local fields, the local density can be determined by the extended GK data. However, these invariants fail for binary quadratic forms over ramified dyadic local fields, as pointed out by Ikeda and Katsurada in \cite[Example A.1]{cho_localdensity-2020}.
	
	Motivated by their work and building on Pfeuffer's results in \cite{pfeuffer_Einklassige_1971,pfeuffer_binarer_1978}, we derive a compact formula for the local density of binary lattices over arbitrary non-archimedean local fields, expressed in terms of the $a$-invariant and $R_{i}$-invariant from BONGs theory 
	 (Theorem \ref{thm:localdensity}). This indicates that these invariants can completely determine the local density for binary lattices over non-dyadic or dyadic local fields (cf. Remark \ref{re:local-density}). We also investigate the relationships between BONG's invariants and the invariants mentioned in \cite[Appendix A]{cho_localdensity-2020}, e.g. GK-invariant and EGK-invariant (cf. Proposition \ref{prop:fB-dB} and Remark \ref{re:fB-dB}). Using the invariants \eqref{a-Ri} and \eqref{R-alpha}, we reinterpret the local parts of  K\"{o}rner's proper class number formula  \cite{korner_classnum_1981} and Pfeuffer's mass formula \cite{pfeuffer_binarer_1978}, thereby deriving BONG's versions of the formulas for the mass $m(L)$ and its reciprocal $1/m(L)$ (Theorems \ref{thm:mass-1} and \ref{thm:mass-2}). As an application, we determine all the corresponding local quantities, in particular the local densities, for binary ADC lattices; see Tables \ref{tab:1}-\ref{tab:5} in Theorem \ref{thm:localinfor-ADC}. Based on these data, for every ramified dyadic local field we construct a series of explicit counterexamples, analogous to  \cite[Example A.1]{cho_localdensity-2020}, which have the same extended GK data but different local densities. We also explain this phenomenon using our local density formula \eqref{gamma}. See Example \ref{ex:GK-EGK} for details.
	 
	To study representations by ternary quadratic lattices, Beli \cite{beli_talk_2025} introduced the concept of the codeterminant set of spinor exceptions and gave equivalent conditions accordingly. Based on his work, we characterize the codeterminant set of spinor exceptions and spinor norm groups for ternary ADC lattices over non-archimedean local fields (cf. Theorems \ref{thm:coM-nondyadic} and \ref{thm:coM-dyadic}, and Lemmas \ref{lem:thetaOM-NE-non-dyadic}, \ref{lem:R-theta-M-nur} and \ref{lem:R-theta-H-1h}), thereby deriving the corresponding global information over algebraic number fields (Theorems \ref{thm:coM-global} and \ref{thm:thetaO+M}). By virtue of these results, we establish two sufficient conditions for an indefinite ternary lattice to be globally ADC (Theorem \ref{thm:coM-global-empty} and Corollary \ref{cor:L-G-ADC}). 
 
	Before proceeding with our discussion, we introduce the notation and definitions used throughout the paper, along with a brief overview of the arithmetic theory of quadratic forms. Any unexplained notation or definition can be found in \cite{He25,omeara_quadratic_1963}. 
	 
\noindent \textbf{General settings.} 

Let $ F $ be an algebraic number field or a non-archimedean local field with $  \chartic\, F\not=2 $, $F^{\times}$ the set of non-zero elements of $F$, $ \mathcal{O}_{F} $ the ring of integers of $ F $ and $ \mathcal{O}_{F}^{\times}$ the group of units. Let $ V $ be a non-degenerate quadratic space over $ F $ associated with the quadratic form $Q:V\to F$. We denote by $\det V$ the determinant, and $dV$ the discriminant of $V$. For an $\mathcal{O}_{F}$-lattice $L$ in $V$, we write $FL$ for the subspace of $V$ spanned by $L$ over $F$. We also denote by $\mathfrak{s}(L)$ the scale, $ \mathfrak{n}(L) $ the norm,  $\mathfrak{v}(L)$ the volume, and $dL$ the discriminant of $ L $, as usual. We call $ L $ \textit{integral} if $ \mathfrak{n}(L)\subseteq \mathcal{O}_{F}$ and \textit{primitive} if $\mathfrak{n}(L)=\mathcal{O}_{F}$. We also call $ L $ \textit{$  \mathcal{O}_{F} $-maximal} if   there is no integral $\mathcal{O}_{F}$-lattice $L^{\prime}$ on $FL$  such that $L\subsetneq L^{\prime}$. 

When $V=FL$, we denote by $O(V)$ the orthogonal group and $O^{+}(V)$ the special orthogonal group of $V$. Put  $O(L):=\{\sigma\in O(V)\mid \sigma(L)=L\} $ and $O^{+}(L):=O^{+}(V)\cap O(L)$. Without confusion, we abuse $\theta $ for local/global spinor norm map, and $N(E/F)$ for the norm group of the extension field $E/F$.

\noindent	\textbf{Global settings.}

When $F$ is an algebraic number field, let $L$ be an $\mathcal{O}_{F}$-lattice. We denote by $\gen(L)$ the genus, $\spn(L)$ the spinor genus, and $\cls(L)$ the class of $L$. Write $h(F)$ and $h(L)$ for the class number of $F$ and $L$, respectively. Also, write  $g(L)$ for the number of spinor genus in the genus $L$.

 We denote by $F^{+}$ the set of all totally positive elements in $F$; that is the set of all elements $c\in F$ such that $\sigma(c)>0$ for all the real embeddings $\sigma$ from $F$ to $\mathbb{R}$. Put $\mathcal{O}_{F}^{+}:=F^{+}\cap \mathcal{O}_{F}$.

\noindent	\textbf{Local settings.}
 
 When $ F $ is a non-archimedean local field, write $ \mathfrak{p} $ for the maximal ideal of $\mathcal{O}_F$, $ \pi \in\mathfrak{p}$ for a uniformizer and $N\mathfrak{p}$ for the number of elements in the residue class field of $F$. For $c\in F^{\times}$, we denote by $\ord(c) $ the order of $c$ and put $ e:=\ord(2)$. For a fractional or zero ideal $\mathfrak{a}$ of $F$, we put $\ord(\mathfrak{a})=\min\{\ord(c)\,|\,c\in \mathfrak{a}\}$. Let  $\Delta$ be a fixed unit such that $F(\sqrt{\Delta})/F$ is an unramified quadratic extension. If $F$ is non-dyadic, then $\Delta$ is an arbitrary non-square unit; if $F$ is dyadic, then $\Delta$ is a non-square unit of the form $\Delta=1-4\rho$, with $\rho\in\mathcal{O}_{F}^{\times}$.
 
 If $F$ is dyadic, we define the \textit{defect order} of $c$ by the map $ d $ from $ F^{\times}/F^{\times 2} $ to $ \mathbb{N}\cup \{\infty\} $: $ d(c):=\ord (c^{-1}\mathfrak{d}(c)) $, where  $ \mathfrak{d}(c)$ denotes the quadratic defect of $c$. Also, we denote by  $ \mathcal{U} $ a complete system of representatives of $\mathcal{O}_{F}^{\times}/\mathcal{O}_{F}^{\times 2}$ such that $ d(\delta)=\ord(\delta-1) $ for all $\delta\in \mathcal{U} $, and by $ \mathcal{V}:=\mathcal{U} \cup \pi\mathcal{U}  $ a set of representatives of $ F^{\times}/F^{\times 2} $. If $F$ is non-dyadic, we put $\mathcal{U}=\{1,\Delta\}$ and $\mathcal{V}=\{1,\Delta,\pi,\Delta\pi\}$. Clearly, $|\mathcal{U}|=2(N\mathfrak{p})^{e}$ and $|\mathcal{V}|=2|\mathcal{U}|=4(N\mathfrak{p})^{e}$.
 For $c\in \mathcal{V}$ and a subset $S\subseteq \mathcal{V}$, we also write $c=S$ (resp. $c\not=S$) for $c\in S$ (resp. $c\not\in S$) for convenience.
 
 Let $x_{1},\ldots,x_{n}\in V$ with $Q(x_{i})=a_{i}$. We write $ V\cong [a_{1},\ldots,a_{n}] $ if $ V=Fx_{1}\perp \ldots \perp Fx_{n} $ and $ L\cong \langle a_{1},\ldots,a_{n}\rangle $ if $ L=\mathcal{O}_{F}x_{1}\perp \ldots\perp \mathcal{O}_{F}x_{n} $. For $ c\in F^{\times}$ and $a,b \in  F$, we denote by $ c A(a,b) $ the binary $ \mathcal{O}_{F} $-lattice associated with the Gram matrix $ c\begin{pmatrix}
 	a  & 1\\
 	1  & b
 \end{pmatrix} $. As usual, denote by $\mathbb{H}$ the hyperbolic plane, and write $\mathbf{H}=2^{-1}A(0,0)$.
 \begin{defn}\label{defn:ADC}
		Let $ M $ be an integral $ \mathcal{O}_{F} $-lattice over a field $ F $. 
		
		 When $F$ is a non-archimedean local field,
		\begin{enumerate}[itemindent=-0.5em,label=\rm (\roman*)]
			\item[(i)] $ M$ is called  universal if it represents all elements $a$ in $ \mathcal{O}_{F}$.
			
			\item[(ii)]  $ M $ is called ADC if it represents every element $a$ in $\mathcal{O}_{F}$ for which $ FM $  represents $ a $. 
		\end{enumerate} 
		
		When $F$ is an algebraic numebr field,
	  \begin{enumerate}[itemindent=-0.5em,label=\rm (\roman*)]
		\item[(iii)] $ M$ is called globally universal, or simply universal, if it represents all elements $a$  in $ \mathcal{O}_{F}$ with $a\rep M_{\mathfrak{p}}$ at all real primes $\mathfrak{p}\in \infty_{F}$.
		
		\item[(iv)]   $ M $ is called globally ADC, or simply  ADC,  if it represents every element $a$ in $ \mathcal{O}_{F} $ for which $FM$ represents $ a $.  
	\end{enumerate} 
	\end{defn}
 In the rest of this section, we assume that $ F $ is an algebraic number field, $V$ is a non-degenerate quadratic space over $F$, and $M$ is an integral $ \mathcal{O}_{F} $-lattice on $V$. Let $ \Omega_{F}$ be the set of all primes of $ F $ and $ \infty_{F} $ be the set of all archimedean primes. For $ \mathfrak{p}\in \Omega_{F} $, let $ F_{\mathfrak{p}} $ be the completion of $ F $ at $ \mathfrak{p} $. Then we write $ M_{\mathfrak{p}}:=\mathcal{O}_{F_{\mathfrak{p}}}\otimes M $ if $ \mathfrak{p}\in \Omega_{F}\backslash \infty_{F} $, and  $ M_{\mathfrak{p}}:=F_{\mathfrak{p}}\otimes M=V_{\mathfrak{p}}$ otherwise. Clearly, $M_{\mathfrak{p}}$ is ADC (resp. universal) for all $ \mathfrak{p}\in \infty_{F} $, so we say that $ M $ is \textit{locally ADC} (resp. \textit{locally universal}) if $ M_{\mathfrak{p}} $ is ADC (resp. universal) for all $ \mathfrak{p}\in \Omega_{F}\backslash \infty_{F} $. 
 
In \cite[Theorem 25]{clark_ADC-I-2012}, Clark proved the relationship among universality, ADC-ness, and regularity:
\begin{align}\label{equiv:uni-adc-regular}
	\text{$M$ is universal}\;\Longrightarrow \; \text{$M$ is globally ADC} \;\Longleftrightarrow\;
	\begin{cases}
		 \text{$M$ is locally ADC},\\
		 \text{$M$ is regular}.
	\end{cases}
\end{align}
 (See also \cite[Theorem 1.3]{He25} for higher-dimensional representations.) For $\mathfrak{p}\in \Omega_{F}\backslash \infty_{F}$, we also have the connection between universality, ADC-ness, and maximality:
 \begin{align*} 
 	\text{$M_{\mathfrak{p}}$ is universal}\;\Longrightarrow \; \text{$M_{\mathfrak{p}}$ is ADC} \;\Longleftarrow\; \text{$M_{\mathfrak{p}}$ is $\mathcal{O}_{F_\mathfrak{p}}$-maximal}
 \end{align*}
 (cf. \cite[Theorems 8 and 19]{clark_ADC-I-2012} or Lemma \ref{lem:ADC-sufficient-equiv}), but not vice verse for each implication. For the first implication, if $\rank\,M_{\mathfrak{p}}=2$ or $3$, the converse holds under additional conditions (cf. Corollary \ref{cor:ADC-universal}). If $\rank\,M_{\mathfrak{p}}\ge 4$, the converse holds unconditionally. Indeed, from \cite[Theorem 1.4]{He25} with $n=1$, we have
\begin{align}\label{equiv:uni-adc}
	\text{$M_{\mathfrak{p}}$ is ADC}\;\Longleftrightarrow \; \text{$M_{\mathfrak{p}}$ is  universal};
\intertext{similarly, if $\rank\, M\ge 4$, then}	
		\text{$M$ is ADC}\;\Longleftrightarrow \; \text{$M $ is  universal}.
\end{align}
The equivalent conditions on universal $\mathcal{O}_{F_\mathfrak{p}}$-lattices have been given by Xu and Zhang for non-dyadic cases and partially for dyadic cases using classical invariants \cite[Propositions 2.3, 2.16 and 2.17]{xu_indefinite_2020}, and by Beli for general dyadic cases using BONGs' invariants \cite[Theorem 2.1]{beli_universal_2020}. So it remains to determine ADC lattices $M_{\mathfrak{p}}$ for $\rank\, M_{\mathfrak{p}}\le 3$.
\begin{thm}\label{thm:locallyn-ADC-binary}
	Suppose $\rank\, M \in \{1,2\} $. Let $ \mathfrak{p}\in \Omega_{F}\backslash \infty_{F} $.
	\begin{enumerate}[itemindent=-0.5em,label=\rm (\roman*)]
		\item If $\rank\, M=1$, then $ M_{\mathfrak{p}} $ is  ADC if and only if it is $ \mathcal{O}_{F_{\mathfrak{p}}} $-maximal.
		
		\item If $\rank\, M=2$ and $ \mathfrak{p}$ is non-dyadic, then $ M_{\mathfrak{p}} $ is  ADC if and only if it is $ \mathcal{O}_{F_{\mathfrak{p}}} $-maximal.
		
		\item If  $\rank\, M=2$ and $ \mathfrak{p}$ is dyadic, then $ M_{\mathfrak{p}} $ is  ADC if and only if it is either $ \mathcal{O}_{F_{\mathfrak{p}}} $-maximal, or isometric to the non $\mathcal{O}_{F_{\mathfrak{p}}}$-maximal lattice 
		\begin{align*}
		    2^{-1}\pi_{\mathfrak{p}}^{\nu_{\mathfrak{p}}} A(2\pi_{\mathfrak{p}}^{-1},2\rho_{\mathfrak{p}}\pi_{\mathfrak{p}}),  
		\end{align*}
		with $\nu_{\mathfrak{p}}\in \{0,1\}$.
	\end{enumerate} 
\end{thm}
\begin{re}
	All $\mathcal{O}_{F_{\mathfrak{p}}}$-maximal lattices of rank $1$ or $2$ are listed in Proposition \ref{prop:maximallattices-binary-ternary}.
\end{re}
We also use the notation $\mathbf{H}$ for the binary  $\mathcal{O}_{F}$-lattice $
 \mathcal{O}_{F}x+\mathcal{O}_{F}y $,
with $Q(x)=Q(y)=0$ and $B(x,y)=1/2$.
\begin{thm}\label{thm:globally ADC-rank2}
	Suppose $\rank \; M\in \{1,2\}$. Then $M$ is globally ADC if and only if  the following conditions hold:
	
		\begin{enumerate}[itemindent=-0.5em,label=\rm (\roman*)]
		\item If $\rank\,M=1$, then $M$ is $\mathcal{O}_{F}$-maximal.
		
		\item If $\rank\,M=2$ and $FM$ is isotropic, then $M$ is $\mathcal{O}_{F}$-maximal and $M\cong \mathbf{H} $.
		
		\item  If $\rank\,M=2$ and $FM$ is anisotropic, then $M$ is locally ADC and $h(M)=1$.
	\end{enumerate}  
\end{thm}
\begin{re}
	Note that $M$ is primitive if and only if $M_{\mathfrak{p}}$ is primitive for every $\mathfrak{p}\in \Omega_{F}\backslash \infty_{F}$. By Theorems \ref{thm:primitive-nondyadic}(i) and \ref{thm:primitive-dyadic}(i), one can easily obtain equivalent conditions for primitive or non-primitive binary ADC $\mathcal{O}_{F}$-lattices.
\end{re}
In what follows, we discuss some results on the finiteness of binary ADC lattices. From definition, the regularity  of an $\mathcal{O}_{F}$-lattice is preserved by scaling it by an element in $F^{\times}$. As in \cite{chan-icaza}, let $L$ be another $\mathcal{O}_{F}$-lattice on $V$. We say that $M$ is \textit{similar}  to $L$ if $M$ is isometric to $ L^{(c)}$ for some $c\in F^{+}$, where $L^{(c)}$ denotes the lattice $L$ scaled by $c$. 

In \cite{peter_class-num-one_1980}, Peters proved that there are finitely many positive definite binary $\mathcal{O}_{F}$-lattices with class number one under certain conditions, by which we have the following corollary.
\begin{cor}\label{cor:primitive-finite}
	
	 Up to similarity,
	 \begin{enumerate}[itemindent=-0.5em,label=\rm (\roman*)]
	\item There are only finitely many primitive positive definite binary ADC $\mathcal{O}_{F}$-lattices, when $F$ is fixed.
	
	\item Fixed a positive integer $n$, there are only finitely many primitive positive definite binary ADC $\mathcal{O}_{F}$-lattices, when $F$ varies through all totally real number fields with $[F:\mathbb{Q}]=n$.
	
	\item
	Under generalized Riemann hypothesis, there are only finitely many primitive positive definite binary ADC $\mathcal{O}_{F}$-lattices, when $F$ varies through all totally real number fields.
\end{enumerate}
\end{cor}

The following theorem is a generalization of \cite[Theorem 3.4]{clark_ADC-II-2014}, which provides a method to constructing non-primitive binary ADC $\mathcal{O}_{F}$-lattices from a given primitive one. Therefore, there may exist infinitely many non-primitive positive definite or indefinite binary ADC lattices (Corollary \ref{cor:imprimitive-infinite}).
\begin{thm}\label{thm:scale-ADC}
	Let $a\in \mathcal{O}_{F}^{+}$. If $\rank\, M=2$, then $M^{(a)}$ is ADC if and only if $M$ is ADC and for each $\mathfrak{p}\mid a\mathcal{O}_{F}$, $\ord_{\mathfrak{p}}(a)=1$, $FM_{\mathfrak{p}}\cong [1,-\Delta_{\mathfrak{p}}]$,
		\begin{align*}
		 \mathfrak{n}(M_{\mathfrak{p}})=\mathcal{O}_{F_{\mathfrak{p}}} \quad\text{and}\quad	\mathfrak{v}(M_{\mathfrak{p}})=
			\begin{cases}
				\mathcal{O}_{F_{\mathfrak{p}}}  &\text{if $\mathfrak{p}$ is non-dyadic}, \\
				\mathfrak{p}^{-2e_{\mathfrak{p}}}\;\text{or}\; \mathfrak{p}^{2-2e_{\mathfrak{p}}}  &\text{if $\mathfrak{p}$ is dyadic}.
			\end{cases}
		\end{align*} 
\end{thm} 

\begin{cor}\label{cor:imprimitive-infinite}
Suppose that $h(F)=1$ and there exists a primitive anisotropic binary ADC $\mathcal{O}_{F}$-lattice. Then there are infinitely many non-primitive binary ADC $\mathcal{O}_{F}$-lattices up to isometry.
\end{cor}

\newpage
\begin{thm}\label{thm:locallyn-ADC-ternary}
	Suppose $ \rank\,M=3 $. Let $ \mathfrak{p}\in \Omega_{F}\backslash \infty_{F} $.
	\begin{enumerate}[itemindent=-0.5em,label=\rm (\roman*)]
		\item If $ \mathfrak{p}$ is non-dyadic, then $ M_{\mathfrak{p}} $ is  ADC if and only if  it is either $\mathcal{O}_{F_{\mathfrak{p}}}$-maximal, or isometric to the non $\mathcal{O}_{F_{\mathfrak{p}}}$-maximal lattice 
			\begin{align*}
				\mathbf{H}\perp \langle \varepsilon_{\mathfrak{p}}\pi_{\mathfrak{p}}^{h_{\mathfrak{p}}}\rangle,
			\end{align*}
			with $\varepsilon_{\mathfrak{p}}\in \mathcal{U}_{\mathfrak{p}}$ and $h_{\mathfrak{p}}\ge 2$.

		\item If $ \mathfrak{p}$ is dyadic, then $ M_{\mathfrak{p}} $ is ADC if and only if  one of the following cases happens:

			\begin{enumerate}[itemindent=-0.5em,label=\rm (\roman*)]
				
				\item[(a)] it is $\mathcal{O}_{F_{\mathfrak{p}}}$-maximal;
				
				\item[(b)] it is isometric to the non $\mathcal{O}_{F_{\mathfrak{p}}}$-maximal lattice 
				\begin{align*}
					\mathbf{H}\perp \langle \varepsilon_{\mathfrak{p}}\pi_{\mathfrak{p}}^{h_{\mathfrak{p}}}\rangle,
				\end{align*}
				with $\varepsilon_{\mathfrak{p}}\in \mathcal{U}_{\mathfrak{p}}$ and $h_{\mathfrak{p}}\ge 2$; 
		 
				 \item[(c)] it is isometric to the non $\mathcal{O}_{F_{\mathfrak{p}}}$-maximal lattice 
				\begin{align*}
					\langle \Delta^{\nu_{\mathfrak{p}}-1},-\Delta^{\nu_{\mathfrak{p}}-1}\delta_{\mathfrak{p}}\pi_{\mathfrak{p}},\varepsilon_{\mathfrak{p}}\pi_{\mathfrak{p}}^{k_{\mathfrak{p}}} \rangle,
				\end{align*}
				with $\nu_{\mathfrak{p}},k_{\mathfrak{p}}\in \{1,2\}$ and $ \delta_{\mathfrak{p}},\varepsilon_{\mathfrak{p}}\in \mathcal{U}_{\mathfrak{p}}$;
				
				\item[(d)] it is isometric to 
				\begin{align*}
					(\delta_{\mathfrak{p}}^{\#})^{\nu_{\mathfrak{p}}-1} \pi_{\mathfrak{p}}^{ -l_{\mathfrak{p}}}A(\pi_{\mathfrak{p}}^{ l_{\mathfrak{p}}},-(\delta_{\mathfrak{p}}-1)\pi_{\mathfrak{p}}^{ -l_{\mathfrak{p}}}) \perp \langle \varepsilon_{\mathfrak{p}}\pi_{\mathfrak{p}}^{k_{\mathfrak{p}}}\rangle,
				\end{align*}
				with $\nu_{\mathfrak{p}}\in \{1,2\}$, $2l_{\mathfrak{p}}=d_{\mathfrak{p}}(\delta_{\mathfrak{p}})-1\le 2e_{\mathfrak{p}}-2$, $\delta_{\mathfrak{p}}\in  \mathcal{U}_{\mathfrak{p}}\backslash\{1,\Delta_{\mathfrak{p}}\}$,  $\varepsilon_{\mathfrak{p}}\in \mathcal{U}_{\mathfrak{p}}$  and  $k_{\mathfrak{p}}\in \{0,1\}$, where $\delta_{\mathfrak{p}}^{\#}=1+4\rho_{\mathfrak{p}}(\delta_{\mathfrak{p}}-1)^{-1}$. 
				
				Also, it is $\mathcal{O}_{F_{\mathfrak{p}}}$-maximal if and only if it is isometric to
				\begin{align*}  2^{-1}\pi_{\mathfrak{p}}A(2,2\rho_{\mathfrak{p}})\perp \langle \Delta_{\mathfrak{p}}\varepsilon_{\mathfrak{p}} \rangle,
				\end{align*}
				with $\varepsilon_{\mathfrak{p}}\in \mathcal{U}_{\mathfrak{p}}$.   
			\end{enumerate}  
 
	\end{enumerate} 
\end{thm}
 \begin{re}
 	\begin{enumerate}[itemindent=-0.5em,label=\rm (\roman*)]
 		\item All $\mathcal{O}_{F_{\mathfrak{p}}}$-maximal lattices of rank $3$ are listed in Proposition \ref{prop:maximallattices-binary-ternary}.
 		
 		\item  According to \cite[Theorem 19]{clark_ADC-I-2012}, $M_{\mathfrak{p}}$ is $\mathcal{O}_{F_{\mathfrak{p}}}$-maximal if and only if it is Euclidean in the sense of \cite[\S 4.2]{clark_ADC-I-2012}. So the term ``$\mathcal{O}_{F_{\mathfrak{p}}}$-maximal" in Theorems \ref{thm:locallyn-ADC-binary} and \ref{thm:locallyn-ADC-ternary} can be replaced with ``Euclidean" without any barrier.
 	\end{enumerate} 
 \end{re}
Different from the binary case, only primitive cases need to be considered for ternary ADC $\mathcal{O}_{F}$-lattices, as indicated by Theorems \ref{thm:primitive-nondyadic}(ii) and \ref{thm:primitive-dyadic}(ii). In the positive definite case, Clark and Jagy have determined all ternary ADC quadratic forms over $\mathbb{Z}$ (cf. \cite[Table 2]{clark_ADC-II-2014}), based on the list of candidates of ternary regular quadratic forms by Jagy, Kaplansky and  Schiemann \cite{JKS_regular913_1997}. But when $F\not=\mathbb{Q}$ is totally real, classifying ADC $\mathcal{O}_{F}$-lattices is challenging because little is known on classification of regular $\mathcal{O}_{F}$-lattices except for the finiteness result proved by Chan and Icaza \cite{chan-icaza}, which states that there are only finitely many ternary regular $\mathcal{O}_{F}$-lattices up to similarity. Therefore, a similar finiteness result for ternary ADC lattices  (Proposition \ref{prop:definite-ADC}(i)) follows from \eqref{equiv:uni-adc-regular} immediately. Note that any ternary $2$-ADC $\mathcal{O}_{F}$-lattice must be ADC. In fact, all such lattices are $\mathcal{O}_{F}$-maximal of class number one by \cite[Theorem 1.7]{He25}, and have been determined by Hanke \cite{Hanke_maximal} and Kirschmer \cite{kirschmer_one-class_2014}, which partially resolves the existence problem of ternary ADC $\mathcal{O}_{F}$-lattices over totally real number fields. In summary, we see that 
\newpage
\begin{prop}\label{prop:definite-ADC}
	Up to similarity, 
	\begin{enumerate}[itemindent=-0.5em,label=\rm (\roman*)]
		\item Given a totally real number field $F$, there are only finitely many positive definite ternary ADC $\mathcal{O}_{F}$-lattices. 
		
		\item There are at least $468$ positive definite ternary 2-ADC (and thus ADC) $\mathcal{O}_{F}$-lattices, when $F$ varies through all totally real number fields.
	\end{enumerate}
\end{prop}
One may conjecture a stronger version of Proposition \ref{prop:definite-ADC}(i): there are finitely many positive definite ternary ADC $\mathcal{O}_{F}$-lattices, when $F$ runs over all totally real number fields. However, according to \eqref{equiv:uni-adc-regular}, this would imply Kitaoka's conjecture, which remains open so far; see \cite{kala_survey_2023} for historical background, and \cite{kalajdyz_Kitaoka's conjecture-quad_2025,kala-yatsyna_Kitaoka's conjecture_2023} for recent progress. Therefore, establishing this stronger statement seems beyond reach at present. 

In the indefinite case, it is somewhat easier to determine  sufficient conditions for regularity, and thus for globally ADC property. We first review the behavior of spinor genera of positive definite or indefinite ternary ADC lattices. Suppose that $\rank\,M=3$ and $b\in \mathcal{O}_{F}$ is represented by $ \gen(M)$. Let $a=-b\det FM$ and $b=Q(x)$ for some $x\in M$. Following the work of Kneser \cite{kneser_darstellungsmasse-1961} and Schulze-Pillot \cite{schulze-pillot-darstellung-1980}, we say that $a$ is a \textit{spinor exception} of $\gen(M)$ if $b$ is represented by exactly half of the spinor genera in $\gen(M)$. It is known that  $b$ is a spinor exception of $\gen(M)$ if and only if $b\not\in F^{\times 2}$ and
\begin{align}\label{spinor-ex}
 \theta(O_{\mathbb{A}}^{+}(M))\subseteq N(F(\sqrt{a})/F)=\theta(X_{\mathbb{A}}(M,b)),
\end{align}
 where we denote $X_{\mathbb{A}}(M,b):=\{\sigma\in O_{\mathbb{A}}^{+}(V)\mid x\in \sigma(M)\}$, because the considered group $\theta(X_{\mathbb{A}}(M,b))$ is independent of the choice of the vector $x$; see \cite[Definition 1]{schulze-pillot-darstellung-1980}. Recall from \cite[Satz 2]{schulze-pillot-darstellung-1980} that the conditions \eqref{spinor-ex} can be reduced to the corresponding conditions over non-archimedean local fields, which were determined by Schulze-Pillot  \cite[\S 3]{schulze-pillot-darstellung-1980} for the non-dyadic and $2$-adic cases, and by Xu \cite[\S2]{xu_ternary_2000} for the general dyadic case. Similar equivalent conditions including higher-dimensional representation cases were recently given by Beli \cite{beli_talk_2025}, using the new terminology $\exc(M,N)$ and $\co(M)$. Following his terminology, for a unary $\mathcal{O}_{F}$-lattice $N\subseteq M$, we define the set
 \begin{align*}
 	X_{\mathbb{A}}(M/N):=\{\sigma \in O_{\mathbb{A}}^{+}(V)\mid N\subseteq \sigma(M)\}.
 \end{align*}
For every $a\in F^{\times}/F^{\times 2}$ with $a\not=1$, we further define the sets
\begin{align*}
	\exc(M,a):=&\{N\mid   N\rep  \gen(M),\;-\det FM\det FN=a,\; \theta(X_{\mathbb{A}}(M/N))=N(F(\sqrt{a})/F)\}, \\
	\co(M):=&\{a\in F^{\times}/F^{\times 2}\mid a\not=1,\; \exc(M,a)\not=\emptyset\}.
\end{align*}
Note that for every $b=Q(x)$ with $x\in M$, one can take $N_{b}:=\mathcal{O}_{F}x$. Then $N_{b}$ is a unary lattice contained in $M$, and thus is represented by $   \gen(M) $. Also, $X(M,b)=X(M/N_{b})$. Hence if $\co(M)=\emptyset$, then there is no unary lattice $N$ such that $\theta(X_{\mathbb{A}}(M/N))=N(F(\sqrt{a})/F)$. So
\begin{align*}
	\theta(X_{\mathbb{A}}(M,b))\not=N(F(\sqrt{a})/F)
\end{align*}
for all such $b$. This means that there is no spinor exception in $\gen(M)$ provided that $\co(M)=\emptyset$.
 
For each non-archimedean prime $\mathfrak{p}$, we characterize the codeterminant set $\co(M_{\mathfrak{p}})$ of ternary ADC lattices $M$ (Theorems \ref{thm:coM-nondyadic} and \ref{thm:coM-dyadic}), thereby determining the equivalent conditions for $\co(M)$. 
\begin{thm}\label{thm:coM-global}
	Let $a\in F^{\times}/F^{\times 2}$ and $a\not=1$. Suppose that $M$ is ternary locally ADC. Then $a\in \co(M)$ if and only if the following conditions hold:
 
	\begin{enumerate}[itemindent=-0.5em,label=\rm (\roman*)]
		\item When $\mathfrak{p}$ is archimedean and real,  
		\begin{enumerate}[itemindent=-0.5em,label=\rm (\roman*)]
			\item[\rm (a)] if $FM_{\mathfrak{p}}$ is isotropic, then $a_{\mathfrak{p}}\in   F_{\mathfrak{p}}^{\times 2}$;
			
			\item[\rm (b)] if $FM_{\mathfrak{p}}$ is anisotropic, then  $a_{\mathfrak{p}}\in  -F_{\mathfrak{p}}^{\times 2}$.			
		\end{enumerate}
		
		\item When $\mathfrak{p}$ is non-dyadic, if $FM_{\mathfrak{p}}$ is isotropic and $\mathfrak{v}(M_{\mathfrak{p}})=\mathfrak{p}^{k_{\mathfrak{p}}}\subseteq \mathcal{O}_{F_{\mathfrak{p}}}$, then  $a_{\mathfrak{p}}\in F_{\mathfrak{p}}^{\times 2}\cup \Delta_{\mathfrak{p}}^{k_{\mathfrak{p}}+1}F_{\mathfrak{p}}^{\times 2} $.

		\item When $\mathfrak{p}$ is dyadic,

		\begin{enumerate}[itemindent=-0.5em,label=\rm (\roman*)]
		 	\item[\rm (a)]  if $FM_{\mathfrak{p}}$ is isotropic and $\mathfrak{v}(M_{\mathfrak{p}})=\mathfrak{p}^{k_{\mathfrak{p}}}\subseteq \mathfrak{p}^{2}$, then  $a_{\mathfrak{p}}\in F_{\mathfrak{p}}^{\times 2}\cup \Delta_{\mathfrak{p}}^{k_{\mathfrak{p}}+1}F_{\mathfrak{p}}^{\times 2} $;
		 	
		 	\item[\rm (b)]  if $FM_{\mathfrak{p}}$ is isotropic and $  2\mathcal{O}_{F_{\mathfrak{p}}}\subseteq 2\mathfrak{s}(M_{\mathfrak{p}})\subseteq\mathcal{O}_{F_{\mathfrak{p}}} $, then $a_{\mathfrak{p}}\in F_{\mathfrak{p}}^{\times 2} $.		 	
			 \end{enumerate}
	\end{enumerate}
\end{thm}
Based on Theorem \ref{thm:coM-global}, we derive a sufficient condition for $\co(M)=\emptyset$, and hence for indefinite ternary ADC lattices.
\begin{thm}\label{thm:coM-global-empty}
	Suppose that $M$ is ternary locally ADC. If there exists some $\mathfrak{p}\in \Omega_{F}\backslash \infty_{F}$ such that $FM_{\mathfrak{p}}$ is anisotropic, then $\co(M)=\emptyset$; thus there is no spinor exception in $\gen(M)$. 
	
	If moreover, $M$ is indefinite, then $M$ is globally ADC.
\end{thm}
For each non-archimedean prime $\mathfrak{p}$, we also characterize the integral spinor norm groups $\theta(O^{+}(M_{\mathfrak{p}}))$ of ternary ADC lattices $M$ (Lemmas \ref{lem:thetaOM-NE-non-dyadic}, \ref{lem:R-theta-M-nur}, and \ref{lem:R-theta-H-1h}), thereby obtaining the following theorem. 
\begin{thm}\label{thm:thetaO+M}
	Suppose that $M$ is ternary locally ADC. Then $\theta(O^{+}(M_{\mathfrak{p}}))\supseteq \mathcal{O}_{F_{\mathfrak{p}}}^{\times}F_{\mathfrak{p}}^{\times 2}$ for all $\mathfrak{p}\in \Omega_{F}\backslash \infty_{F}$.
\end{thm}
\begin{re}
	Since an ADC $\mathcal{O}_{F_{\mathfrak{p}}}$-lattice $M_{\mathfrak{p}}$ of rank $\ge 4$ is universal by \eqref{equiv:uni-adc}, it follows that $\theta(O^{+}(M_{\mathfrak{p}}))\supseteq \mathcal{O}_{F_{\mathfrak{p}}}^{\times}F_{\mathfrak{p}}^{\times 2}$. Therefore, Theorem \ref{thm:thetaO+M} holds for $\rank\,M\ge 4$, which generalizes \cite[91:8]{omeara_quadratic_1963} to ADC lattices.
\end{re}
By Theorem \ref{thm:thetaO+M}, we provide a sufficient condition for regularity.
\begin{thm}\label{thm:gM-hM=1}
 	Suppose that $M$ is ternary locally ADC. If $h(F)$ is odd, then $g(M)=1$; thus there is only one spinor genus in $\gen(M)$.
 	 
 	 If moreover, $M$ is indefinite, then $h(M)=1$; thus $M$ is regular.
 \end{thm}
Combining Theorem \ref{thm:gM-hM=1} with the equivalent condition \eqref{equiv:uni-adc-regular}, we derive another sufficient condition for indefinite ternary ADC lattices.
\begin{cor}\label{cor:L-G-ADC}
	 Let $M$ be an indefinite ternary $\mathcal{O}_{F}$-lattice. Suppose that $h(F)$ is odd. If $M$ is locally ADC, then $M$ is globally ADC.
\end{cor}
From the equivalence \eqref{equiv:uni-adc}, \cite[Proposition 2.3(1)]{xu_indefinite_2020}, and \cite[Theorem 2.1 II]{beli_universal_2020}, one can see that there are infinitely many quaternary or higher-dimensional ADC lattices over non-archimedean local fields up to isometry. For the remaining cases, we provide explicit counting formulas.
 \begin{thm}\label{thm:count-sol-23}
 	 Denote by $B(m)$ the number of ADC  $\mathcal{O}_{F}$-lattices $M$ with rank $m\in \{1,2,3\}$ over a non-archimedean local field $F$.  
 	 	\begin{enumerate}[itemindent=-0.5em,label=\rm (\roman*)]
 	 	\item If $m=1$, then $B(1)=4(N\mathfrak{p})^{e}$.
 	 		
 	 	\item If $m=2$, then
 	   \begin{align*}
 	 	B(2)=\begin{cases}
 	 		8(N\mathfrak{p})^{e}-1+0  &\text{if $e=0$}, \\
 	 	  8(N\mathfrak{p})^{e}-1+2  &\text{if $e \ge 1$}. 
 	 	\end{cases}
 	 \end{align*}
 	 
 	 \item If $m=3$, then  $B(3)=\infty$. Also, the number can be refined as follows, 
 	\begin{align*}
 	 \begin{cases}	 
 				4(N\mathfrak{p})^{e}+\infty    &\text{if  $e=0$ and $FM$ is isotropic}, \\
 					4(N\mathfrak{p})^{e}+0     &\text{if  $e=0$ and $FM$ is anisotropic}, \\
 			4(N\mathfrak{p})^{e}  + \infty &\text{if  $e \ge 1$ and $FM$ is isotropic},\\
 				4(N\mathfrak{p})^{e}   +(4e+2)(N\mathfrak{p})^{e}   &\text{if   $e \ge 1$ and $FM$ is anisotropic}.
 		\end{cases}
 	\end{align*}
  \end{enumerate}
  Here the second addend counts the number of those lattices that are  ADC, but not $\mathcal{O}_{F}$-maximal.
 \end{thm}

The rest of the paper is organized as follows. In Section  \ref{sec:BONGs}, we briefly review Beli's BONGs theory related to representations and spinor norms of quadratic lattices in dyadic local fields. In Section \ref{sec:maximal-lattices}, we discuss some results on quadratic spaces and lattices over non-archimedean local fields, particularly for the binary and ternary cases. We then characterize binary and ternary ADC lattices, and determine the integral spinor norm groups and codeterminant sets for ternary ADC lattices over non-dyadic local fields in Section \ref{sec:non-dyadic}, and over dyadic local fields in Section \ref{sec:dyadic}. In Section \ref{sec:proof-main-results}, we will show our main results mentioned above (Theorems \ref{thm:locallyn-ADC-binary}, \ref{thm:globally ADC-rank2}, \ref{thm:scale-ADC}, \ref{thm:locallyn-ADC-ternary}, \ref{thm:coM-global}, \ref{thm:coM-global-empty}, \ref{thm:thetaO+M}, \ref{thm:gM-hM=1}, and \ref{thm:count-sol-23}, and Corollaries \ref{cor:primitive-finite} and \ref{cor:imprimitive-infinite}). Finally, we will discuss the local density and mass for binary lattices over algebraic number fields.
	
All lattices considered in the following discussion are assumed to be integral.
 
\section{A brief overview of BONGs theory}\label{sec:BONGs}
Throughout this section, we assume that $F$ is dyadic, and  thus $ e\ge 1 $. Write $ [h,k]^{E} $ for the set of all even integers $ i $ such that $ h\le i\le k$. For $ c_{i}\in F^{\times} $, we also write $ c_{i,j}=c_{i}\cdots c_{j} $ for short and put $ c_{i,i-1}=1 $. 
	
The vectors $x_{1},\ldots, x_{m}$ of $FM$ is called a \textit{BONG} for an $\mathcal{O}_{F} $-lattice $M$ if $\mathfrak{n}(M)=Q(x_{1})\mathcal{O}_{F}$ and $x_{2},\ldots,x_{m}$ is a BONG for $\pr_{x_{1}^{\perp}}M$, and it is said to be \textit{good} if $\ord(Q(x_{i}))\le \ord(Q(x_{i+2}))$ for $1\le i\le m-2$. We denote by $ M\cong\prec a_{1},\ldots,a_{m}\succ $ if $ x_{1},\ldots, x_{m} $ is a BONG for $M$ with $ Q(x_{i})=a_{i} $.
\begin{lem}[{\cite[Lemma 2.2]{HeHu2}}]\label{lem:goodBONGequivcon}
	Let $ x_{1},\ldots, x_{m}  $ be pairwise orthogonal vectors in a quadratic space with $ Q(x_{i})=a_{i} $ and $ R_{i}=\ord (a_{i}) $.  Then $ x_{1},\ldots,x_{m} $ forms a good BONG for some lattice is equivalent to the conditions
	\begin{align}\label{eq:GoodBONGs}
		R_{i}\le R_{i+2} \quad \text{for all $ 1\le i\le m-2$}
	\end{align}
	and  
	\begin{align}\label{eq:BONGs}
		R_{i+1}-R_{i}+d(-a_{i}a_{i+1})\ge 0 \quad\text{and}\quad  R_{i+1}-R_{i}+2e\ge 0\quad \text{for all $ 1\le i\le m-1 $}. 
	\end{align}
	\end{lem}
	\begin{re}\label{re:mathscrA}
		By \cite[Lemma 3.5]{beli_integral_2003}, $a_{i+1}/a_{i}\in \mathscr{A}$ if and only if $a_{i+1}/a_{i}$ satisfies \eqref{eq:BONGs}. So if $M\cong \prec a_{1},\ldots,a_{m}\succ$ relative to a good BONG,  $a_{i+1}/a_{i}\in \mathscr{A}$ holds for all $1\le i\le m-1$.
	\end{re}
	
Let $M\cong \prec a_{1},\ldots, a_{m} \succ $ be an $ \mathcal{O}_{F} $-lattice  relative to some good BONG. Define the \textit{$ R_{i} $-invariants} by $  R_{i}(M):=\ord (a_{i})$ for $ 1\le i\le m $ and the \textit{$ \alpha_{i} $-invariants} by $  \alpha_{i}(M)\coloneqq\min\limits_{1\le j\le m-1}\{T_{j}^{(i)}\} $ for $1\le i\le m-1$, where
	\begin{align*}
		T_{j}^{(i)}:=\begin{cases}
			(R_{i+1}-R_{i})/2+e       &\text{if $j=0 $}, \\
			R_{i+1}-R_{j}+d(-a_{j}a_{j+1}) &\text{if $ 1\le j\le i $}, \\
			R_{j+1}-R_{i}+d(-a_{j}a_{j+1})  &\text{if $ i\le j\le m-1 $}.
		\end{cases}
	\end{align*}
For $0\le i-1\le j\le m$ and $c\in F^{\times}$, we define
 \begin{align*}
 	 d[ca_{i,j}]:=\min\{d(ca_{i,j}),\alpha_{i-1},\alpha_{j}\}.
 \end{align*}
 Here we ignore the terms $\alpha_{0}$ and $\alpha_{m}$ in the minimum, if present, since they are undefined. In particular, if $m=2$, we have $d[-a_{1,2}]=d(-a_{1}a_{2})$.
 
We have the following properties for these invariants (cf. \cite[\S 2]{HeHu2}).
\begin{prop}\label{prop:Rproperty}
	Suppose $ 1\le i\le m-1 $. 
		\begin{enumerate}[itemindent=-0.5em,label=\rm (\roman*)]
			\item $ R_{i+1}-R_{i}>2e$ (resp. $ =2e $, $ <2e $) if and only if $ \alpha_{i}>2e $ (resp. $ =2e $, $ <2e $).
			
			\item If $ R_{i+1}-R_{i}\ge 2e$ or $ R_{i+1}-R_{i}\in \{-2e,2-2e,2e-2\} $, then $ \alpha_{i}=(R_{i+1}-R_{i})/2+e $.
			
			\item If $ R_{i+1}-R_{i}\le 2e $, then $ \alpha_{i}\ge R_{i+1}-R_{i} $. Also, the equality holds if and only if $ R_{i+1}-R_{i}=2e $ or $ R_{i+1}-R_{i} $ is odd.
			
			\item If $ R_{i+1}-R_{i} $ is odd, then $ R_{i+1}-R_{i}>0 $.
		\end{enumerate} 
\end{prop}
\begin{prop}\label{prop:alphaproperty}
	Suppose $ 1\le i\le m-1 $.
		\begin{enumerate}[itemindent=-0.5em,label=\rm (\roman*)]
			\item Either $ 0\le \alpha_{i}\le 2e $ and $ \alpha_{i}\in \mathbb{Z} $, or $ 2e<\alpha_{i}<\infty $ and $ 2\alpha_{i}\in \mathbb{Z} $; thus $ \alpha_{i}\ge 0 $.
			
			\item $\alpha_{i}=\min\{(R_{i+1}-R_{i})/2+e,R_{i+1}-R_{i}+d[-a_{i,i+1}]\}$.
			
			\item $ \alpha_{i}=0 $ if and only if $ R_{i+1}-R_{i}=-2e $.
			
			\item $ \alpha_{i}=1 $ if and only if either $ R_{i+1}-R_{i}\in \{2-2e,1\} $, or $ R_{i+1}-R_{i}\in [4-2e,0]^{E} $ and $ d[-a_{i,i+1}]=R_{i}-R_{i+1}+1 $.
			
			\item If $ \alpha_{i}=0 $, i.e., $ R_{i+1}-R_{i}=-2e $, then $ d[-a_{i,i+1}]\ge 2e $.
	 
			\item If $ \alpha_{i}=1 $, then $ d[-a_{i,i+1}]\ge R_{i}-R_{i+1}+1 $. Also, the equality holds if $ R_{i+1}-R_{i}\not=2-2e $.
		\end{enumerate}  
	\end{prop}

By using property A or B, in the sense of \cite[Definitions 7 and 10]{beli_integral_2003}, one can compute the integral spinor norm group $\theta(O^{+}(M))$ effectively. Here we collect some results related to these properties (cf. \cite[Corollary 4.2, Lemma 4.11, Remark 4.12, \S 7]{beli_integral_2003}).
\begin{prop}\label{prop:property-A}
	The following statements hold for $M$:
	\begin{enumerate}[itemindent=-0.5em,label=\rm (\roman*)]
	 	 \item 	 $M$ has property A if and only if $R_{i}<R_{i+2}$ for all $1\le i\le m-2$.
	 		
	 	\item If $M$ has property B, then $R_{i}+2<R_{i+2}$ for all $1\le i\le m-2$.
	 		
		\item  If $M$ does not have property A, then $\theta(O^{+}(M))=\mathcal{O}_{F}^{\times }F^{\times 2}$ or $F^{\times}$.
		
		\item If $M$ has property A but does not have property B, then $\theta(O^{+}(M))=\mathcal{O}_{F}^{\times}F^{\times 2}$ or $F^{\times}$.
	\end{enumerate}
\end{prop}
	
Let $N\cong \prec b_{1}, \ldots,b_{n}\succ$ be another $\mathcal{O}_{F}$-lattice relative to some good BONG. Write $S_{i}=R_{i}(N)$ and $\beta_{i}=\alpha_{i}(N)$. Assume that $m\in \{2,3\}$ and $n=1$. Put
\begin{align*}
	d[a_{1}b_{1}]&:=\min\{d(a_{1}b_{1}),\alpha_{1}\},\\
			A_{1}=A_{1}(M,N)&:=\min\{(R_{2}-S_{1})/2+e,R_{2}-S_{1}+d[-a_{1,2}]\}.
\end{align*}
Then the terms $\alpha_{3}$ and $\beta_{1}$ are ignored in $d[-a_{1,3}b_{1}]$, so $d[-a_{1,3}b_{1}]=d(-a_{1,3}b_{1})$. Thus, from \cite[Remarks 1]{beli_universal_2020}, \cite[Theorem 1.2]{beli_universal_2020} can be reformulated as follows.  
\begin{thm}\label{thm:beligeneral-1}
 		 $ N\rep M $ if and only if  $ FN\rep FM $ and the following conditions hold:
 	\begin{enumerate}[itemindent=-0.5em,label=\rm (\roman*)]
 		\item $R_{1}\le S_{1}$. 
 		
 		\item  $ d[a_{1}b_{1}]\ge A_{1} $.
 		
 		\item  If $m=3$,
 			$R_{3}>S_{1}$ and $d[-a_{1,2}]+d(-a_{1,3}b_{1})>2e+S_{1}-R_{3}$, 
 		then $ [b_{1}]\rep [a_{1},a_{2}] $.
 	\end{enumerate}  
 \end{thm}
 
\section{Representations of quadratic lattices over non-archimedean local fields}\label{sec:maximal-lattices}
In this section, we briefly review the setting of quadratic spaces and lattices over non-archimedean local fields, as discussed in \cite[\S 4]{He25}, and formulate the results needed for binary or ternary cases. Any unexplained notation or result can be found therein. Unless otherwise stated, we always assume that $F$ is a non-archimedean local field and $M$ is an $\mathcal{O}_{F}$-lattice of rank $m$ associated with the quadratic form $Q$.
\begin{defn}\label{defn:space-maximallattice}
	Let $n\ge 1$. For $c\in \mathcal{V}$, define the $ n $-dimensional quadratic space over $F$:
	\begin{align*}
			W_{1}^{n}(c):=
			\begin{cases}
				\mathbb{H}^{\frac{n-2}{2}}\perp [1,-c] &\text{if $ n $ is even}, \\
				\mathbb{H}^{\frac{n-1}{2}}\perp [c]    &\text{if $ n $ is odd},  			\end{cases}
		\end{align*}
		and  if $n\not=1$ and $(n,c)\not=(2,1)$, define the $ n $-dimensional quadratic space $ W_{2}^{n}(c) $ with $ \det  W_{2}^{n}(c)=\det  W_{1}^{n}(c) $ and $ W_{2}^{n}(c)\not\cong W_{1}^{n}(c) $. We further define the $ \mathcal{O}_{F} $-maximal lattice on $ W_{\nu}^{n}(c) $ by $ N_{\nu}^{n}(c) $ provided that $W_{\nu}^{n}(c)$ is defined.
	\end{defn}		
\begin{re}
 When a quadratic space $W_{\nu}^{n}(c)$ or an $\mathcal{O}_{F}$-maximal lattice $N_{\nu}^{n}(c)$ is discussed, we always assume that $(n,\nu)\not=(1,2)$ and $(n,\nu,c)\not=(2,2,1)$.
\end{re}
	
When $F$ is dyadic, let $ c\in \mathcal{V}\backslash \{1,\Delta\}$, we write $c\pi^{-\ord(c)}=a^{2}(1+b\pi^{d(c)})$ with $a,b\in \mathcal{O}_{F}^{\times}$ when $\ord(c)$ is even, and then put
	\begin{align*}
		c^{\#}:=
		\begin{cases}
			\Delta    &\text{if $\ord(c)$ is odd}, \\
			1+4\rho b^{-1}\pi^{-d(c)}  &\text{if $\ord(c)$ is even}.
		\end{cases}
	\end{align*}
From \cite[Proposition 3.2]{HeHu2}, we see that 
\begin{align}\label{csharp}
	c^{\#}\in \mathcal{O}_{F}^{\times}, \quad d(c^{\#})=2e-d(c)\quad\text{and}\quad(c^{\#},c)_{\mathfrak{p}}=-1.
\end{align}
  The quadratic spaces $W_{\nu}^{n}(c)$ and lattices $N_{\nu}^{n}(c)$ have been determined in \cite{HeHu2,hhx_indefinite_2021} (cf. \cite[Proposition 4.2(i) and Lemmas 4.7 and 4.9 for details]{He25}). We will enumerate $W_{\nu}^{n}(c)$ and $N_{\nu}^{n}(c)$ for $n\in \{1,2,3\}$, which, in fact, exhausts all dimension $n$ quadratic spaces and rank $n$ $\mathcal{O}_{F}$-maximal lattices from \cite[Remark 4.3]{He25}.
	\begin{prop}\label{prop:space}
		Let $ n\in \{1,2,3\} $, $\nu\in \{1,2\}$ and $c\in \mathcal{V}$. The quadratic space $ W_{\nu}^{n}(c) $  is given by the following table, 
			\begin{center}
 				\renewcommand\arraystretch{1.5}
				\hskip -0.8cm \begin{tabular}{c|c|c|c}
					\toprule[1.2pt]
					& $ c $ & $ W_{1}^{n}(c) $  & $ W_{2}^{n}(c) $  \\
					\hline
					\multirow{2}*{\text{$n=1$}}	 & $ \delta,\delta\in\mathcal{U}  $ & $    [\delta]  $  &    \\
					\cline{2-4}
					& $ \delta\pi,\delta\in\mathcal{U}  $ & $  [\delta\pi] $  &    \\
					\hline
					\multirow{4}*{\text{$n=2$}}	 & $ 1 $ & $ \mathbb{H}  $ &  \\
					\cline{2-4}
					& $ \Delta $  & $     [1,-\Delta]  $ & $   [ \pi,-\Delta\pi ]  $  \\
					\cline{2-4}
					& $ \delta,\delta\in \mathcal{U}\backslash \{1,\Delta\} $  & $  [1,-\delta  ] $ & $    [\delta^{\#},-\delta^{\#}\delta  ] $  \\
					\cline{2-4}
					& $ \delta\pi, \delta\in \mathcal{U} $  & $    [1,-\delta\pi] $ & $   [\Delta,-\Delta\delta\pi]  $  \\
					\hline
					\multirow{2}*{\text{$n=3$}}	 & $ \delta,\delta\in\mathcal{U}  $ & $ \mathbb{H} \perp   [\delta] $  &   $    [\pi,-\Delta\pi,\Delta\delta]
					$  \\
					\cline{2-4}
					& $ \delta\pi,\delta\in\mathcal{U}  $ & $ \mathbb{H} \perp  [\delta\pi]  $  &   $     [1,-\Delta,\Delta\delta\pi]  $  \\
					\bottomrule[1.2pt]
				\end{tabular}
			\end{center}	
			where the third case is ignored when $e=0$ and $n=2$.  
	\end{prop}
	 \begin{prop}\label{prop:maximallattices-binary-ternary}
	 	Let $n\in \{1,2,3\}$, $\nu\in\{1,2\}$ and $c\in \mathcal{V}$.
	 	
	 	\begin{enumerate}[itemindent=-0.5em,label=\rm (\roman*)]
	 		
	 		\item If $F$ is non-dyadic, then the $ \mathcal{O}_{F} $-maximal lattice $ N_{\nu}^{n}(c) $ is given by the following table, 
	 		\begin{center}
	 			\renewcommand\arraystretch{1.5}
	 			\begin{tabular}{c|c|c|c}
	 				\toprule[1.2pt]
	 				& $ c $ & $ N_{1}^{n}(c) $  & $ N_{2}^{n}(c) $  \\
	 				\hline
	 				\multirow{2}*{\text{$n=1$}} & $ \delta,\delta\in \mathcal{U} =\{1,\Delta\}$ & $  \langle \delta\rangle  $  &   \\
	 				\cline{2-4}
	 				& $ \delta\pi,\delta\in \mathcal{U} =\{1,\Delta\} $ & $   \langle \delta\pi\rangle $  &   \\
	 				\hline
	 				\multirow{3}*{\text{$n=2$}}	  & $ 1 $ & $ \mathbf{H} $ &   \\
	 				\cline{2-4}
	 				& $ \Delta $  & $   \langle 1,-\Delta\rangle  $ & $  \langle \pi,-\Delta\pi\rangle $  \\
	 				\cline{2-4}
	 				& $ \delta\pi, \delta\in \mathcal{U} =\{1,\Delta\} $  & $  \langle 1,-\delta\pi\rangle $ & $   \langle \Delta,-\Delta\delta\pi\rangle$  \\
	 				\hline
	 				\multirow{2}*{\text{$n=3$}} & $ \delta,\delta\in \mathcal{U} =\{1,\Delta\}$ & $ \mathbf{H} \perp\langle \delta\rangle  $  &   $  \langle \pi,-\Delta\pi,\Delta\delta\rangle $  \\
	 				\cline{2-4}
	 				& $ \delta\pi,\delta\in \mathcal{U} =\{1,\Delta\} $ & $ \mathbf{H} \perp \langle \delta\pi\rangle $  &   $   \langle 1,-\Delta ,\Delta\delta\pi\rangle$  \\
	 				\bottomrule[1.2pt]
	 			\end{tabular}
	 		\end{center}   
	 		\item If $F$ is dyadic, then the $ \mathcal{O}_{F} $-maximal lattice $ N_{\nu}^{n}(c) $ is given by the following table, 
	 		\begin{center}
	 			\renewcommand\arraystretch{1.5}
	 			\hskip -0.8cm \begin{tabular}{c|c|c|c}
	 				\toprule[1.2pt]
	 				& $ c $ & $ N_{1}^{n}(c) $  & $ N_{2}^{n}(c) $  \\
	 				\hline
	 				\multirow{2}*{\text{$n=1$}}	 & $ \delta,\delta\in\mathcal{U}  $ & $  \prec \delta\succ $  &    \\
	 				\cline{2-4}
	 				& $ \delta\pi,\delta\in\mathcal{U}  $ & $  \prec\delta\pi\succ $  &    \\
	 				\hline
	 				\multirow{4}*{\text{$n=2$}}	 & $ 1 $ & $ \mathbf{H}  $ &  \\
	 				\cline{2-4}
	 				& $ \Delta $  & $    \prec 1,-\Delta\pi^{-2e}\succ $ & $  \prec \pi,-\Delta\pi^{1-2e}\succ $  \\
	 				\cline{2-4}
	 				& $ \delta,\delta\in \mathcal{U}\backslash \{1,\Delta\} $  & $  \prec 1,-\delta\pi^{1-d(\delta)}\succ$ & $   \prec\delta^{\#},-\delta^{\#}\delta\pi^{1-d(\delta)}\succ$  \\
	 				\cline{2-4}
	 				& $ \delta\pi, \delta\in \mathcal{U} $  & $   \prec  1,-\delta\pi\succ$ & $   \prec\Delta,-\Delta\delta\pi\succ $  \\
	 				\hline
	 				\multirow{2}*{\text{$n=3$}}	 & $ \delta,\delta\in\mathcal{U}  $ & $ \mathbf{H} \perp \prec \delta\succ $  &   $   \prec \delta\kappa^{\#}, -\delta\kappa^{\#}\kappa\pi^{2-2e}, \delta \kappa\succ 
	 				$  \\
	 				\cline{2-4}
	 				& $ \delta\pi,\delta\in\mathcal{U}  $ & $ \mathbf{H} \perp \prec\delta\pi\succ $  &   $   \prec 1,-\Delta\pi^{-2e},\Delta\delta\pi\succ $  \\
	 				\bottomrule[1.2pt]
	 			\end{tabular}
	 		\end{center}
	 		where $ \kappa $ is a fixed unit with $ d(\kappa)=2e-1 $.  
	 	\end{enumerate}
	 \end{prop}	
 \begin{prop}\label{prop:ternary-space-rep}
 	Let $a,b,c\in F^{\times}$. The following statements are equivalent:
 		\begin{enumerate}[itemindent=-0.5em,label=\rm (\roman*)]
 	\item   $[a,b]$ represents $[c]$.
 	
 	\item  $[a,b,-c]$ is isotropic.
 	
 	\item $(-ab,ac)_{\mathfrak{p}}=1$.
 	
 	\item $(-ba,bc)_{\mathfrak{p}}=1$.
 \end{enumerate}
 \end{prop}
 \begin{proof}
 	This is clear from  \cite[58:6]{omeara_quadratic_1963}.
 \end{proof}
 \begin{lem}\label{lem:lattice-rep-binary}
 	Let $\nu\in \{1,2\}$, $\varepsilon\in \mathcal{U}$, $\delta\in \mathcal{U}\backslash \{1,\Delta\}$ and $c\in \mathcal{V}$. Suppose that $M$ is ADC. 
 	\begin{enumerate}[itemindent=-0.5em,label=\rm (\roman*)]
 		\item  If $FM\cong W_{1}^{2}(1)$, then  $M$ is  universal.		
 	 		
 		\item  If $FM\cong W_{1}^{2}(\Delta)$, then $M$ represents $\langle\varepsilon\rangle$, but does not reprsent $\langle\varepsilon\pi\rangle$
 		
 		\item   If $ FM\cong W_{2}^{2}(\Delta)$, then $M$ represents $\langle\varepsilon\pi\rangle$, but does not represent $\langle\varepsilon\rangle$.
 		
 		\item  If $FM\cong W_{\nu}^{2}(\varepsilon\pi)$, then $M$ represents $\langle \Delta^{\nu-1}\rangle$ and $\langle-\Delta^{\nu-1}\varepsilon\pi\rangle$, but does not represent  $\langle\Delta^{\nu}\rangle$ and $\langle-\Delta^{\nu}\varepsilon\pi\rangle$.
 		
 		\item  If $F$ is dyadic and $FM\cong W_{\nu}^{2}(\delta)$, then $M$ represents $\prec (\delta^{\#})^{\nu-1}\succ$ and $\prec -(\delta^{\#})^{\nu-1}\delta\succ$,  but does not represent $\prec (\delta^{\#})^{\nu}\succ$ and $\prec -(\delta^{\#})^{\nu}\delta\succ$.
 		
 		\item If $F$ is dyadic and $FM\cong W_{\nu}^{2}(\delta)$, then $M$ represents precisely one of $\prec-\delta\pi\succ$ and $\prec-\delta^{\#}\delta\pi\succ$.
 		
 		\item  If $FM\cong W_{1}^{3}(c)$, then $M$ is universal.
 		
 		\item If $FM\cong W_{2}^{3}(c)$, then $M$ represents $\langle c^{\prime}\rangle$ for all $c^{\prime}\in \mathcal{V}$ with $c^{\prime}\not=c$.
 	\end{enumerate}
 \end{lem}
 \begin{proof}
 	Since $M$ is ADC, $M$ represents $N_{1}^{1}(c)$ if and only if $FM$ represents $W_{1}^{1}(c)$. Hence, it suffices to verify that $FM\cong W_{\nu^{\prime}}^{m}(c^{\prime})$ ($m=2,3$) represents $W_{1}^{1}(c)$ or not. This is a direct calculation of Hilbert symbols by applying   \cite[Lemma 4.4(ii) and (iii)]{He25}.
 \end{proof}
 \begin{re}
 The statements (i)-(iv) and (vii)-(viii) also hold when $F$ is dyadic. In this case, the notation $\langle c\rangle$ denotes the unary lattice $\prec c\succ$ relative to some good BONG.
 \end{re}
 
 By Propositions \ref{prop:space} and \ref{prop:ternary-space-rep}, and Lemma \ref{lem:lattice-rep-binary}, we have the following proposition.
 \begin{prop}\label{prop:nu-isotropic}
 	Let $c\in \mathcal{V}$. Suppose that $M$ is ADC.
 	\begin{enumerate}[itemindent=-0.5em,label=\rm (\roman*)]
 		\item  If $FM\cong   W_{1}^{3}(c)$ or $W_{1}^{2}(1)$, then $FM$ is isotropic and $M$ is universal.
 		
 		\item If $FM\cong  W_{2}^{3}(c)$ or $W_{1}^{2}(c)$ with $c\not=1$, then $FM$ is anisotropic, but $M$ is not universal.
 		
 	\end{enumerate}
 \end{prop}
    Recall from \cite[Proposition 2.3 and Corollary 2.9]{xu_indefinite_2020} that a binary $\mathcal{O}_{F}$-lattice $M$ is universal if and only if $M\cong\mathbf{H}$. Hence, by Propositions \ref{prop:space} and \ref{prop:nu-isotropic}, one can show the following corollary.
 	\begin{cor}\label{cor:ADC-universal}
 	 Let $m\in \{2,3\}$. Suppose that $M$ is ADC. Then $M$ is universal if and only if $FM$ is isotropic. In this case, if $m=2$, then $M\cong \mathbf{H}$. 
 \end{cor}
 \begin{lem}\label{lem:ADC-sufficient-equiv}
    	Let $m\ge 1$.
 	 \begin{enumerate}[itemindent=-0.5em,label=\rm (\roman*)]
 	 \item If $M$ is $\mathcal{O}_{F}$-maximal, then $M$ is ADC.
 	 
      \item $ M $ is ADC if and only if $M$ represents every unary $\mathcal{O}_{F}$-maximal lattice $N$  for which $FM$ represents $FN$.
 \end{enumerate}
 \end{lem}
 \begin{proof}
 	These two assertions follow from \cite[Lemmas 4.14 and 2.1]{He25} directly.
 \end{proof}
  \begin{cor}\label{cor:counting-maximal}
 	There are $4(N\mathfrak{p})^{e}$ ternary isotropic (resp. anisotropic) ADC lattices that are $\mathcal{O}_{F}$-maximal, up to isometry.
 \end{cor}
 \begin{proof}
 	By Lemma \ref{lem:ADC-sufficient-equiv}(i), an $\mathcal{O}_{F}$-maximal lattice must be ADC. Hence, from \cite[Remark 4.3]{He25} and Proposition \ref{prop:nu-isotropic}, there are $|\mathcal{V}|=4(N\mathfrak{p})^{e}$ $\mathcal{O}_{F}$-maximal lattices $M$ with $FM\cong W_{1}^{3}(c)$ or $W_{2}^{3}(c)$, according as $FM$ is isotropic or anisotropic.
 \end{proof}
 
 Following \cite{beli_talk_2025}, for $c\in F^{\times}$ we define the sets
 \begin{align*}
 	\exc(M,a):=&\{N\mid N\rep M,\;-\det FM\det FN=a,\; \theta(X(M/N))=N(F(\sqrt{a})/F)\}, \\
 	\co(M):=&\{a\in F^{\times}/F^{\times 2}\mid  \exc(M,a)\not=\emptyset\},
 \end{align*}
 and then reformulate some lemmas on $\co(M)$ from \cite{beli_talk_2025} with $m=3$.
 \begin{lem}\label{lem:co-nondyadic}
 	 Suppose that $F$ is non-dyadic and $M$ has a Jordan splitting $M=M_{1}\perp \cdots \perp M_{t}$. Assume that $\theta(O^{+}(M))\subseteq N(F(\sqrt{a})/F)$.
 	\begin{enumerate}[itemindent=-0.5em,label=\rm (\roman*)]
 		\item If $a=\Delta$, then $a\in \co(M)$.
 		
 		\item If $a\in \pi\mathcal{U}$, then $a\in \co(M)$ if and only if $\ord(\mathfrak{s}(M_{i}))\not\equiv \ord(\mathfrak{s}(M_{j})) \pmod{2}$ for all $1\le i,j\le t$.
 	\end{enumerate}
 \end{lem}
 \begin{re}\label{re:co-1-Delta}
 	We also have
 	\begin{enumerate}[itemindent=-0.5em,label=\rm (\roman*)]
 		\item If $F$ is non-archimedean, then $1\in \co(M)$ if and only if $FM$ is isotropic.		
 	 
 		\item If $F$ is archimedean, then
 		\begin{align*}
 			\co(M)=
 			\begin{cases}
 				\{1\}   &\text{if $F$ is complex}, \\
 				\{1\}  &\text{if $F$ is real and $FM$ is isotropic}, \\
 				\{-1\}    &\text{if $F$ is real and $FM$ is anisotropic}.
 			\end{cases}
 		\end{align*}
 	\end{enumerate}
 \end{re}
\begin{lem}\label{lem:co-dyadic}
 	Let $a\in \mathcal{V}\backslash \{1\}$. Suppose that $F$ is dyadic and $M\cong \prec a_{1},a_{2},a_{3}\succ$ relative to a good BONG. Put $R_{i}=R_{i}(M)$.  Assume that $\theta(O^{+}(M))\subseteq N(F(\sqrt{a})/F)$.
 	\begin{enumerate}[itemindent=-0.5em,label=\rm (\roman*)]
 		\item  If $R_{1}<R_{3}$, then $a\in \co(M)$ if and only if  $d(-a_{1}a_{2}a)+d(-a_{2}a_{3}a)+d(a)>4e$
 		
 		\item  If $R_{1}=R_{3}$, then $\co(M)\subseteq \{1,\Delta\}$. Also, $\Delta\in \co(M)$ if and only if $\theta(O^{+}(M))=\mathcal{O}_{F}^{\times}F^{\times 2}$ and $R_{2}-R_{1}=2e$.
 	\end{enumerate}
 \end{lem}
 
 In the remainder of this section, we discuss the representation densities of ADC lattices. Following \cite[\S 2]{clark-density-2024}, for $c\in \mathcal{V}$, we denote by $Q_{c}(M):=Q(M)\bigcap cF^{\times 2}$, and define the \textit{local representation measure} $\delta(M)$ of $M$ by
 \begin{align}\label{deltaM}
 	\delta(M):=\mu\left(\bigcup_{c\in \mathcal{V}}Q_{c}(M)\right),
 \end{align}
 where $\mu$ denotes the Haar measure on $\mathcal{O}_{F}$ with unit mass. From \cite[Theorem 2.2]{clark-density-2024}, we have the following formula for $\delta(M)$:     
\begin{align}\label{lambda-formula}
\delta(M)=\sum_{c\in \mathcal{V}} \dfrac{1}{|\mathcal{U}|(N\mathfrak{p})^{v_{c}-1}(N\mathfrak{p}+1)},
\end{align}
 where $v_{c}:=\min\{\ord(a)\mid a\in  Q_{c}(M)\}$, or $\infty$ if $ Q_{c}(M)=\emptyset$. 
 
 When $M$ is ADC, the values of $\delta(M)$ have been determined under certain conditions in \cite{clark-density-2024}. Here we provide a unified and explicit formula for $\delta(M)$ over arbitrary non-archimedean local fields, which depends only on the structure of the quadratic space $W_{\nu}^{m}(c)$.

\begin{thm}\label{thm:density}
Let $FM\cong W_{\nu}^{m}(c)$ with $\nu\in \{1,2\}$ and $c\in \mathcal{V}$.  Suppose that $M$ is ADC.
\begin{enumerate}[itemindent=-0.5em,label=\rm (\roman*)]
	\item If $m=1$, then 
		\begin{align*}
		\delta(M)=\begin{cases}
			 \dfrac{N\mathfrak{p}}{2(N\mathfrak{p})^{e}(N\mathfrak{p}+1)} &\text{if $\ord(c)=0$},  \\
			 \dfrac{1}{2(N\mathfrak{p})^{e}(N\mathfrak{p}+1)} &\text{if $\ord(c)=1$}.  
		\end{cases}
	\end{align*}
	\item If $m=2$, then
	\begin{align*}
	   \delta(M)=\begin{cases}
	     	1 &\text{if $c=1$},  \\
	     	\dfrac{1}{(N\mathfrak{p})^{\nu-2}(N\mathfrak{p}+1)}  &\text{if  $c=\Delta$}, \\
	     	1/2 &\text{if $c\in \mathcal{V}\backslash \{1,\Delta\}$}.
	     \end{cases}
	\end{align*}
\item 	If $m=3$, then
		\begin{align*}
		\delta(M)&=\begin{cases}
			 1  & \text{if $\nu=1$},  \\
			\dfrac{(2(N\mathfrak{p})^{e}-1)N\mathfrak{p}+2(N\mathfrak{p})^{e}}{2(N\mathfrak{p})^{e}(N\mathfrak{p}+1)} &\text{if $\nu=2$ and $\ord(c)=0$},  \\
			\dfrac{2(N\mathfrak{p})^{e}N\mathfrak{p}+2(N\mathfrak{p})^{e}-1}{2(N\mathfrak{p})^{e}(N\mathfrak{p}+1)} &\text{if $\nu=2$ and $\ord(c)=1$}.
		\end{cases}
	\end{align*}
	\item If $m\ge 4$, then $\delta(M)=1$.
\end{enumerate}
\end{thm}
\begin{proof}
    Since $M$ is ADC, $M$ represents $N_{1}^{1}(a)$ if and only if $FM$ represents $W_{1}^{1}(a)=[a]$. This can be verified by \cite[Lemma 4.4(i)-(iii)]{He25}. Also, by Proposition \ref{prop:space}, $v_{a}\in \{0,1\}$ for $a\in \mathcal{V}$.
	
	If $FM\cong W_{1}^{2}(1)$ or $ W_{1}^{3}(c)$ or $m\ge 4$, by Proposition \ref{prop:nu-isotropic} and \eqref{equiv:uni-adc}, $M$ is universal and so $\delta(M)=1$.
	
	If $m=1$, then by \cite[Lemma 4.4(i)]{He25}, $FM$ represents $a$ if and only if $FM\cong [a]$. So $c$ is the only one element in $\mathcal{V}$ that is represented by $FM$. Hence, by \eqref{lambda-formula},
	\begin{align*}
		\delta(M)=\dfrac{1}{|\mathcal{U}|(N\mathfrak{p})^{-1}(N\mathfrak{p}+1)}\quad\text{or}\quad \dfrac{1}{|\mathcal{U}|(N\mathfrak{p})^{0}(N\mathfrak{p}+1)},
	\end{align*}
	according as $\ord(c)=0$ or $1$.
	
	Let $FM\cong W_{\nu}^{2}(\Delta) $. For $a\in \mathcal{V}$, since $(\Delta,a)_{\mathfrak{p}}=(-1)^{\ord(a)}$, by \cite[Lemma 4.4(ii)]{He25}, exactly half of the elements $a$ in $\mathcal{V}$ with $v_{a}=0$ (resp. $v_{a}=1$) are represented by $W_{1}^{2}(\Delta)$ (resp. $W_{2}^{2}(\Delta)$). Thus, $v_{a}=\nu-1$. Hence, by \eqref{lambda-formula}, we have
	\begin{align*}
		\delta(M)=\dfrac{|\mathcal{V}|}{2}\dfrac{1}{|\mathcal{U}|(N\mathfrak{p})^{v_{a}-1}(N\mathfrak{p}+1)}=	\dfrac{|\mathcal{V}|}{2}\dfrac{1}{|\mathcal{U}|(N\mathfrak{p})^{\nu-2}(N\mathfrak{p}+1)}.
	\end{align*}
	
	Let $FM_{\nu}\cong W_{\nu}^{2}(c)$ with $c\in \mathcal{V}\backslash\{1,\Delta\}$. 		For $a\in \mathcal{V}$, note from \eqref{csharp} that $\ord(c^{\#}a)=\ord(a)$ and $(c,c^{\#}a)_{\mathfrak{p}}=-(c,a)_{\mathfrak{p}}$. Hence, by \cite[Lemma 4.4(ii)]{He25}, exactly half of the elements $a$ in $\mathcal{V}$ with $v_{a}=0$ (resp. $v_{a}=1$) are represented by $FM_{\nu}$, where $\nu=1,2$. So $\delta(M_{1})=\delta(M_{2})$ and $\delta(M_{1})+\delta(M_{2})=1$. Hence $\delta(M_{1})=\delta(M_{2})=1/2$.
	
	Let $FM\cong W_{2}^{3}(c)$ with $c\in \mathcal{V}$. Then $\ord(c)\in \{0,1\}$. By \cite[Lemma 4.4(iii)]{He25}, $FM$ represents all $a$ in $\mathcal{V}$ with $a\not=c$. So in $\mathcal{V}$, there are $|\mathcal{U}|-1+\ord(c)$ elements $a$ with $v_{a}=0$ and $|\mathcal{U}|-\ord(c)$ elements $a$ with $v_{a}=1$, that are represented by $FM$. Hence
	\begin{align*}
		\delta(M)=\dfrac{|\mathcal{U}|-1+\ord(c)}{|\mathcal{U}|(N\mathfrak{p})^{-1}(N\mathfrak{p}+1)}+\dfrac{|\mathcal{U}|-\ord(c)}{|\mathcal{U}|(N\mathfrak{p}+1)}.
	\end{align*}
  In all cases, the formula follows from the fact that $|\mathcal{V}|=2|\mathcal{U}|=4(N\mathfrak{p})^{e}$. 
\end{proof}

\section{ADC lattices over non-dyadic local fields}\label{sec:non-dyadic}
In this section, we assume that $F$ is non-dyadic and $M$ is an $\mathcal{O}_{F}$-lattice of rank $m $. We also denote by $ J_{k}(M) $ the Jordan component of $ M $, with possible zero rank and scale $  \mathfrak{p}^{k} $, and write $ J_{i,j}(M):=J_{i}(M)\perp J_{i+1}(M)\perp\cdots\perp J_{j}(M) $ for integers $ i\le j $. 
\begin{thm}\label{thm:nondyadic-ADC-binary}
		If $m=2$, then $ M $ is ADC if and only if $M$ is $ \mathcal{O}_{F} $-maximal.
\end{thm}
For $h\ge 0$ and $c\in \mathcal{V}$, we define the ternary $\mathcal{O}_{F}$-lattice $H_{1,h}^{3}(c):=\mathbf{H}\perp \langle c\pi^{2h} \rangle$. From Proposition \ref{prop:maximallattices-binary-ternary}(i),  $H_{1,h}^{3}(c)$ is $\mathcal{O}_{F}$-maximal if and only if $h=0$. Clearly, $FH_{1,h}^{3}(c)\cong W_{1}^{3}(c)$ is isotropic.
\begin{thm}\label{thm:nondyadic-ADC-ternary}
		If $m=3$, then $ M $ is ADC if and only if the following conditions hold:
		\begin{enumerate}[itemindent=-0.5em,label=\rm (\roman*)]
			\item  If $FM $ is isotropic, then one of the following cases happens:			
			\begin{enumerate}[itemindent=-0.5em,label=\rm (\roman*)]			
				\item[(a)] $M$ is $\mathcal{O}_{F}$-maximal;
 			
				\item[(b)]	$
				M\cong  H_{1,h}^{3}(c)$ with $h\ge 1$ and $c\in \mathcal{V}$.
			\end{enumerate}
  	\item If $FM$ is anisotropic, then $M$ is $\mathcal{O}_{F}$-maximal. 
		\end{enumerate} 
	\end{thm}
\begin{re}\label{re:nondyadicACD-3}
	\begin{enumerate}[itemindent=-0.5em,label=\rm (\roman*)]
	\item $M$ is not $\mathcal{O}_{F}$-maximal in the cases (i)(b) of Theorem \ref{thm:nondyadic-ADC-ternary}.
		
	\item From Theorem \ref{thm:nondyadic-ADC-ternary}, if $M$ is ternary ADC, then $M$ is either of the form $N_{\nu}^{3}(c)$ with $\nu\in \{1,2\}$ and $c\in \mathcal{V}$, or of the form $H_{1,h}^{3}(c)$ with $h\ge 1$ and $c\in \mathcal{V}$.
	\end{enumerate}
\end{re}
By Theorem \ref{thm:nondyadic-ADC-ternary} and Corollary \ref{cor:counting-maximal}, we have the following corollary immediately.
\begin{cor}\label{cor:count-nondyadic}
	Up to isometry, 
	\begin{enumerate}[itemindent=-0.5em,label=\rm (\roman*)]	
		\item  there are infinitely many ternary isotropic ADC lattices, of which $4(N\mathfrak{p})^{e}$ are $\mathcal{O}_{F}$-maximal. 
		
		\item there are $4(N\mathfrak{p})^{e}$ ternary anisotropic ADC lattices, and all of them are $\mathcal{O}_{F}$-maximal. 
	\end{enumerate}
\end{cor}
 \begin{thm}\label{thm:primitive-nondyadic}
	Let $m\in \{2,3\}$.
	\begin{enumerate}[itemindent=-0.5em,label=\rm (\roman*)]
		\item  If $m=2$, then $M$ is non-primitive ADC if and only if $FM\cong [\pi,-\Delta\pi]$, $\mathfrak{n}(M)=\mathfrak{p}$ and $\mathfrak{v}(M)=\mathfrak{p}^{2}$.
		
		\item  If $m=3$, then $M$ is primitive ADC if and only if $M$ is ADC.
	\end{enumerate}
\end{thm}
\begin{thm}\label{thm:coM-nondyadic}
	Let $\nu\in \{1,2\}$, $h\ge 1$ and $c\in \mathcal{V}$. Suppose that $M$ is ternary ADC.
	\begin{enumerate}[itemindent=-0.5em,label=\rm (\roman*)]
		\item  When $M=N_{\nu}^{3}(c)$, we have
		\begin{align*}
			\co(M)=
			\begin{cases}
				\{1,\Delta\}   &\text{if $(\nu,c)=(1,\mathcal{U})$}, \\
				\{1\}   &\text{if $(\nu,c)=(1,\pi\mathcal{U})$}, \\
				\emptyset   &\text{if $\nu=2$}.
			\end{cases}
		\end{align*}
				
		\item  When $M=H_{1,h}^{3}(c)$, we have
		\begin{align*}
			\co(M)=\begin{cases}
				\{1,\Delta\} &\text{if $c\in \mathcal{U}$}, \\
				\{1\}  &\text{if $c\in \pi\mathcal{U}$}.
			\end{cases}
		\end{align*} 
	\end{enumerate}
\end{thm}

We first treat Theorems \ref{thm:nondyadic-ADC-binary} and \ref{thm:nondyadic-ADC-ternary}, and then show Theorems \ref{thm:primitive-nondyadic} and \ref{thm:coM-nondyadic}.
For sufficiency of Theorems \ref{thm:nondyadic-ADC-binary} and \ref{thm:nondyadic-ADC-ternary}, if the case of Theorem \ref{thm:nondyadic-ADC-ternary}(i)(b) happens, then by \cite[Proposition 2.3(1)]{xu_indefinite_2020}, $M$ is universal and thus is ADC. For the remaining cases, the ADC-ness of $M$ follows by Lemma \ref{lem:ADC-sufficient-equiv}(i) immediately. To show necessity, we need some lemmas.
\begin{lem}\label{lem:maximal-rep-nondyadic}
	Let $N=\langle \varepsilon\pi^{k}\rangle$ with $\varepsilon\in \mathcal{U}$ and $k\in \{0,1\}$. Then $M$ represents $N$ if and only if $FJ_{0}(M)$ represents $FJ_{0}(N)$ and $FJ_{0,1}(M)$ represents $FN$.
\end{lem}
\begin{proof}
	This is  \cite[Lemma 4.8]{He25} with $n=1$.
\end{proof}
\begin{lem}\label{lem:non-dyadic-adc-rank}
	Let $m\in \{2,3\}$. 
	\begin{enumerate}[itemindent=-0.5em,label=\rm (\roman*)]
		\item If $M$ represents $\langle \varepsilon \rangle$ for some $\varepsilon\in \mathcal{U} $, then $\rank\,J_{0}(M)\ge 1$.
			
		\item  If $M$ represents $\langle \varepsilon\rangle$ for all $\varepsilon\in \mathcal{U}$, then $\rank\,J_{0}(M)\ge 2$.			
			
	    \item If $M$ represents $\langle \varepsilon \pi\rangle$ for all $\varepsilon\in \mathcal{U}$, then $\rank\,J_{0,1}(M)\ge 2$.
			
	    \item If $m=2$, $\rank\, J_{0}(M)=1$ and $M$ represents $\langle \varepsilon \pi\rangle$ for some $\varepsilon\in \mathcal{U}$, then $  \rank\,J_{1}(M)=1$.
			
	    \item  If $m=3$, $\rank\, J_{0}(M)=1$ and $M$ represents $\langle \varepsilon \pi\rangle$ for all $\varepsilon\in \mathcal{U}$, then $\rank\,J_{1}(M)=2$.
			
		\item  If $  J_{0}(M)\cong \langle 1,-\Delta \rangle $ and $M$ represents $\langle \varepsilon \pi\rangle$ for some $\varepsilon\in \mathcal{U}$, then $\rank\,J_{1}(M)\ge 1$. 
		\end{enumerate}	
\end{lem}  
\begin{proof}
		 (i)  By Lemma \ref{lem:maximal-rep-nondyadic}, $FJ_{0}(M)$ represents $[\varepsilon]$ for some $\varepsilon\in \mathcal{U}$, which implies that
		\begin{align*}
			\rank\,J_{0}(M)=\dim FJ_{0}(M)\ge  1.
		\end{align*}
		
		(ii) By Lemma \ref{lem:maximal-rep-nondyadic}, $FJ_{0}(M)$ represents $[\varepsilon]$  for $\varepsilon=1,\Delta$. Since $[1]\not\cong [\Delta]$, this implies that
		\begin{align*}
			\rank\, J_{0}(M)=\dim FJ_{0}(M)\ge 2.
		\end{align*}	
		
		(iii) By Lemma \ref{lem:maximal-rep-nondyadic}, $FJ_{0,1}(M)$ represents $[\varepsilon\pi]$ for $\varepsilon=1,\Delta$. Since $ [\pi]\not\cong [\Delta\pi]$ , this implies that
		\begin{align*}
			\rank\, J_{0,1}(M) =\dim FJ_{0,1}(M) \ge 2.
		\end{align*}		
		
		 (iv) By Lemma \ref{lem:maximal-rep-nondyadic}, $FJ_{0,1}(M)$ represents $[\varepsilon\pi]$ for some $\varepsilon \in \mathcal{U} $. If $\rank\, J_{1}(M)=0$, then $FJ_{0,1}(M)=FJ_{0}(M)=[\delta]$ for some $\delta\in \mathcal{U}$, which fails to represent $[\varepsilon\pi]$. Hence $
		\rank\, J_{1}(M)\ge 1$ and thus $\rank\, J_{1}(M)=1$.	
		
		 (v) Similar to (iv), we have $\rank\, J_{1}(M)\ge 1$. If $\rank\, J_{1}(M)=1$, then $FJ_{0,1}(M)\cong  W_{\nu}^{2}(\eta\pi)$ for some $\eta\in \mathcal{U}$. Since $(\eta\pi,-\Delta^{\nu}\eta\pi)_{\mathfrak{p}}=(\delta\pi,\Delta)_{\mathfrak{p}}^{\nu}=(-1)^{\nu}\not=(-1)^{\nu+1}$, by \cite[Lemma 4.4(ii)]{He25}, $FJ_{0,1}(M)$ fails to represent $[-\Delta^{\nu}\eta\pi]$. But by the hypothesis and Lemma \ref{lem:maximal-rep-nondyadic}, $FJ_{0,1}(M)$ represents $[\varepsilon\pi]$ for all $\varepsilon\in \mathcal{U}$, a contradiction. Hence $\rank\,J_{1}(M)=2$. 
		
		 (vi) By Lemma \ref{lem:maximal-rep-nondyadic}, $FJ_{0,1}(M)$ represents $[\varepsilon\pi]$ for some $\varepsilon \in \mathcal{U} $. If $\rank\, J_{1}(M)=0$, then $FJ_{0}(M)=FJ_{0,1}(M)\cong [1,-\Delta]$ represents $[\varepsilon\pi]$. By Proposition \ref{prop:ternary-space-rep}, $(\Delta,\varepsilon\pi)_{\mathfrak{p}}=1$, a contradiction. Hence $\rank\, J_{1}(M)\ge 1$.
\end{proof}
\begin{proof}[Proof of necessity of Theorem \ref{thm:nondyadic-ADC-binary}]
  	If $FM\cong W_{1}^{2}(1)$, by Proposition \ref{prop:nu-isotropic}(i) and Corollary \ref{cor:ADC-universal}, $M\cong\mathbf{H}=N_{1}^{2}(1)$.
  	
  	If $FM\cong W_{1}^{2}(\Delta)$, by Lemma \ref{lem:lattice-rep-binary}(ii), $M$ represents $\langle 1\rangle $ and $\langle \Delta \rangle$.  Hence, by Lemma \ref{lem:non-dyadic-adc-rank}(ii), $\rank\, J_{0}(M)\ge 2$. Thus $M=J_{0}(M)$ is unimodular. Since $dM=dN_{1}^{2}(\Delta)=-\Delta$, by \cite[92:2]{omeara_quadratic_1963}, $M\cong \langle 1,-\Delta\rangle=N_{1}^{2}(\Delta)$. 
  	
   If $FM\cong W_{2}^{2}(\Delta)$, by Lemma \ref{lem:lattice-rep-binary}(iii), $M$ represents $\langle \pi\rangle$ and $\langle \Delta\pi\rangle$. Hence, by Lemma \ref{lem:non-dyadic-adc-rank}(iii), $\rank\, J_{0,1}(M)\ge 2$. But $M$ does not represent $\langle 1\rangle$ and $\langle \Delta\rangle$, so $\rank\, J_{0}(M)=0$. Hence $M=J_{1}(M)$ is $\mathfrak{p}$-modular. Since $dM=dN_{2}^{2}(\Delta)=-\Delta\pi^{2}$, again by \cite[92:2]{omeara_quadratic_1963}, $M\cong  \langle \pi,-\Delta\pi\rangle= N_{2}^{2}(\Delta)$. 
   
   If $FM\cong W_{\nu}^{2}(\varepsilon\pi)$, by Lemma \ref{lem:lattice-rep-binary}(iv), $M$ represents $\langle \Delta^{\nu-1} \rangle$ and $\langle-\Delta^{\nu-1}\varepsilon\pi\rangle$. Hence, by Lemma \ref{lem:non-dyadic-adc-rank}(i), $\rank\, J_{0}(M)\ge 1$. Since $\det FM=-\varepsilon\pi$, $\rank\, J_{0}(M)=1$. So, by Lemma \ref{lem:non-dyadic-adc-rank}(iv),  $\rank\,J_{1}(M)=1$.  We may let $M\cong \langle \eta,-\eta \varepsilon\pi \rangle$ with $\eta \in \{1,\Delta\}$ and $(\eta,\varepsilon\pi)_{\mathfrak{p}}=(-1)^{\nu-1}$. By \cite[92:2]{omeara_quadratic_1963}, $M\cong \langle 1,-\varepsilon\pi\rangle=N_{1}^{2}(\varepsilon\pi)$ or  $\cong \langle \Delta,-\Delta\varepsilon\pi\rangle=N_{2}^{2}(\varepsilon\pi)$, according as $\nu=1$ or $2$
\end{proof}
\begin{proof}[Proof of necessity of Theorem \ref{thm:nondyadic-ADC-ternary}]
	If $FM$ is isotropic, by Proposition \ref{prop:nu-isotropic}, $FM\cong W_{1}^{3}(c)$ with $c\in \mathcal{V}$ and $M$ is universal. From \cite[Proposition 2.3(i)]{xu_indefinite_2020}, $M\cong \mathbf{H}\perp \langle c\pi^{2h}\rangle$ with $h\ge 0$. Also, from Proposition \ref{prop:maximallattices-binary-ternary}(i), $M\cong N_{1}^{3}(c)$ is $\mathcal{O}_{F}$-maximal when $h=0$.
		
	If $FM$ is anisotropic, by Proposition \ref{prop:nu-isotropic}, we may assume $FM\cong W_{2}^{3}(\varepsilon)$ or $W_{2}^{3}(\varepsilon\pi)$, with $\varepsilon\in \mathcal{U}$.
		
	(i) Let $FM\cong W_{2}^{3}(\varepsilon)$. By Lemma \ref{lem:lattice-rep-binary}(viii), $M$ represents  $\langle \eta\pi\rangle$ for all $\eta\in \mathcal{U}$. Hence, by Lemma \ref{lem:non-dyadic-adc-rank}(i), $\rank\,J_{0}(M)\ge 1$. If $\rank\,J_{0}(M)\ge 2$, then  $J_{0}(M)$ is split by $ \mathbf{H}$ or $\langle 1,-\Delta \rangle $. Hence $J_{0}(M)$ (and thus $M$) represents   $\langle \eta\rangle$ for all $\eta\in \mathcal{U}$. This contradicts Lemma \ref{lem:lattice-rep-binary}(viii). Hence $\rank\,J_{0}(M)=1$. So, by Lemma \ref{lem:non-dyadic-adc-rank}(v), $\rank\, J_{1}(M)=2$.  
		
	Let $J_{1}(M)\cong \langle \eta \pi, -\eta^{\prime}\Delta\pi\rangle$  with $\eta ,\eta^{\prime}\in \{1,\Delta\}$. As $FM$ is anisotropic, $\det FJ_{1}(M)\not=-1$. Hence $\det FJ_{1}(M)=-\Delta\eta\eta^{\prime}=-\Delta$ and thus $\eta=\eta^{\prime}$. So, by \cite[92:2]{omeara_quadratic_1963}, $J_{1}(M)\cong \langle\eta \pi, -\eta\Delta \pi\rangle\cong \langle \pi, -\Delta \pi\rangle$. Let $M\cong \langle \delta,\pi,-\Delta\pi \rangle$ with $\delta\in \{1,\Delta\}$. Since $\det FM=\det W_{2}^{3}(\varepsilon)=-\varepsilon$, $J_{0}(M)\cong \langle \Delta\varepsilon\rangle$. Hence, by \cite[92:2]{omeara_quadratic_1963}, $M\cong  N_{2}^{3}(\varepsilon)$.
		
	(ii) Let $FM\cong W_{2}^{3}(\varepsilon\pi)$. Since $1,\Delta\not=\varepsilon\pi$, by Lemma \ref{lem:lattice-rep-binary}(viii), $M$ represents $\langle 1\rangle$ and $\langle \Delta\rangle$. Hence, by Lemma  \ref{lem:non-dyadic-adc-rank}(ii), $\rank\, J_{0}(M)\ge 2$. But $M$ is not unimodular, so $\rank\,J_{0}(M)=2$. 
		
	By \cite[92:1]{omeara_quadratic_1963}, we may let $J_{0}(M)\cong \langle 1,-\eta\rangle$ with $\eta\in \{1,\Delta\}$. Since $FM$ is anisotropic, $J_{0}(M)\cong \langle 1, -\Delta \rangle$. Since $\Delta\varepsilon\pi\not=\varepsilon\pi$, by Lemma \ref{lem:lattice-rep-binary}(viii), $M$ represents $\langle \Delta \varepsilon\pi\rangle $. So, by Lemma \ref{lem:non-dyadic-adc-rank}(vi), $\rank\, J_{1}(M)\ge 1$ and thus $\rank\, J_{1}(M)=1$.
		
	 Let $M\cong \langle 1,-\Delta,\delta\pi  \rangle$ with $\delta\in \{1,\Delta\}$. Since $\det FM=\det W_{2}^{3}(\varepsilon\pi)=-\varepsilon\pi$, $J_{1}(M)\cong \langle \Delta\varepsilon\pi\rangle$.  Hence, by \cite[92:2]{omeara_quadratic_1963}, $M\cong N_{2}^{3}(\varepsilon\pi)$.
\end{proof}
\begin{proof}[Proof of Theorem \ref{thm:primitive-nondyadic}]
	(i) For necessity, since $M$ is non-primitive, $\mathfrak{n}(M)\subsetneq \mathcal{O}_{F}$. By Theorem \ref{thm:nondyadic-ADC-binary} and Proposition \ref{prop:maximallattices-binary-ternary}(i), $M\cong N_{2}^{2}(\Delta)=\langle \pi,-\Delta\pi \rangle$.   Clearly, $FM\cong W_{2}^{2}(\Delta)=[\pi,-\Delta\pi]$, $\mathfrak{n}(M)=\mathfrak{p}$ and $\mathfrak{v}(M)=\mathfrak{p}^{2}$.
	
	For sufficiency, by \cite[92:2]{omeara_quadratic_1963}, $M\cong N_{2}^{2}(\Delta)=\langle \pi,-\Delta\pi\rangle$ is $\mathcal{O}_{F}$-maximal, and so by Lemma \ref{lem:ADC-sufficient-equiv}(i), $M$ is ADC.
	
	(ii) Necessity is clear. For sufficiency, by Theorem \ref{thm:nondyadic-ADC-ternary} and Proposition \ref{prop:maximallattices-binary-ternary}(i), one can see that each ternary ADC $\mathcal{O}_{F}$-lattice $M$ has $\mathfrak{n}(M)=\mathcal{O}_{F}$, and thus is primitive.
\end{proof}
 \begin{lem}\label{lem:thetaOM-NE-non-dyadic}
 	\begin{enumerate}[itemindent=-0.5em,label=\rm (\roman*)]
 		\item Suppose that $M=N_{\nu}^{3}(c)$ with $\nu\in \{1,2\}$ and $c\in \mathcal{V}$. Then
 		\begin{align*}
 			\theta(O^{+}(M))=\begin{cases}
 				 \mathcal{O}_{F}^{\times}F^{\times 2} &\text{if $(\nu,c)=(1,\mathcal{U})$}, \\
 				 F^{\times}  &\text{if $(\nu,c)\not=(1,\mathcal{U})$}.
 			\end{cases}
 		\end{align*}
 		
 		 \item  Suppose that $M=H_{1,h}^{3}(c)$ with $h\ge 1$ and $c\in \mathcal{V}$. Then
 		 \begin{align*}
 		 	\theta(O^{+}(M))=\begin{cases}
 		 		\mathcal{O}_{F}^{\times}F^{\times 2} &\text{if $c\in \mathcal{U}$}, \\
 		 		F^{\times}   &\text{if $c\in \pi\mathcal{U}$}.
 		 	\end{cases}
 		 \end{align*}
 		
 		\item Suppose that $M$ is ternary ADC. Then
 		 \begin{enumerate} 
 			\item[(a)] $\theta(O^{+}(M))\supseteq \mathcal{O}_{F}^{\times}F^{\times 2}$.
 			
 			\item[(b)] If $\theta(O^{+}(M))\subseteq N(F(\sqrt{a})/F)$, then $a\in  \{1,\Delta\}$; if moreover, $\theta(O^{+}(M))=F^{\times}$, then $a=1$.
 		\end{enumerate}	
 		\end{enumerate}
 \end{lem}
\begin{proof}
	 (i) By Proposition \ref{prop:maximallattices-binary-ternary}(i), $M$ has a Jordan component of rank at least $2$. So $\theta(O^{+}(M))\supseteq \mathcal{O}_{F}^{\times}F^{\times 2}$. Also, $\theta(O^{+}(M))=F^{\times}$ if $M$ has a Jordan component with odd scale, i.e., $(\nu,c)\not=(1,\mathcal{U})$. 
	 
	 (ii) Since $\theta(O^{+}(M))\supseteq \theta(O^{+}(\mathbf{H}))=\mathcal{O}_{F}^{\times}F^{\times 2}$, $\theta(O^{+}(M))=\mathcal{O}_{F}^{\times}F^{\times 2}$ or $F^{\times}$. Similar to (i), the latter case happens if $M$ has a Jordan component with odd scale, i.e., $c\in \pi\mathcal{U}$.
	 
	 (iii) From Remark \ref{re:nondyadicACD-3}, $M$ is either of the form $N_{\nu}^{3}(c)$ with $\nu\in \{1,2\}$ and $c\in \mathcal{V}$, or of the form $H_{1,h}^{3}(c)=\mathbf{H}\perp \langle c\pi^{2h} \rangle$ with $h\ge 1$ and $c\in \mathcal{V}$. This combined with (i) and (ii) shows the assertion (iii).
\end{proof}
\begin{proof}[Proof of Theorem \ref{thm:coM-nondyadic}]
		By Lemma \ref{lem:thetaOM-NE-non-dyadic}(iii)(b), $\co(M)\subseteq  \{1,\Delta\}$, and if moreover $\theta(O^{+}(M))=F^{\times}$, then $\co(M)\subseteq \{1\}$.
		
	    Note from Remark \ref{re:co-1-Delta} that $1\in \co(M)$ if and only if $FM$ is isotropic, i.e., $\nu=1$. Also, if $\theta(O^{+}(M))=\mathcal{O}_{F}^{\times}F^{\times 2}$, then by Lemma \ref{lem:co-nondyadic}(i), $\Delta\in \co(M)$. The proof is completed by a case-by-case check.
\end{proof}
	
\section{ADC lattices over dyadic local fields}\label{sec:dyadic}
In this section, we assume that $F$ is dyadic. Let $M\cong \prec a_{1},\ldots,a_{m}\succ $ be an $\mathcal{O}_{F}$-lattice of rank $m $ relative to some good BONG. Put $R_{i}=R_{i}(M)$ and $\alpha_{i}=\alpha_{i}(M)$. When an $\mathcal{O}_{F}$-lattice $N$ is considered, we also suppose $N=\prec b_{1},\ldots,b_{n}\succ$ relative to some good BONG, and denote by $S_{i}=R_{i}(N)$ and $\beta_{i}=\alpha_{i}(N)$.

Define the binary $\mathcal{O}_{F}$-lattice $ M_{\nu}^{2}(\Delta):=\prec \pi^{\nu-1},-\Delta\pi^{\nu+1-2e}\succ  $ relative to some good BONG. Such lattice indeed exists from Lemma \ref{lem:goodBONGequivcon}. Clearly, $FM_{\nu}^{2}(\Delta)\cong W_{\nu}^{2}(\Delta)$. 

\begin{thm}\label{thm:dyadic-binary}
	If $m=2$, then $ M $ is ADC if and only if either $M$ is $ \mathcal{O}_{F} $-maximal, or
	\begin{align*}
		M&\cong M_{\nu}^{2}(\Delta)\cong  2^{-1}\pi^{\nu} A(2\pi^{-1},2\rho\pi),  
	\end{align*}
	with $\nu\in \{1,2\}$, which is not $\mathcal{O}_{F}$-maximal. 
\end{thm}

 We keep the notation $H_{1,h}^{3}(c):=\mathbf{H}\perp \langle c\pi^{2h} \rangle$ with $c\in \mathcal{V} $ and  $h\ge 0$ as the previous section. Similarly, from Proposition \ref{prop:maximallattices-binary-ternary}(ii), $H_{1,h}^{3}(c)$ is $\mathcal{O}_{F}$-maximal if and only if $h=0$, and $FH_{1,h}^{3}(c)\cong W_{1}^{3}(c)$ is isotropic.
\begin{thm}\label{thm:dyadic-ternary}
 	If $m=3$, then $ M $ is ADC if and only if the following conditions hold:
    \begin{enumerate}[itemindent=-0.5em,label=\rm (\roman*)]
 	
 	\item If $FM $ is isotropic, then one of the following cases happens:
 	
 	\begin{enumerate}[itemindent=-0.5em,label=\rm (\roman*)]
 		
 		\item[(a)] $M$ is $\mathcal{O}_{F}$-maximal;
 		
 		\item[(b)]		$
 		M\cong  H_{1,h}^{3}(c)$ with $h\ge 1$ and $c\in \mathcal{V}$;
 		
 	   \item[(c)]  $
 	   M\cong N_{\nu}^{2}(\delta)\perp \langle \varepsilon \pi^{k}\rangle$,
 	   with $\nu\in\{1,2\}$,  $ \delta\in \mathcal{U}\backslash \{1,\Delta\}$, $ \varepsilon\in \mathcal{U} $ and $ k\in \{0,1\} $.
 \end{enumerate}
 	\item  If $FM $ is anisotropic, then one of the following cases happens:
 	\begin{enumerate}[itemindent=-0.5em,label=\rm (\roman*)]
 			
 	\item[(a)] $M$ is $\mathcal{O}_{F}$-maximal; 
 	
 	 \item[(b)] $
 	M\cong N_{\nu}^{2}(\delta\pi)\perp \langle \varepsilon \pi^{k}\rangle$,
 	with $\nu\in\{1,2\}$,  $ \delta\in \mathcal{U} $, $ \varepsilon\in \mathcal{U} $ and $ k\in \{1,2\} $;
 	
 	 \item[(c)] $
 	M\cong N_{\nu}^{2}(\delta)\perp \langle \varepsilon \pi^{k}\rangle$,
 	with $\nu\in\{1,2\}$,  $ \delta\in \mathcal{U}\backslash \{1,\Delta\}$, $ \varepsilon\in \mathcal{U} $ and $ k\in \{0,1\} $. Also, $M$ is $\mathcal{O}_{F}$-maximal if and only if $M\cong N_{2}^{3}(\varepsilon)$ with $\varepsilon\in \mathcal{U}$.
 	 \end{enumerate}
 \end{enumerate}
 \end{thm}
\begin{re}\label{re:dyadicACD-3}
		\begin{enumerate}[itemindent=-0.5em,label=\rm (\roman*)]		
		\item  From Proposition \ref{prop:maximallattices-binary-ternary}(ii), $M$ is not $\mathcal{O}_{F}$-maximal in the cases (i)(b), and (i)(c) and (ii)(b) of Theorem \ref{thm:dyadic-ternary}.
			
		\item For $\nu\in \{1,2\}$, from Proposition \ref{prop:maximallattices-binary-ternary}(ii) and \cite[Corollary 3.4(ii)]{beli_integral_2003}, we have 
		\begin{align*}
		N_{\nu}^{2}(\delta\pi)\cong \prec \Delta^{\nu-1},-\Delta^{\nu-1}\delta\pi \succ\cong \langle \Delta^{\nu-1},-\Delta^{\nu-1}\delta\pi \rangle,
	\end{align*}
	with $\delta\in \mathcal{U}$. From Proposition \ref{prop:maximallattices-binary-ternary}(ii) and \cite[Remark 3.8 and Lemma 3.9]{HeHu2}, we also see that
	 	\begin{align*}
		 N_{\nu}^{2}(\delta)\cong \prec  (\delta^{\#})^{\nu-1} ,- (\delta^{\#})^{\nu-1}\delta\pi^{1-d(\delta)}\succ\cong  (\delta^{\#})^{\nu-1}\pi^{-l}A(\pi^{l},-(\delta-1)\pi^{-l}),
 	\end{align*}
	with $\delta\in \mathcal{U}\backslash \{1,\Delta\}$ and $2l=d(\delta)-1\le 2e-2$; and
	\begin{align*}
		N_{2}^{3}(\varepsilon)\cong \prec \varepsilon \kappa^{\#},-\varepsilon\kappa^{\#}\kappa \pi^{2-2e},\varepsilon\kappa  \succ\cong 2^{-1}\pi A(2,2\rho)\perp \langle \Delta \varepsilon\rangle,
	\end{align*}
	with $\varepsilon\in \mathcal{U}$.
	
	\item Combining Theorem \ref{thm:dyadic-ternary} with Definition \ref{defn:alpha1-ternary} and Lemma \ref{lem:alpha1-ternary} below, we see that if $M$ is ternary ADC, then $M$ is either of the form $M_{\nu,r}^{3}(c)$, with  $\nu\in \{1,2\}$, $r\in \{-1/2,0,\ldots,e\}$ and $c\in \mathcal{V}$ satisfying $(\nu,r,c)\notin \{(2,e,\mathcal{U}),(1,-1/2,\mathcal{V})\}$, or of the form $H_{1,h}^{3}(c)$, with $h\ge 1$ and $c\in \mathcal{V}$.
\end{enumerate}
\end{re}
 \begin{thm}\label{thm:primitive-dyadic}
	Let $m\in \{2,3\}$.
	\begin{enumerate}[itemindent=-0.5em,label=\rm (\roman*)]
		\item  If $m=2$, then $M$ is non-primitive  ADC if and only if $FM\cong [\pi,-\Delta\pi]$, $\mathfrak{n}(M)=\mathfrak{p}$ and $\mathfrak{v}(M)=\mathfrak{p}^{2-2e}$ or $\mathfrak{p}^{4-2e}$.
		
		\item  If $m=3$, then $M$ is primitive ADC if and only if $M$ is ADC.
	\end{enumerate}
\end{thm}
\begin{thm}\label{thm:coM-dyadic}
 	Let $\nu \in \{1,2\}$, $h\ge 1$, $r\in \{-1/2,0,\ldots,e\}$ and $c \in \mathcal{V}$. Suppose that $M$ is ternary ADC.
 	\begin{enumerate}[itemindent=-0.5em,label=\rm (\roman*)]
 		\item  When $M=M_{\nu,r}^{3}(c)$, we have
 		\begin{align*}
 			\co(M)=
 			\begin{cases}
 				\{1\}  &\text{if $(\nu,c)=(1,\mathcal{V})$ and $0\le r\le e$},\\
 				\emptyset  &\text{if $\nu=2$},
 			\end{cases}
 		\end{align*}
 		\item  When $M=H_{1,h}^{3}(c)$, we have
 		\begin{align*}
 			\co(M)=\begin{cases}
 				 \{1,\Delta\} &\text{if $c\in \mathcal{U}$}, \\
 				\{1\}  &\text{if $c\in \pi\mathcal{U}$}.
 			\end{cases}
 		\end{align*} 
 	\end{enumerate}
 \end{thm}
 \begin{re}\label{re:r-sM-relation}
    When $M=M_{\nu,r}^{3}(c)$, by Lemma \ref{lem:R2} and Definition \ref{defn:alpha1-ternary}, we have $R_{1}=0$ and $R_{2}=-2r$. Hence, if $r\not=-1/2$, then $-r\le 0$, and so by \cite[Corollary 4.4(iv)]{beli_integral_2003}, we have $\mathfrak{s}(M)=\mathfrak{p}^{-r}$. 
 \end{re}
 
As the previous section, we first handle Theorems \ref{thm:dyadic-binary} and \ref{thm:dyadic-ternary}, and then prove Theorems \ref{thm:primitive-dyadic} and \ref{thm:coM-dyadic}.
\begin{lem}\label{lem:dyadic-adc-R}
	Let $m\in \{2,3\}$.
	\begin{enumerate}[itemindent=-0.5em,label=\rm (\roman*)]
	 \item  If $M$ represents $\prec \delta a_{1}\succ$ for some $\delta\in \mathcal{U}$, with $d(\delta)=1$, then $\alpha_{1}\in \{0,1\}$ and $R_{2}-R_{1}\in [-2e,0]^{E}\cup \{1\}$.
	
 	\item  If $M$ represents $\prec \varepsilon_{1} \succ$ and $\prec \varepsilon_{2} \pi \succ$ for some $\varepsilon_{1},\varepsilon_{2}\in \mathcal{U}$, then $R_{1}=0$ and either $R_{2}=-2e$ or $ R_{2}\le 1-d[-a_{1,2}]$. Therefore, $R_{2}\in [-2e,0]^{E}\cup \{1\} $ and $\alpha_{1}\in \{0,1\}$.
	\end{enumerate}	
\end{lem}  
\begin{proof}
	Since $M$ is integral, by \cite[Lemma 2.2]{beli_universal_2020}, we have $R_{1}\ge 0$. By \eqref{eq:BONGs}, $R_{2}\ge R_{1}-2e\ge -2e$.
	
%	For (i) and (ii), let $N=\prec \varepsilon\succ$ or $\prec \varepsilon\pi\succ$. Then $S_{1}=0$ or $1$. If $M$ represens $N$, by Theorem \ref{thm:beligeneral-1}(i), $R_{1}\le S_{1}$. Hence $R_{1}\le 0$ or $R_{1}\le 1$, as desired.
	
	 For (i), let $N=\prec\delta a_{1}\succ$ with $\delta\in \mathcal{U}$ and $d(\delta)=1 $. Then $b_{1}=\delta a_{1} $ and thus $S_{1}=R_{1} $. It follows that $d(a_{1}b_{1})=d(\delta)=1$. Since $M$ represents $\prec \delta a_{1}\succ$, by Proposition \ref{prop:alphaproperty}(ii) and Theorem \ref{thm:beligeneral-1}(ii), we see that
	\begin{align*}
		\alpha_{1}=\min\{\dfrac{R_{2}-S_{1}}{2}+e,R_{2}-S_{1}+d[-a_{1,2}]\}=A_{1}\le d[a_{1}b_{1}]\le d(a_{1}b_{1})=1.
	\end{align*}
	So, by Proposition \ref{prop:alphaproperty}(i), $\alpha_{1}\in \{0,1\}$. Hence, by Proposition \ref{prop:alphaproperty}(iii) and (iv),  $R_{2}-R_{1}\in [-2e,0]^{E}\cup \{1\} $.
	
%	If $d(-a_{1}a_{2})\ge \alpha_{1}$, then $d[-a_{1,2}]=\alpha_{1}$ and so $
%		\alpha_{1}=\min\{(R_{2}-R_{1})/2+e,R_{2}-R_{1}+\alpha_{1}\} $, as required.
   
   For (ii), since $M$ represents $\prec \varepsilon_{1}\succ$, by Theorem \ref{thm:beligeneral-1}(i), $R_{1}\le \ord(\varepsilon_{1})=0$ and thus $R_{1}=0$. Let $N=\prec \varepsilon_{2}\pi\succ$. Then $S_{1}=1$. It follows that $\ord(a_{1}b_{1})$ is odd and thus $d[a_{1}b_{1}]=0$. Since $M$ represents $N$, by Theorem \ref{thm:beligeneral-1}(ii), we see that
   \begin{align*}
   	 \min\{\dfrac{R_{2}-1}{2}+e,R_{2}-1+d[-a_{1,2}]\}=A_{1}\le d[a_{1}b_{1}]=0.
   \end{align*}
   Hence $R_{2}\le 1-2e$ or $R_{2}\le 1-d[-a_{1,2}]\le 1 $. So, by Proposition \ref{prop:Rproperty}(iv), $R_{2}=R_{2}-R_{1}$ must be even if $R_{2}<0$. Hence, in the former case, $R_{2}=-2e$. Also, $R_{2}\in [-2e,0]^{E}\cup \{1\} $. Since $R_{2}-R_{1}\le R_{2}\le  1$, Propositions \ref{prop:Rproperty}(i) and \ref{prop:alphaproperty}(i) implies $\alpha_{1}\in\{0,1\} $.
\end{proof}

\noindent\textbf{Binary cases.}

 We need two lemmas to characterize ADC lattices that is not $\mathcal{O}_{F}$-maximal.
\begin{lem}\label{lem:Mnu-ADC}
	Let $\nu\in \{1,2\}$. Then $M_{\nu}^{2}(\Delta)$ is ADC, but not $\mathcal{O}_{F}$-maximal.
\end{lem}

\begin{proof}
	Let $N_{\nu}=\prec \varepsilon\pi^{\nu-1}\succ$ with $\varepsilon\in \mathcal{U}$. Since $FM_{\nu}^{2}(\Delta)\cong W_{\nu}^{2}(\Delta)$ represents $FN_{\nu}=[\varepsilon\pi^{\nu-1}]$ for all $\varepsilon\in \mathcal{U}$, by Lemma \ref{lem:ADC-sufficient-equiv}(ii), it suffices to show that $M_{\nu}$ represents all the corresponding lattices $N_{\nu}$.   Then we will verify conditions (i) and (ii) in Theorem \ref{thm:beligeneral-1} for $M=M_{\nu}^{2}(\Delta)$ and $N=N_{\nu}$.
	
	Clearly, $R_{1}=S_{1}=\nu-1$ and $R_{2}=\nu+1-2e$; thus condition (i) holds. Also, we have  $R_{2}-R_{1}=R_{2}-S_{1}=2-2e$ and $d[-a_{1,2}]=d(-a_{1}a_{2})=d(\Delta)=2e$. So, Proposition \ref{prop:alphaproperty}(iv) implies $\alpha_{1}=1$.
	
	 Since $\ord(a_{1}b_{1})=R_{1}+S_{1}$ is even,  $  d(a_{1}b_{1})\ge 1$. Thus, $d[a_{1}b_{1}]=\min\{d(a_{1}b_{1}),\alpha_{1}\}=1$. So
	\begin{align*}
		A_{1}=\min\{\dfrac{R_{2}-S_{1}}{2}+e,R_{2}-S_{1}+d[-a_{1,2}]\}=\min\{1,2\}=1=d[a_{1}b_{1}];
	\end{align*} 
	thus condition (ii) holds. Therefore, $M_{\nu}^{2}(\Delta)$ is ADC.
	
	Since $(R_{1} ,R_{2})=(\nu-1,\nu+1-2e)=(0,2-2e)$ or $(1,3-2e)$, by \cite[Lemma 4.11(i)]{He25} with $n=2$, $M_{\nu}^{2}(\Delta)$ is not $\mathcal{O}_{F}$-maximal.
\end{proof}
\begin{lem}\label{lem:binary-exceptions}
	Let $\nu\in\{1,2\}$ and $M_{\nu}$ be an $\mathcal{O}_{F}$-lattice. If $FM_{\nu}\cong W_{\nu}^{2}(\Delta)$, $R_{1}(M_{\nu})=\nu-1$ and $R_{2}(M_{\nu})=\nu+1-2e$, then $M_{\nu}\cong M_{\nu}^{2}(\Delta)=\prec \pi^{\nu-1},-\Delta\pi^{\nu+1-2e}\succ$. 
\end{lem}
\begin{proof}
 To show $M_{\nu}\cong M_{\nu}^{2}(\Delta)$, it suffices to verify that conditions (i)-(iii) in \cite[Theorem 3.2]{beli_Anew_2010} are fulfilled for $M=M_{\nu}$ and $N=M_{\nu}^{2}(\Delta)$.

Clearly, $ R_{1}=S_{1}=\nu-1$ and $ R_{2}=S_{2}=\nu+1-2e$; thus condition (i) holds. Since  $R_{2}-R_{1}=2-2e$, Proposition \ref{prop:alphaproperty}(iv) implies $\alpha_{1}=\beta_{1}=1$; thus condition (ii) holds. Since $\ord(a_{1}b_{1})=R_{1}+S_{1}=2R_{1}$ is even, $d(a_{1}b_{1})\ge 1=\alpha_{1}$; thus condition (iii) holds.
\end{proof}
\begin{proof}[Proof of Theorem \ref{thm:dyadic-binary}]	
	Sufficiency follows by Lemmas \ref{lem:ADC-sufficient-equiv}(i) and  \ref{lem:Mnu-ADC}. To show the necessity, we divide into five cases.
	
	(i) If $FM\cong W_{1}^{2}(1) $, by Proposition \ref{prop:nu-isotropic}(i) and Corollary \ref{cor:ADC-universal}, $M\cong\mathbf{H}=N_{1}^{2}(1)$.
	
	(ii)  If $FM\cong W_{\nu}^{2}(\Delta)$, by Lemma \ref{lem:lattice-rep-binary}(ii) and (iii), $M$ represents $\prec 
	\varepsilon \pi^{\nu-1}\succ$, but not $\prec \varepsilon\pi^{2-\nu}\succ$, for all $\varepsilon\in \mathcal{U}$. Hence, by Theorem \ref{thm:beligeneral-1}(i), $R_{1}\le \nu-1$ and thus $R_{1}\in \{0,1\}$. Since $FM $ represents $[a_{1}]$ and $M$ is ADC, $M$ represents $\prec a_{1}\succ$. So $R_{1}=\ord(a_{1})\equiv \nu-1\pmod{2}$ and thus $R_{1}=\nu-1$. Let $a_{1}=\eta\pi^{\nu-1}$ with $\eta\in\mathcal{U}$. Choose $\delta\in \mathcal{U}$ such that $d(\delta)=1$ (e.g. $\delta=1+\pi$). As $\delta\eta\in \mathcal{U}$, $M$ represents $\prec \delta\eta\pi^{\nu-1}\succ=\prec \delta a_{1}\succ$. Hence, by Lemma \ref{lem:dyadic-adc-R}(i), $\alpha_{1}=0$ or $1$.
	
	If $\alpha_{1}=0$, then $R_{2}=R_{1}-2e=\nu-1-2e$. Hence, by \cite[Lemma 4.11(i) and (ii)]{He25}, $M\cong N_{\nu}^{2}(\Delta)$. If $\alpha_{1}=1$, Proposition \ref{prop:alphaproperty}(ii) implies that 
	\begin{align*}
	&\dfrac{R_{2}-R_{1}}{2}+e=1 \quad\text{or} \quad R_{2}-R_{1}+d(-a_{1}a_{2})=1.
	\end{align*}
Note that $R_{2}-R_{1}=1-d(a_{1}a_{2})=1-2e<0$ is odd and hence the latter case is ruled out by Proposition \ref{prop:Rproperty}(iv). So the former case holds, which gives  $R_{2}=\nu+1-2e$. Hence, by Lemma \ref{lem:binary-exceptions}, $M\cong M_{\nu}^{2}(\Delta) $, which is isometric to $2^{-1}\pi^{\nu} A(2\pi^{-1},2\rho\pi)$ from \cite[Corollary 3.4(iii)]{beli_integral_2003} and \cite[93:17]{omeara_quadratic_1963}.
	
	 (iv) If $FM\cong W_{\nu}^{2}(\varepsilon\pi)$ with $\varepsilon\in \mathcal{U}$, by Lemma \ref{lem:lattice-rep-binary}(iv), $M$ represents $\prec\varepsilon_{1}\succ$ and $\prec\varepsilon_{2}\pi\succ$. So,  by Lemma \ref{lem:dyadic-adc-R}(ii), $R_{1}=0$ and $R_{2}\in [-2e,0]^{E}\cup\{1\}$. Since $\det FM=-\varepsilon\pi$, $\ord(a_{1}a_{2}) $ is odd, so $R_{2}=1$. Hence, by \cite[Lemma 4.11(iii)]{He25}, $M\cong N_{\nu}^{2}(\varepsilon\pi)$.
   
  	(v) If $FM\cong W_{\nu}^{2}(\delta)$ with $\delta\in \mathcal{U}\backslash \{1,\Delta\}$. Then $1\le d(-a_{1}a_{2})=d(\delta)<2e$ is odd. By Lemma \ref{lem:lattice-rep-binary}(v) and (vi), $M$ represents $\prec \varepsilon_{1}\succ$ and $\prec \varepsilon_{2}\pi \succ$ for some $\varepsilon_{1},\varepsilon_{2}\in \mathcal{U}$. Hence, by Lemma \ref{lem:dyadic-adc-R}(ii), $R_{1}=0$ and either $R_{2}=-2e$ or $R_{2}\le 1-d[-a_{1,2}]=1-d(-a_{1}a_{2})$. If $R_{2}=-2e$, then Proposition \ref{prop:alphaproperty}(v) implies that $d(\delta)=d(-a_{1}a_{2})\ge d[-a_{1,2}]\ge 2e $, a contradiction. So $R_{2}\not=-2e$ and thus $R_{2}\le 1-d(-a_{1}a_{2})$. Recall from \eqref{eq:BONGs} that $R_{2}+d(-a_{1}a_{2})=R_{2}-R_{1}+d(-a_{1}a_{2})\ge 0$. Combining these two inequalities, we have
  \begin{align*}
  	 -d(-a_{1}a_{2})\le R_{2}\le 1-d(-a_{1}a_{2}).
  \end{align*}
   However, $R_{2}=R_{1}+R_{2}=\ord(a_{1}a_{2})$ is even, so $R_{2}\not=-d(-a_{1}a_{2})$. Therefore, $R_{2}=1-d(-a_{1}a_{2})=1-d(\delta)$. Hence, by \cite[Lemma 4.11(iii)]{He25}, $M\cong N_{\nu}^{2}(\delta)$.
\end{proof}

\noindent\textbf{Ternary cases.}
 
Assume that $m=3$ and $M\cong\prec a_{1},a_{2},a_{3}\succ$ relative to some good BONG. To characterize ternary ADC lattices, we consider two cases: $FM\cong W_{1}^{3}(c)$  or $ W_{2}^{3}(c)$.
\begin{lem}\label{lem:W13k}
	Let $FM\cong W_{1}^{3}(\varepsilon\pi^{k})$ with  $\varepsilon\in \mathcal{U}$ and $k\in \{0,1\}$. Then $M$ is ADC if and only if $R_{1}=0$ and one of the following conditions holds:
	\begin{enumerate}[itemindent=-0.5em,label=\rm (\roman*)]
		\item $\alpha_{1}=0$, i.e., $R_{2}=-2e$, $R_{3}\ge 0$ and $-a_{1}a_{2}\in F^{\times 2}$; equivalently, $M\cong H_{1,h}^{3}(c)$ with $c\in \mathcal{V}$ and $h\ge 0$.
		
		\item $\alpha_{1}=1$, $R_{2}\in [2-2e,0]^{E}$ and $R_{3}=k$. 
	\end{enumerate}
\end{lem}
\begin{proof}
	From Proposition \ref{prop:nu-isotropic}(i), $FM$ is isotropic and $M$ is universal. Hence conditions (i) and (ii) of the lemma follows from \cite[Theorem 2.1]{beli_universal_2020} with $m=3$.
	
	If $\alpha_{1}=0$, then $R_{2}=R_{1}-2e=-2e$. Since $-a_{1}a_{2}\in F^{\times 2}$, $\prec a_{1},a_{2}\succ\cong \mathbf{H}$. Since $R_{2}<0\le R_{1}\le R_{3}$, by \cite[Corollary 4.4(i)]{beli_integral_2003}, 
	\begin{align*}
		M\cong \prec a_{1},a_{2}\succ \perp \prec a_{3}\succ\cong \mathbf{H}\perp \langle c\pi^{2h}\rangle=H_{1,h}^{3}(c),
	\end{align*}
	 with $h\ge 0$ and $c=\varepsilon\pi^{k}\in \mathcal{V}$, as required. 
\end{proof}
\begin{lem}\label{lem:W23k}
	Let $FM\cong W_{2}^{3}(\varepsilon\pi^{k})$ with $\varepsilon\in \mathcal{U}$ and $k\in \{0,1\}$. Then $M$ is ADC if and only if $R_{1}=0$ and one of the following conditions holds:
	\begin{enumerate}[itemindent=-0.5em,label=\rm (\roman*)]
		\item $\alpha_{1}=0$, i.e., $R_{2}=-2e$, and $R_{3}=k$;
		
		\item $\alpha_{1}=1$, $R_{2}\in [2-2e,0]^{E}$  and $R_{3}=k$;
		
		\item $ R_{2}=1$ and $R_{3}=k+1$.
	\end{enumerate}
\end{lem}

First, by Proposition \ref{prop:alphaproperty}(iv), if $R_{2}-R_{1}=1$, then $\alpha_{1}=1$. This combined with Lemma \ref{lem:alpha1} shows that Lemma \ref{lem:W23k} is equivalent to Lemma \ref{lem:W23k-d}. We will show Lemma \ref{lem:W23k-d} instead.
\begin{lem}\label{lem:alpha1}
	If $R_{1}=0$, $\alpha_{1}=1$ and $R_{3}\in \{0,1\}$, then $d[-a_{1,2}]=1-R_{2}$.
\end{lem}
\begin{proof}
	Since $R_{1}=0$, by Proposition \ref{prop:alphaproperty}(vi),  $d[-a_{1,2}]\ge 1-R_{2}$, and $d[-a_{1,2}]=1-R_{2}$ if $R_{2}\not=2-2e$. So it suffices to show that $d[-a_{1,2}]\le 1-R_{2}=2e-1$ under the assumption that $R_{2}=2-2e$. Since $R_{3}-R_{2}=2e-2$ or $2e-1$, Proposition \ref{prop:Rproperty}(ii) and (iii) implies that  $d[-a_{1,2}]\le \alpha_{2}=2e-1$, as required.
\end{proof}
\begin{lem}\label{lem:W23k-d}
	Let $FM\cong W_{2}^{3}(\varepsilon\pi^{k})$ with $\varepsilon\in \mathcal{U}$ and $k\in \{0,1\}$. Then $M$ is ADC if and only if $R_{1}=0$ and one of the following conditions holds:
	\begin{enumerate}[itemindent=-0.5em,label=\rm (\roman*)]
		\item $\alpha_{1}=0$, i.e., $R_{2}=-2e$, and $R_{3}=k$;
		
		\item $\alpha_{1}=1$, $d[-a_{1,2}]= 1-R_{2}$, $R_{2}\in [2-2e,0]^{E}$ and $R_{3}=k$;
		
		\item $\alpha_{1}=1$, $d[-a_{1,2}]= 1-R_{2}$, $R_{2}=1$ and $R_{3}=k+1$.
	\end{enumerate}
\end{lem}

In what follows, we assume that $FM\cong W_{2}^{3}(\varepsilon\pi^{k})$ with $\varepsilon\in \mathcal{U}$ and $k\in \{0,1\}$, and  $a_{i}=\varepsilon_{i}\pi^{R_{i}}$ with $\varepsilon_{i}\in \mathcal{U}$. Then $-\varepsilon\pi^{k}=a_{1,3}\in  \varepsilon_{1,3}\pi^{k}F^{\times 2}$. Also, from Proposition \ref{prop:nu-isotropic}(ii), $FM$ is anisotropic. 
\begin{prop}\label{prop:W23k-R-alpha}
	  Suppose that $M$ is $n$-ADC. Then  $R_{1}=0$ and either $\alpha_{1}=0$, or $\alpha_{1}=1$ and $d[-a_{1,2}]=1-R_{2}$.
\end{prop}
\begin{proof}
	By Lemma \ref{lem:lattice-rep-binary}(viii), $M$ represents $\prec \varepsilon_{1}\succ$ and $\prec\varepsilon_{2}\pi\succ$ for some $\varepsilon_{1},\varepsilon_{2}\in \mathcal{U}$ and $\varepsilon_{1},\varepsilon_{2}\not=\varepsilon$. Hence, by Lemma \ref{lem:dyadic-adc-R}(ii), $R_{1}=0$ and either $R_{2}=-2e$ or $R_{2}\le 1-d[-a_{1,2}]$; and $\alpha_{1}\in \{0,1\}$. If $\alpha_{1}\not=0$, then, by Proposition \ref{prop:alphaproperty}(iii), $R_{2}\not=-2e$ and thus $ d[-a_{1,2}]\le 1-R_{2}$. Since $\alpha_{1}=1$, by Proposition \ref{prop:alphaproperty}(vi), $d[-a_{1,2}]\ge 1-R_{2}$ and thus $d[-a_{1,2}]=1-R_{2}$. 
\end{proof}
\begin{lem}\label{lem:W23k-alpha0}
	Suppose that $R_{1}=0$ and $R_{2}=-2e$. If $R_{3}>1$, then Theorem \ref{thm:beligeneral-1}(iii) fails for $M$ and $N=\prec  \eta\pi \succ$ with $\eta\in \mathcal{U}$.
\end{lem}
\begin{proof}
	Since $R_{2}-R_{1}=-2e$, by Proposition \ref{prop:alphaproperty}(v), $d(-a_{1}a_{2})\ge d[-a_{1,2}]\ge 2e$, so $-a_{1}a_{2}\in F^{\times 2}\cup \Delta F^{\times 2}$. However, $FM$ is anisotropic, so is $[a_{1},a_{2}]$ and thus $-a_{1}a_{2}\in \Delta F^{\times 2}$. As $\ord(a_{1})=R_{1}=0$, $[a_{1},a_{2}]\cong [1,-\Delta]$.
	
	Let $N=\prec \eta\pi\succ$. Then  $R_{3}>1=S_{1}$. Since $d[-a_{1,2}]\ge 2e$ and $d (-a_{1,3}b_{1}) \ge 0$, we see that
	\begin{align*}
		d[-a_{1,2}]+d (-a_{1,3}b_{1})\ge  2e+0=2e>2e+S_{1}-R_{3}.
	\end{align*}
	However, by Proposition \ref{prop:ternary-space-rep}, $[a_{1},a_{2}]\cong [1,-\Delta]$ fails to represent $[b_{1}]=[\eta\pi]$.
\end{proof}
 \begin{lem}\label{lem:W23k-alpha1}
	Suppose that $R_{1}=0$, $R_{2}\in [2-2e,0]^{E}\cup \{1\}$, $\alpha_{1}=1$ and $d[-a_{1,2}]= 1-R_{2}$. 
	\begin{enumerate}[itemindent=-0.5em,label=\rm (\roman*)]
	\item If $R_{2}\in [2-2e,0]^{E}$ and $R_{3}>k$, then Theorem \ref{thm:beligeneral-1}(iii) fails for $M$ and $N=\prec -\Delta \varepsilon_{1,3}\pi^{k}\succ$.
	
	\item	If $R_{2}=1$  and $R_{3}>k+1$, then Theorem \ref{thm:beligeneral-1}(iii) fails for $M$ and $N=\prec -\kappa\varepsilon_{1,3}\pi^{k} \succ$ or $\prec -\kappa\Delta\varepsilon_{1,3}\pi^{k} \succ$, where $\kappa$ is the unit defined as in  Proposition \ref{prop:maximallattices-binary-ternary}(ii).
 	\end{enumerate} 
\end{lem}
\begin{proof}
	  Recall from Proposition \ref{prop:maximallattices-binary-ternary}(ii) that $d(\kappa)=2e-1$ and thus $d(\kappa\Delta)=2e-1$, from the domination principle. So
	\begin{align*}
	 d(-a_{1,3}b_{1})=
		\begin{cases}
			  d(\Delta)=2e     &\text{if $N=\prec -\Delta \varepsilon_{1,3}\pi^{k} \succ$},\\
			  d(\kappa)=2e-1  &\text{if $N=\prec -\kappa \varepsilon_{1,3}\pi^{k} \succ$},\\
			  d(\kappa\Delta)=2e-1 &\text{if $N=\prec -\kappa\Delta\varepsilon_{1,3}\pi^{k} \succ$}.
		\end{cases}
	\end{align*}
	 Hence $R_{3}>k=S_{1}$ and
	\begin{align*}
		\;&d[-a_{1,2}]+d (-a_{1,3}b_{1})\\
		=\;& \begin{cases}
		(1-R_{2})+2e\ge 2e+1  &\text{if $R_{2}\in [2-2e,0]^{E}$ and $N=\prec -\Delta \varepsilon_{1,3}\pi^{k} \succ$},\\
		(1-R_{2})+2e-1\ge 2e-1&\text{if $R_{2}=1$ and $N=\prec -\kappa \varepsilon_{1,3}\pi^{k} \succ$ or $\prec -\kappa\Delta\varepsilon_{1,3}\pi^{k} \succ$},
		\end{cases}\\
	 >\;&2e+S_{1}-R_{3}.
	\end{align*}
	
	 (i)  Since $\ord(a_{1}a_{2}) $ is even, $(-a_{1}a_{2}, \Delta)_{\mathfrak{p}}=1$. Also, since $FM$ is anisotropic, by Proposition \ref{prop:ternary-space-rep}, $(-a_{1}a_{2},-a_{1}a_{3})_{\mathfrak{p}}=-1$. So
	  	 	 $(-a_{1}a_{2},-\Delta a_{1}a_{3})_{\mathfrak{p}}=-1$. Hence, by Proposition \ref{prop:ternary-space-rep} again, $[a_{1},a_{2}]$ fails to represent $[b_{1}]\cong [-\Delta a_{1,3}]$.
	  	 
	  	  (ii) Since  $\ord(a_{1}a_{2}) $ is odd, by \cite[63:11]{omeara_quadratic_1963}, $[a_{1},a_{2}]$ fails to represent $[b_{1}]\cong [-\kappa a_{1,3}]$ or $ [-\kappa\Delta a_{1,3}]$.
\end{proof}
 \begin{proof}[Proof of the necessity Lemma \ref{lem:W23k-d}]
 	 By Proposition \ref{prop:W23k-R-alpha}, $R_{1}=0$ and either $\alpha_{1}=0$, or $\alpha_{1}=1$ and $d[-a_{1,2}]=1-R_{2}$. Also, $R_{3}\ge 0$.
 	 
 	 If $\alpha_{1}=0$, by Proposition \ref{prop:alphaproperty}(iii), $R_{2}=R_{2}-R_{1}=-2e$. Assume $R_{3}>1$. By Lemma \ref{lem:W23k-alpha0}, $M$ fails to represent $N=\prec\eta\pi\succ$ for all $\eta\in \mathcal{U}$. But by Lemma \ref{lem:lattice-rep-binary}(viii), $M$ represents $\prec \eta\pi\succ$ for all $\eta\in \mathcal{U}$ with $\eta\not=\varepsilon$, a contradiction. So $R_{3}\in \{0,1\}$. Note that $R_{3}\equiv \ord(a_{1,3}) \equiv k\pmod{2} $ and hence $R_{3}=k$.
 	 
 	 Assume that $\alpha_{1}=1$ and $d[-a_{1,2}]=1-R_{2}$. Then $R_{2}\in [2-2e,0]^{E}\cup \{1\}$.
 	 
 	 If $R_{2}\in [2-2e,0]^{E}$ and $R_{3}>k$, by Lemma \ref{lem:W23k-alpha1}(i), $M$ fails to represent $N=\prec -\Delta \varepsilon_{1,3}\pi^{k}\succ$. But since $-\Delta\varepsilon_{1,3}\pi^{k}\not=\varepsilon\pi^{k}$, by Lemma \ref{lem:lattice-rep-binary}(viii), $M$ represents $N$, a contradiction. Hence $R_{3}\le k $ and thus $R_{3}\in \{0,1\}$. Similar to the previous case, we deduce $R_{3}=k$.
 	 	 
 	 If $R_{2}=1$ and $R_{3}>k+1$, by Lemma \ref{lem:W23k-alpha1}(ii), $M$ fails to represent  either $ \prec -\kappa  \varepsilon_{1,3}\pi^{k}\succ$ or $ \prec -\kappa\Delta \varepsilon_{1,3}\pi^{k}\succ$. But $-\kappa \varepsilon_{1,3}\pi^{k}\not=\varepsilon\pi^{k}$ and $-\kappa\Delta\varepsilon_{1,3}\pi^{k}\not=\varepsilon\pi^{k}$, by Lemma \ref{lem:lattice-rep-binary}(viii), $M$ represents both $\prec-\kappa \varepsilon_{1,3}\pi^{k}\succ $ and $\prec-\kappa\Delta\varepsilon_{1,3}\pi^{k}\succ$, a contradiction. So $R_{3}\le k+1$. Note that $R_{3}\equiv \ord(a_{1,3})+1 \equiv k+1\pmod{2}  $ and hence $(k,R_{3})\in \{(0,1),(1,0),(1,2)\}$. If $(k,R_{3})=(1,0)$, then $R_{2}>R_{3}$, which contradicts Proposition \ref{prop:Rproperty}(iv). Therefore, $R_{3}=k+1$.
 \end{proof}
 \begin{lem}\label{lem:k=0}
 	Let $N=\prec b_{1}\succ$. Suppose that $R_{1}=0$,  $R_{2}=1$ and $FM $ represents $FN=[b_{1}]$. If $d(-a_{1,3}b_{1})\ge 2e$, then $[a_{1},a_{2}]$ represents $[b_{1}]$.
 \end{lem}
 \begin{proof}
   If $d(-a_{1,3}b_{1})\ge 2e$, then $-a_{1,3}b_{1}\in F^{\times 2}\cup \Delta F^{\times 2}$. Note that $-a_{1,3}\in \varepsilon\pi^{k} F^{\times 2}$ and hence $b_{1}\in \varepsilon \pi^{k}F^{\times 2}\cup \Delta\varepsilon\pi^{k}F^{\times 2}$. By Lemma \ref{lem:lattice-rep-binary}(viii), $FM$ does not represent $[\varepsilon\pi^{k}]$, so $b_{1}\in \Delta\varepsilon\pi^{k} F^{\times 2}$ from the hypothesis. The subspace $[a_{1},a_{2}]$ of $FM$ does not represent $[\varepsilon\pi^{k}]$, neither. Since $\ord(a_{1}a_{2})$ is odd, by \cite[63:11]{omeara_quadratic_1963}, $[a_{1},a_{2}]$ represents $[b_{1}]=[\Delta\varepsilon\pi^{k}]$.
 \end{proof}
\begin{proof}[Proof of sufficiency of Lemma \ref{lem:W23k-d}]
	Since $FM \cong W_{\nu}^{2}(\varepsilon\pi^{k})$ represents $[c]$ for all $c\in \mathcal{V}$ with $c\not=\varepsilon\pi^{k}$. Hence, by Lemma \ref{lem:ADC-sufficient-equiv}(ii), it suffices to show that $M$ represents all the corresponding lattices $\prec c\succ$. Let $N=\prec c\succ$. Then $S_{1}=\ord(c)\in \{0,1\}$.
	
	Clearly, $R_{1}=0\le S_{1}$. Thus, Theorem \ref{thm:beligeneral-1}(i) is verified.  
	
	Since $\ord(a_{1}b_{1})\equiv S_{1}\pmod{2}$, if $S_{1}=0$, then $d(a_{1}b_{1})\ge 1\ge \alpha_{1}$; if  $S_{1}=1$, then $d(a_{1}b_{1})=0$. Hence
	\begin{align*}
		d[a_{1}b_{1}]=\min\{d(a_{1}b_{1}),\alpha_{1}\}=\begin{cases}
			\alpha_{1}  &\text{if $S_{1}=0$}, \\
			0  &\text{if $S_{1}=1$}. 
		\end{cases}
	\end{align*}
 If $S_{1}=0$, then $S_{1}=R_{1}$. By Proposition \ref{prop:alphaproperty}(ii),
	\begin{align*}
		A_{1}=\min\{\dfrac{R_{2}-S_{1}}{2}+e,R_{2}-S_{1}+d[-a_{1,2}]\}=\alpha_{1}= d[a_{1}b_{1}].
	\end{align*}
		If $\alpha_{1}=0$, then, by Proposition \ref{prop:alphaproperty}(iii), $R_{2}=R_{2}-R_{1}=-2e$. Hence
	\begin{align*}
		A_{1}\le (R_{2}-S_{1})/2+e=(-2e-S_{1})/2+e=-S_{1}/2\le 0\le d[a_{1}b_{1}].
	\end{align*}
	Assume that $S_{1}=\alpha_{1}=1$ and $d[-a_{1,2}]=1-R_{2}$. Then
		\begin{align*}
		A_{1}\le   R_{2}-S_{1}+d[-a_{1,2}]=R_{2}-1+(1-R_{2})=0=d[a_{1}b_{1}].
	\end{align*}
	Thus, Theorem \ref{thm:beligeneral-1}(ii) is verified. 
	
	Assume that $R_{3}>S_{1}$. Otherwise, Theorem \ref{thm:beligeneral-1}(iii) is trivial.
	
	If $\alpha_{1}=0$, then $R_{3}\in \{0,1\}$ and thus $S_{1}=0$. Since $R_{2}-R_{1}=-2e$ and $c\in \mathcal{U}$, by \cite[Corollary 2.3(ii)]{HeHu2}, $[a_{1},a_{2}]\cong \mathbb{H}$ or $[1,-\Delta]$, representing $[b_{1}]=[c]$ in either case.
	
	Suppose that $\alpha_{1}=1$ and $d[-a_{1,2}]=1-R_{2}$. Then  $(k,R_{3})\in \{(0,0),(0,1),(1,1),(1,2)\}$. From the assumption $R_{3}>S_{1}$, we only need to consider four cases:
	\begin{align*}
		(k,R_{3},S_{1})=\{(0,1,0),(1,1,0),(1,2,0),(1,2,1)\}.
	\end{align*}
	 
	 When $(k,R_{3},S_{1})= (1,1,0)$ or $(1,2,0)$, since $\ord(a_{1,3})$ is odd and $\ord(a_{1})=\ord(b_{1})=0$, both $\ord(a_{1,3}b_{1})$ and $\ord(a_{2}a_{3})$ are odd. So $d(-a_{1,3}b_{1})=0$ and 
	\begin{align*}
		R_{3}=\begin{cases}
			2  &\text{if $R_{2}=1$},\\
			1  &\text{if $R_{2}\in [2-2e,0]^{E}$}.
		\end{cases}
	\end{align*}
	Hence
	\begin{align*}
		d[-a_{1,2}]=1-R_{2}\begin{cases}
			=0\le 2e-2=2e+S_{1}-R_{3}&\text{if $R_{2}=1$},\\
			\le 2e-1=2e+S_{1}-R_{3}&\text{if $R_{2}\in [2-2e,0]^{E}$}.
		\end{cases}
	\end{align*}
	So, in both cases, $d[-a_{1,2}]+d (-a_{1,3}b_{1})\le 2e+S_{1}-R_{3}$.
	
	When $(k,R_{3},S_{1})= (0,1,0)$ or $(1,2,1)$, let $N=\prec \eta\pi^{k}\succ$ with $\eta\in \mathcal{U}$ and $\eta\not=\varepsilon$.  Then $FM\cong W_{2}^{3}(\varepsilon\pi^{k})$ represents $[b_{1}]=[\eta\pi^{k}]$. If $d(-a_{1,3}b_{1})\ge 2e$, by Lemma \ref{lem:k=0}, $[a_{1},a_{2}]$ represents $[b_{1}]$. Assume $d(-a_{1,3}b_{1})\le 2e-1$. Since $R_{2}=1$ and $(R_{3},S_{1})\in \{(1,0),(2,1)\}$, we have
	 \begin{align*}
	 	d[-a_{1,2}]+d (-a_{1,3}b_{1}) \le (1-R_{2})+(2e-1)=2e-1=2e+S_{1}-R_{3}.
	 \end{align*}
	 Thus, in all cases, Theorem \ref{thm:beligeneral-1}(iii) is verified.
\end{proof}
To show Theorem \ref{thm:dyadic-ternary}, we need to derive the explicit structure from the invariants obtained in Lemmas \ref{lem:W13k} and \ref{lem:W23k}. The treatment is similar to that in \cite[\S 7]{He25}, but there are additional cases needed to be addressed. 

Unless otherwise stated, we assume that  
\begin{align}\label{H1h}
	FM\cong W_{\nu}^{3}(c)\quad\text{and}\quad M\not\cong H_{1,h}^{3}(c),
\end{align}
with $\nu\in \{1,2\}$, $c\in \mathcal{V}$ and $h\ge 1$. Then Lemmas \ref{lem:W13k} and \ref{lem:W23k} can be rephrased as follows.
\begin{lem}\label{lem:R2}
	Let $\nu\in \{1,2\}$, $k\in \{0,1\}$ and $\varepsilon\in \mathcal{U}$. Suppose that $M$ is ADC. Then $R_{1}=0$, $R_{2}\in [-2e,0]^{E}\cup \{1\}$ and $\alpha_{1}\in \{0,1\}$. Also,
	\begin{enumerate}[itemindent=-0.5em,label=\rm (\roman*)]	
		\item  if $FM\cong W_{\nu}^{3}(\varepsilon\pi^{k})$ and $R_{2}\in [-2e,0]^{E}$, then $R_{3}=k$;
		
		\item  if $FM\cong W_{2}^{3}(\varepsilon\pi^{k})$ and $R_{2}=1$, then $R_{3}=k+1$.
	\end{enumerate}
\end{lem}
\begin{lem}\label{lem:ADC-isometry-odd-R2}
	Let $M$ and $M^{\prime}$ be two ADC $\mathcal{O}_{F}$-lattices. Then  $M\cong M^{\prime}$ if and only if $FM\cong FM^{\prime} $ and $R_{2}(M)=R_{2}(M^{\prime})$.
\end{lem}
\begin{proof}
	Necessity is clear. To show the sufficiency, it suffices to verify conditions (i)-(iv) of \cite[Theorem 3.1]{beli_Anew_2010} for $M$ and $M^{\prime}$. Let $M^{\prime}\cong \prec b_{1},b_{2},b_{3}\succ$ relative to some good BONG. Write $S_{i}=R_{i}(M^{\prime})$ and $\beta_{i}=\alpha_{i}(M^{\prime})$.
	
	First, by Lemma \ref{lem:R2}, we have 
	\begin{align}\label{ADCcon}
		 R_{1}=0,\quad R_{2} \in [-2e,0]^{E}\cup \{1\},\quad\text{and}\quad \alpha_{1}\in \{0,1\}.
	\end{align}
	 If  $R_{2} \in [-2e,0]^{E}$, then by Lemma \ref{lem:R2}(i), $R_{3}\in \{0,1\}$. This combined with \eqref{ADCcon} agrees with the conditions \cite[(7.3), p.\hskip 0.1cm 1013]{He25} with $n=1$. So we are done by following the proof of \cite[Lemma 7.15]{He25} with $n=1$. 
	
 	 Let $FM\cong W_{2}^{3}(\varepsilon\pi^{k})$. Suppose that $R_{2}=1$. Then $S_{2}=1$ from the hypothesis. Hence, by Lemma \ref{lem:R2}(ii), $R_{1}=S_{1}$ and $R_{3}=S_{3}=k+1$. Thus condition (i) is fulfilled.
	
	Since $R_{2}-R_{1}=1$,  Proposition \ref{prop:alphaproperty}(iv) implies that $\alpha_{1}=1$. If $k=1$, then $R_{3}=R_{2}+1=2$. So, Proposition \ref{prop:alphaproperty}(iv) also implies that $ \alpha_{2}=1$. If $k=0$, then $R_{3}=R_{2}=1$. Since $\ord(a_{2}a_{3})$ is even, $d(-a_{2}a_{3})\ge 1=\alpha_{1}$. Hence $d[-a_{2}a_{3}]=\min\{d(-a_{2}a_{3}),\alpha_{1}\}=1$, so, by Proposition \ref{prop:alphaproperty}(ii), 
	\begin{align*}
		\alpha_{2}=\min\{\dfrac{R_{3}-R_{2}}{2}+e,R_{3}-R_{2}+d[-a_{2,3}]\}=\min\{e,1\}=1.
	\end{align*}
	 Similarly, $\beta_{1}=\beta_{2}=1$. Thus condition (ii) is fulfilled.
	
 	For $1\le i\le 2$, since $R_{i}=S_{i}$, $\ord(a_{1,i}b_{1,i})$ is even, and hence $d(a_{1,i}b_{1,i})\ge 1=\alpha_{i}$. Thus condition (iii) is fulfilled. For $k\in \{0,1\}$, we always have $\alpha_{1}+\alpha_{2}=2\le 2e$. Thus condition (iv) is trivially fulfilled.
\end{proof}

We extend \cite[Definition 7.16, Remark 7.17 and Lemma 7.20]{He25} to the case $n=1$, including $R_{2}(M)=1$.
\begin{defn}\label{defn:alpha1-ternary}
	Let $\nu\in \{1,2\}$, $r\in \{-1/2,0,\ldots,e\}$ and $c\in \mathcal{V}$. We denote by $M_{\nu,r}^{3}(c)$ the only ADC lattice $M$ with $FM\cong W_{\nu}^{3}(c)$ and $R_{2}(M)=-2r$.
\end{defn}
\begin{re}\label{re:alpha1-ternary}
	By Lemma \ref{lem:ADC-isometry-odd-R2}, such ternary lattice is unique up to isometry, if it exists.
	
	If a ternary $\mathcal{O}_{F}$-lattice $M$ is ADC, then we have  $R_{2}(M)\in [-2e,0]^{E}\cup \{1\}$, i.e., $R_{2}(M)=-2r$ for some $r\in \{-1/2,0,1,\ldots,e\}$. Hence $M\cong M_{\nu,r}^{3}(c)$, where $FM\cong W_{\nu}^{3}(c)$. So every ternary ADC lattice that is isometric to $M_{\nu,r}^{3}(c)$ for some $\nu\in \{1,2\}$, $r\in \{-1/2,0,1,\ldots,e\}$ and $c\in \mathcal{V}$, under the assumption \eqref{H1h}.
\end{re}
\begin{lem}\label{lem:undefined}
	Suppose that $M$ is ADC. 
	\begin{enumerate}[itemindent=-0.5em,label=\rm (\roman*)]	
		\item  If $FM\cong W_{2}^{3}(c)$ for some $c\in\mathcal{U}$, then $R_{2}\not=-2e$, equivalently, $M_{2,e}^{3}(c)$ is not defined.
		
		\item  If $FM\cong W_{1}^{3}(c)$ for some $c\in\mathcal{V}$, then $R_{2}\not=1$, equivalently, $M_{1,-1/2}^{3}(c)$ is not defined.
	\end{enumerate}
\end{lem}
\begin{proof}
	 These are clear from Lemmas \ref{lem:W23k} and \ref{lem:W13k}.
\end{proof}
\begin{lem}\label{lem:case-iii}
	Let $\nu\in \{1,2\}$, $c\in \mathcal{V}$ and $\omega\in \mathcal{V}$ with $d(\omega)<2e$. Then $M=N_{\nu}^{2}(\omega)\perp \langle \omega c\rangle$ is ADC and $R_{2}(M)=1-d(\omega)\in [2-2e,0]^{E}\cup \{1\}$. Also,
	 \begin{align*}
		M \cong\begin{cases}
			\prec (\omega^{\#})^{\nu-1},-(\omega^{\#})^{\nu-1}\omega\pi^{1-d(\omega)},\omega c\succ  &\text{if $1\le d(\omega)<2e$},\\
			  \prec \Delta^{\nu-1},-\Delta^{\nu-1}\omega,\omega c\succ  &\text{if $ d(\omega)=0$}.
		\end{cases} 
	\end{align*}
\end{lem}
\begin{proof}
 	Let $N_{\nu}^{2}(\omega)\cong \prec  a_{1},a_{2}\succ $ relative to a good BONG. Apply Proposition \ref{prop:maximallattices-binary-ternary}(ii). If $1\le d(\omega)<2e$, then $(a_{1},a_{2})=(1,-\omega\pi^{1-d(\omega)})$ or $(\omega^{\#},-\omega^{\#}\omega\pi^{1-d(\omega)})$, according as $\nu=1$ or $2$; if $d(\omega)=0$, let $\omega=\varepsilon\pi$ with $\varepsilon\in \mathcal{U}$. Then $(a_{1},a_{2})=(1,-\varepsilon\pi)$ or $(\Delta,-\Delta\varepsilon\pi)$, according as $\nu=1$ or $2$.

	Put $R_{i}:=R_{i}(N_{\nu}^{2}(\omega))$. Then $R_{1}=0$ and $R_{2}=1-d(\omega)$. Since $\omega,c\in \mathcal{V}$, both of $\ord(\omega)$ and $\ord(c)$ fall inside $\{0,1\}$. So if $a_{3}:=\omega c$ and $R_{3}:=\ord(a_{3})$, then $R_{3}=\ord(\omega)+\ord(c)\in \{0,1\}$ or $\{1,2\}$ according as $d(\omega)\ge 1$ or not. In the former case, $R_{3}\ge 0\ge 1-d(\omega)$; in the latter case, $R_{3}\ge 1=1-d(\omega)$. So, in both cases, $R_{3}\ge R_{2} $. Recall that $R_{3}\ge 0=R_{1}$. Combining these with \cite[Corollary 4.4(v)]{beli_integral_2003}, we see that
\begin{align*}
	M\cong \prec a_{1},a_{2}\succ\perp \prec a_{3}\succ \cong \prec a_{1},a_{2},a_{3}\succ
\end{align*}
relative to a good BONG and $R_{i}(M)=R_{i}$. Note that if $1\le d(\omega)<2e$, i.e., $\omega\in \mathcal{U}\backslash\{1,\Delta\} $, then $d(\omega)\in [1,2e-1]^{O}$. Hence $R_{2}(M)=R_{2}=1-d(\omega)\in [2-2e,0]^{E}\cup \{1\}$.

Write $\alpha_{1}=\alpha_{1}(M)$. Since $R_{2}-R_{1}=R_{2}>-2e$,  Proposition \ref{prop:alphaproperty}(iii) implies that $\alpha_{1}\ge 1$. In all cases, we have $-a_{1}a_{2}=\omega$ and
\begin{align*}
	\alpha_{1}\le R_{2}-R_{1}+d(-a_{1}a_{2})=(1-d(\omega))-0+d(\omega)=1.
\end{align*}
So $\alpha_{1}=1$.

With above discussion, $M$ is ADC by Lemmas \ref{lem:W13k}(ii) and \ref{lem:W23k}(ii)-(iii).
\end{proof}
\begin{lem}\label{lem:alpha1-ternary}
	Let $\nu\in \{1,2\}$, $r\in \{-1/2,0,\ldots,e\}$ and $c\in \mathcal{V}$. Then $M_{\nu,r}^{3}(c)$ is defined except for $(\nu,r,c)\in \{(2,e,\mathcal{U}),(1,-1/2,\mathcal{V})\}$.
	\begin{enumerate}[itemindent=-0.5em,label=\rm (\roman*)]	
		\item 	If $r=e$ and $(\nu,c)\not=(2,\mathcal{U})$, then  $M_{\nu,e}^{3}(c)\cong N_{\nu}^{3}(c)$.
		
		\item If $r=e-1$ and $(\nu,c)=(2,\mathcal{U})$, then $M_{2,e-1}^{3}(c)\cong N_{2}^{3}(c)$.
		
		\item If $0\le r\le e-1$, or $r=-1/2$ and $\nu=2$, then 
		$M_{\nu,r}^{3}(c)\cong N_{\nu^{\prime}}^{2}(\omega_{r})\perp \langle \omega_{r}c\rangle$, where $\omega_{r}\in \mathcal{V}$ is arbitrary such that $d(\omega_{r})=2r+1$ and $ \nu^{\prime}\in \{1,2\}$ satisfies $(-1)^{\nu^{\prime}}=(-1)^{\nu}(\omega_{r},c)_{\mathfrak{p}}$. 
	\end{enumerate}	 
\end{lem}
\begin{proof}
	 By Lemma \ref{lem:undefined}, $M_{2,e}^{3}(c)$ is undefined for every $c\in \mathcal{U}$ and $M_{1,-1/2}^{3}(c)$ is undefined for every $c\in \mathcal{V}$. Then, we will show that the lattice $M_{\nu,r}^{3}(c)$ is defined for the remaining cases, thereby proving the lemma.
	 
	 For (i) and (ii), since $N_{\nu}^{3}(c)$ is $\mathcal{O}_{F}$-maximal, by Lemma \ref{lem:ADC-sufficient-equiv}(i), it is ADC. We also have $FN_{\nu}^{3}(c)\cong W_{\nu}^{3}(c)$. If $(\nu,c)\not=(2,\mathcal{U})$, by \cite[Lemma 4.12(i) and (iii)]{He25}, $R_{2}(N_{\nu}^{3}(c))=-2e$. So, from Definition \ref{defn:alpha1-ternary}, $N_{\nu}^{3}(c)\cong M_{\nu,e}^{3}(c)$; if $(\nu,c)=(2,\mathcal{U})$, then, by \cite[Lemma 4.12(ii)]{He25},  $R_{2}(N_{2}^{3}(c))=2-2e$. Hence, from Definition \ref{defn:alpha1-ternary} again, $N_{2}^{3}(c)\cong M_{2,e-1}^{3}(c)$.
	 
	 For (iii), let $M=N_{\nu^{\prime}}^{2}(\omega_{r})\perp \langle\omega_{r}c \rangle$, with $0\le r\le e-1$, or $r=-1/2$ and $(\nu,c)\not=(1,\mathcal{V})$. Since $(c,\omega_{r})=(-1)^{\nu+\nu^{\prime}}$, by \cite[Lemma 4.4(ii)]{He25}, $W_{\nu}^{3}(c)$ represents $W_{\nu^{\prime}}^{2}(\omega_{r})$. Since $\det W_{\nu^{\prime}}^{2}(\omega_{r})\det W_{\nu}^{3}(c)=\omega_{r}c$, we see that $FM\cong W_{\nu^{\prime}}^{2}(\omega_{r})\perp [\omega_{r}c]\cong W_{\nu}^{3}(c)$. Also, by Lemma \ref{lem:case-iii}, $M$ is ADC and $R_{2}(M)=1-d(\omega_{r})=-2r$. Hence, from Definition \ref{defn:alpha1-ternary}, $M\cong M_{\nu,r}^{n+2}(c)$.
\end{proof}
\begin{cor}\label{cor:count-dyadic}
	Up to isometry, 
	\begin{enumerate}[itemindent=-0.5em,label=\rm (\roman*)]	
		\item  there are infinitely many ternary isotropic ADC lattices, of which $4(N\mathfrak{p})^{e}$ are $\mathcal{O}_{F}$-maximal. 
	
	\item there are $(4e+6)(N\mathfrak{p})^{e}$ ternary anisotropic ADC lattices, of which $4(N\mathfrak{p})^{e}$ are $\mathcal{O}_{F}$-maximal. 
\end{enumerate}
\end{cor}
\begin{proof}
	(i) By Theorem \ref{thm:dyadic-ternary}(i)(b), there are infinitely many isotropic ADC lattices. 
	
  	(ii) Suppose that $FM$ is anisotropic. By Proposition \ref{prop:nu-isotropic}(ii), we have $\nu=2$. 
	
	If $M$ is of the form (i) in Lemma \ref{lem:alpha1-ternary}, then $c\in \pi\mathcal{U}$ and thus the number of these lattices is given by $|\mathcal{U}|=2(N\mathfrak{p})^{e}$.
	
	If $M$ is of the form (iii), then $r\in \{-1/2,0,\ldots,e-1\}$ and $c\in \mathcal{V}$. So the number of these lattices is given by
	\begin{align*}
		  (e+1)|\mathcal{V}|=4(e+1)(N\mathfrak{p})^{e}.
	\end{align*}
	So, there are $(4e+6)(N\mathfrak{p})^{e}$ anisotropic ADC lattices in total.
	
	In (i) and (ii), the remaining assertions follow by Corollary \ref{cor:counting-maximal}.
\end{proof}
\begin{proof}[Proof of Theorem \ref{thm:dyadic-ternary}]
	This follows from Definition \ref{defn:alpha1-ternary}, Remark \ref{re:alpha1-ternary} and Lemma \ref{lem:alpha1-ternary}.
\end{proof}
 \begin{proof}[Proof of Theorem \ref{thm:primitive-dyadic}]
	(i) For necessity, since $M$ is non-primitive, $\mathfrak{n}(M)\subsetneq \mathcal{O}_{F}$. By Theorem \ref{thm:dyadic-binary} and Proposition \ref{prop:maximallattices-binary-ternary}(ii), $M\cong N_{2}^{2}(\Delta)=\prec \pi,-\Delta\pi^{1-2e}\succ$ or $\cong M_{2}^{2}(\Delta)=\prec \pi,-\Delta\pi^{3-2e}\succ$. Then $R_{1}=1$ and $R_{2}=1-2e$ or $3-2e$. So, by \cite[Lemma 2.1]{beli_integral_2003}, $\ord(\mathfrak{n}(M))=R_{1}=1$ and $\ord(\mathfrak{v}(M))=R_{1}+R_{2}=2-2e$ or $4-2e$. Clearly, $FM\cong W_{2}^{2}(\Delta)=[\pi,-\Delta\pi]$.
	
	For sufficiency, we have $R_{1}=\ord(\mathfrak{n}(M))=1$ and $R_{2}=\ord(\mathfrak{v}(M))-R_{1}=1-2e$ or $3-2e$. Since $FM\cong [\pi,-\Delta\pi]=W_{2}^{2}(\Delta)$, the ADC-ness of $M$ follows by \cite[Lemma 4.11(ii)]{He25} with $n=2$ and Lemma \ref{lem:binary-exceptions} with $\nu=2$ and Theorem \ref{thm:dyadic-binary}.
	
	(ii) By Lemmas \ref{lem:W13k} and \ref{lem:W23k}, we have $R_{1}=0$ and thus $\mathfrak{n}(M)=\mathfrak{p}^{R_{1}}=\mathcal{O}_{F}$. Therefore, $M$ is primitive.
\end{proof}
In the remainder of this section, we will determine the codeterminant set of ternary ADC lattices. To do so, we first characterize their integral spinor groups.
\begin{lem}\label{lem:R-theta-M-nur}
		Let $\nu\in \{1,2\}$, $r\in \{-1/2,0,\ldots,e\}$ and $c\in \mathcal{V}$. Suppose that $M=M_{\nu,r}^{3}(c)$ with $(\nu,r,c)\not\in \{(2,e,\mathcal{U}),(1,-1/2,\mathcal{V})\}$.
	\begin{enumerate}[itemindent=-0.5em,label=\rm (\roman*)]
		\item  If $r=e$ and $(\nu,c)=(1,\mathcal{U})$, then $R_{1}=R_{3}=0$, $R_{2}=-2e$, and $\theta(O^{+}(M))=\mathcal{O}_{F}^{\times}F^{\times 2}$.
		
		\item  If $r=e$ and $c\in \pi\mathcal{U}$, then  $R_{1}=0<R_{3}=1$, and $\theta(O^{+}(M))=F^{\times}$.
		
 		\item  If $0\le r\le e-1$ and $c\in \mathcal{U}$, then $R_{1}=R_{3}=0$, $R_{2}=-2r\in [2-2e,0]^{E}$ and $\theta(O^{+}(M))=F^{\times}$. 
		
		\item If $0\le r\le e-1$ and $c\in \pi\mathcal{U}$, then $R_{1}=0<R_{3}=1$, and $\theta(O^{+}(M))=F^{\times}$.
	
	 	\item  If $r=-1/2$, $\nu=2$ and $c\in \mathcal{U}$, then $R_{1}=0<R_{3}=1$, $R_{2}=1$, and $\theta(O^{+}(M))=F^{\times}$.
	 	
	 	\item  If $r=-1/2$, $\nu=2$ and $c\in \pi\mathcal{U}$, then $R_{1}=0<R_{3}=2$, $R_{2}=1$, and $\theta(O^{+}(M))=F^{\times}$.
	\end{enumerate}
\end{lem}
\begin{proof}
	In all cases, the $R_{i}$-invariants of $M$ are clear from Definition \ref{defn:alpha1-ternary}, Lemmas \ref{lem:W13k} and \ref{lem:W23k}. It remains to determine the spinor norm group $\theta(O^{+}(M))$.
	
	For (ii), (iv) and (v), by Proposition \ref{prop:property-A}, $M$ has property A, but does not have property B. So $\theta(O^{+}(M))=F^{\times}$.
	
	For (i) and (iii), since $R_{1}=R_{3}=0$, Proposition \ref{prop:property-A} implies that $\theta(O^{+}(M))=\mathcal{O}_{F}^{\times}F^{\times 2}$ or $F^{\times}$. Also, from Remark \ref{re:mathscrA}, $a_{i+1}/a_{i}\in \mathscr{A}$ for $1\le i\le 2$.
	
	For (i), by Lemma \ref{lem:alpha1-ternary}(i), $M=N_{1}^{3}(c)\cong \prec 1,-\pi^{-2e},c\succ$ with $c\in \mathcal{U}$. Then $R_{2}-R_{1}=-2e $ and $R_{3}-R_{2}=2e$ are even. Also,
	\begin{align*}
		&d(-a_{2}/a_{1})=d(1)=\infty>2e=e-(R_{2}-R_{1})/2
		\intertext{and}
		&d(-a_{3}/a_{2})=d(c)\ge 1>0=e-(R_{3}-R_{2})/2.
	\end{align*} 
	So, by \cite[Lemma 7.2(i)]{beli_integral_2003},  $G(a_{i+1}/a_{i})\subseteq \mathcal{O}_{F}^{\times}F^{\times 2}$ for $1\le i\le 2$. Also, $(R_{2}-R_{1})/2=-e\equiv e\pmod{2}$. Applying \cite[Theorem 3]{beli_integral_2003}, we deduce  $\theta(O^{+}(M))\subseteq \mathcal{O}_{F}^{\times}F^{\times 2}$, as desired.
 
  	For (iii), by \cite[Theorem 3]{beli_integral_2003}, it suffices to show $G(a_{j+1}/a_{j})\not\subseteq \mathcal{O}_{F}^{\times}F^{\times 2}$ for some $j$. By Lemmas \ref{lem:case-iii} and \ref{lem:alpha1-ternary}(iii), 
   \begin{align*}
   	M=  M_{\nu,r}^{3}(c)\cong \prec (\omega_{r}^{\#})^{\nu^{\prime}-1},-(\omega_{r}^{\#})^{\nu^{\prime}-1}\omega_{r}\pi^{-2r},\omega_{r} c\succ,
   \end{align*}
    where $0\le r\le e-1$, and $\omega_{r}\in \mathcal{V}$ such that $d(\omega_{r})=2r+1$, and $ \nu^{\prime}\in \{1,2\}$ satisfies $(-1)^{\nu^{\prime}}=(-1)^{\nu}(\omega_{r},c)_{\mathfrak{p}}$. . Then 
   \begin{align*}
   	d(-a_{2}/a_{1})=d(\omega_{r})=2r+1\le e+r= e-(-2r-0)/2=e-(R_{2}-R_{1})/2,
   \end{align*}
    so \cite[Lemma 7.2(i)]{beli_integral_2003} implies that $G(a_{2}/a_{1})\not\subseteq \mathcal{O}_{F}^{\times}F^{\times 2}$, as desired.  
   
    For (vi), since $R_{1}=0<R_{3}=2$, $M$ has property A and $R_{3}-R_{1}=2$. Hence, by \cite[Theorem 1]{beli_integral_2003}, we deduce 
    \begin{align*}
    	\theta(O^{+}(M))\supseteq (1+\mathfrak{p})F^{\times 2}=\mathcal{O}_{F}^{\times}F^{\times 2},
    \end{align*}
    and thus $\theta(O^{+}(M))=\mathcal{O}_{F}^{\times}F^{\times 2}$ or $F^{\times}$. However, $R_{2}-R_{1}=R_{3}-R_{2}=1$, so \cite[Lemma 7.2(i)]{beli_integral_2003} implies that $G(a_{i+1}/a_{i})\not\subseteq \mathcal{O}_{F}^{\times}F^{\times 2}$ for $1\le i\le 2$. Hence  $\theta(O^{+}(M))=F^{\times}$.
\end{proof}   
\begin{lem}\label{lem:R-theta-H-1h}
Suppose that $M=H_{1,h}^{3}(c)$ with $h\ge 1$ and $c\in \mathcal{V}$. Then $R_{1}=0$, $R_{2}=-2e$, $R_{3}=2h+\ord(c)$, and 
\begin{align*}
	\theta(O^{+}(M))=\begin{cases}
		\mathcal{O}_{F}^{\times}F^{\times 2} &\text{if $c\in \mathcal{U}$}, \\
		F^{\times}   &\text{if $c\in \pi\mathcal{U}$}.
	\end{cases}
\end{align*}
\end{lem}
\begin{proof}
	Let $c=\varepsilon\pi^{k}$ with $k\in \{0,1\}$.  Since $2h+k>-2e$, from definition and \cite[Corollary 4.4(i)]{beli_integral_2003}, we have
	\begin{align}\label{H1hc-BONG}
		M=H_{1,h}^{3}(c)=\mathbf{H}\perp \langle \varepsilon\pi^{2h+k}\rangle\cong \prec 1,-\pi^{-2e}\succ \perp \prec \varepsilon\pi^{2h+k}\succ\cong  \prec 1,-\pi^{-2e}, c\pi^{2h+k} \succ.
	\end{align}
	Clearly, $R_{1}=0$, $R_{2}=-2e$ and $R_{3}=2h+k$. We have
	\begin{align*}
		R_{2}-R_{1}=-2e\quad\text{and}\quad d(-a_{2}/a_{1})=d(1)=\infty,
	\end{align*}
	so \cite[Lemma 7.2(i)]{beli_integral_2003} implies that $G(a_{3}/a_{2})\subseteq \mathcal{O}_{F}^{\times }F^{\times 2}$. Also,
	\begin{align*}
		d(-a_{3}/a_{2})=d(\varepsilon\pi^{2h+2e+k})\ge 0>-h-k/2=e-(2h+k-(-2e))=e-(R_{3}-R_{2})/2.
	\end{align*}
	Hence \cite[Lemma 7.2(i)]{beli_integral_2003} implies that $G(a_{3}/a_{2})\subseteq \mathcal{O}_{F}^{\times  }F^{\times 2}$ or not, according as $R_{3}-R_{2}=2h+k+2e$ is even or not, i.e., $k=0$ or $1$. 
\end{proof}
\begin{lem}\label{lem:thetaOM-NE-dyadic}
	 Suppose that $M$ is ternary ADC. Then
	 \begin{enumerate}[itemindent=-0.5em,label=\rm (\roman*)]
	 	\item $\theta(O^{+}(M))\supseteq \mathcal{O}_{F}^{\times}F^{\times 2}$.
	 	
	 	\item If $\theta(O^{+}(M))\subseteq N(F(\sqrt{a})/F)$, then $a\in  \{1,\Delta\}$; if moreover, $\theta(O^{+}(M))=F^{\times}$, then $a=1$.
	 \end{enumerate}	
\end{lem}
\begin{proof}
	These two statements are clear from Remark \ref{re:dyadicACD-3}(iii), Lemmas \ref{lem:R-theta-M-nur} and \ref{lem:R-theta-H-1h}.
\end{proof}
\begin{proof}[Proof of Theorem \ref{thm:coM-dyadic}]
	 Assume that $M=M_{\nu,r}^{3}(c)$ with $\nu\in \{1,2\}$ satisfies one of the conditions: 
	 \begin{enumerate}
	 	\item[(a)] $r=e$ and $c=\pi\mathcal{U}$;
	 	
	 	\item[(b)]  $0\le r\le e-1$ and $c=\mathcal{V}$;
	 	
	 	\item[(c)] $r=-1/2$ and $c=\mathcal{V}$,
	 \end{enumerate}
	 or $M=H_{1,h}^{3}(c)$ with $h\ge 1$ and $c\in \pi\mathcal{U}$, then by Lemmas \ref{lem:R-theta-M-nur}(ii)-(vi) and  \ref{lem:R-theta-H-1h}, $\theta(O^{+}(M))=F^{\times}$ and so by Lemma \ref{lem:thetaOM-NE-dyadic}(ii), $\co(M)\subseteq \{1\}$. From Remark \ref{re:co-1-Delta}(i), $1 \in \co(M)$ if and only if $FM$ is isotropic. So $\co(M)=\{1\}$ or $\emptyset$, according as $FM$ is isotropic or not, which, by Proposition \ref{prop:nu-isotropic}, is equivalent to $\nu=1$ or $2$.
	
	 If $M=M_{\nu,r}^{3}(c)$ with  $(\nu,r,c)=(1,e,\mathcal{U})$ or $M=H_{1,h}^{3}(c)$ with $h\ge 1$ and $c\in  \mathcal{U}$, then by Lemmas \ref{lem:R-theta-M-nur}(i), \ref{lem:R-theta-H-1h} and \ref{lem:thetaOM-NE-dyadic}(ii), $\theta(O^{+}(M))=\mathcal{O}_{F}^{\times}F^{\times 2}$ and so $\co(M)\subseteq \{1,\Delta\}$. Since $\nu=1$, $FM$ is isotropic, so from Remark \ref{re:co-1-Delta}(i), $1\in \co(M)$. In the former case, $R_{1}=R_{3}=0$ and $R_{2}-R_{1}=$, so it follows from Lemma \ref{lem:co-dyadic}(ii) that $\Delta\not\in \co(M)$. Thus $\co(M)=\{1\}$. In the latter case, $R_{1}=0<R_{3}=2h$ and
	 \begin{align*}
	 	d(-a_{1}a_{2}\Delta)+d(-a_{2}a_{3}\Delta)+d(\Delta)=2d(\Delta)+d(c\Delta)\ge 4e+1>4e,
	 \end{align*}
	 so it follows from Lemma \ref{lem:co-dyadic}(i) that $\Delta\in \co(M)$. Thus $\co(M)=\{1,\Delta\}$.
\end{proof}
 
\section{Proof of Main Results}\label{sec:proof-main-results}
\begin{proof}[Proof of Theorems \ref{thm:locallyn-ADC-binary} and \ref{thm:locallyn-ADC-ternary}]
	
	If $\rank\, M=1$, this is clear from the proof \cite[Proposition 4.15]{He25} with $n=1$.
	
	 If $ \rank\,M=2 $, then the theorem follows by Theorem  \ref{thm:nondyadic-ADC-binary} if $\mathfrak{p}$ is non-dyadic, and Theorem \ref{thm:dyadic-binary} if $\mathfrak{p}$ is dyadic.
	
     If $\rank\, M=3$, then the theorem follows by Theorem  \ref{thm:nondyadic-ADC-ternary} and Remark \ref{re:nondyadicACD-3} if $\mathfrak{p}$ is non-dyadic, and by Theorem \ref{thm:dyadic-ternary} and Remark \ref{re:dyadicACD-3} if $\mathfrak{p}$ is dyadic.
\end{proof} 
 \begin{proof}[Proof of Theorem \ref{thm:globally ADC-rank2}]
 	(i) If $M$ is unary, by \cite[65:15]{omeara_quadratic_1963}, $h(M)=1$ and thus it is regular. So the equivalent condition holds by Theorem \ref{thm:locallyn-ADC-binary}(i) and \cite[\S82K]{omeara_quadratic_1963}.
 	
 	(ii) Sufficiency is clear. For necessity, if $M$ is binary and $FM$ is isotropic, then for each $\mathfrak{p}\in \Omega_{F}$, $FM_{\mathfrak{p}}$ is isotropic and thus is universal. Hence $M $ is locally universal. Then by regularity, $M$ is globally universal, so by \cite[Corollary 3.6]{xu_indefinite_2020}, $M\cong \mathbf{H}$.
 	
 	(iii) Since $M$ is binary and $FM$ is anisotropic, by \cite[Theorem A3]{chan-icaza}, $M$ is regular if and only if $h(M)=1$. The assertion follows by \eqref{equiv:uni-adc-regular}.
 \end{proof}
  \begin{proof}[Proof of Corollary \ref{cor:primitive-finite}]
 	If a binary $\mathcal{O}_{F}$-lattice $M$ is primitive ADC, then by Theorem \ref{thm:globally ADC-rank2}(iii), $h(M)=1$. Note that $m(M)\le h(M)/2$ and thus the argument in \cite{peter_class-num-one_1980} still applies. So, the corollary follows from \cite[Theorem]{peter_class-num-one_1980}.
 \end{proof}
 \begin{proof}[Proof of Theorem \ref{thm:scale-ADC}]
 	For necessity, by \cite[Lemma 2.1(b)]{clark_ADC-II-2014}, $M$ is ADC. Observe from Theorem \ref{thm:locallyn-ADC-binary}(ii) and (iii) that $\mathfrak{n}(L)\subseteq \mathfrak{p}$ for any ADC $\mathcal{O}_{F_{\mathfrak{p}}}$-lattice $L$. Assume that $\ord_{\mathfrak{q}}(a)>1$ for some $\mathfrak{q}\mid a\mathcal{O}_{F}$. Then  $\mathfrak{q}^{2}\mathfrak{n}(M_{\mathfrak{q}})=\mathfrak{n}(M_{\mathfrak{q}}^{(a)})\subseteq \mathfrak{q}$. Hence $\mathfrak{n}(M_{\mathfrak{q}})\subseteq \mathfrak{q}^{-1}$, This contradicts that $M$ is integral. So $\ord_{\mathfrak{p}}(a)=1$ for all $\mathfrak{p}\mid a\mathcal{O}_{F}$. Clearly, $M_{\mathfrak{p}}^{(a)}$ is non-primitive. Applying Theorems \ref{thm:primitive-nondyadic}(i) and \ref{thm:primitive-dyadic}(i) to $M_{\mathfrak{p}}^{(a)}$ and the properties of scaling \cite[\S 82J]{omeara_quadratic_1963}, one can obtain the remaining conditions.
 	
 	For sufficiency, since $M$ is ADC, so is $M_{\mathfrak{p}}$ for all $\mathfrak{p}$. If $\mathfrak{p}\nmid a\mathcal{O}_{F}$, $M_{\mathfrak{p}}^{(a)}=M_{\mathfrak{p}}$ is ADC. Assume $\mathfrak{p}\mid a\mathcal{O}_{F}$. Since $\ord_{\mathfrak{p}}(a)=1$, if $M_{\mathfrak{p}}$ satisfies the hypothesis, then $M_{\mathfrak{p}}^{(a)}$ satisfies the conditions of Theorems \ref{thm:primitive-nondyadic}(i) and \ref{thm:primitive-dyadic}(i). So $M_{\mathfrak{p}}^{(a)}$ is ADC. Recall that the regularity is preserved by scaling the lattices. So, by \eqref{equiv:uni-adc-regular}, $M^{(a)}$ is ADC.
 \end{proof}
 \begin{proof}[Proof of Corollary \ref{cor:imprimitive-infinite}]
 	Assume that $M$ is a primitive binary ADC $\mathcal{O}_{F}$-lattice with $c=-\det FM\not=1$. By global square theorem,  there are infinitely many non-dyadic primes $\mathfrak{q}\nmid  c\mathcal{O}_{F}$ such that $c_{\mathfrak{q}}$ is a non-square, i.e., $c_{\mathfrak{q}}\in \Delta_{\mathfrak{q}} F_{\mathfrak{q}}^{\times 2}$. Let $\mathfrak{q}_{i}$ be such prime. Since $\mathcal{O}_{F}$ is PID, we may take $a_{i}\in \mathcal{O}_{F}^{+}$ such that $a_{i}\mathcal{O}_{F}=\mathfrak{q}_{1}\cdots \mathfrak{q}_{i}$, where $i=1,2,\ldots$. For each $i$, clearly, $\ord_{\mathfrak{q}_{i}}(a_{i})=1$. Since $M_{\mathfrak{q}_{i}}$ is primitive and $c_{\mathfrak{q}_{i}}\in \mathcal{O}_{F_{\mathfrak{q}_{i}}}^{\times}$,  $\mathfrak{n}(M_{\mathfrak{q}_{i}})=\mathfrak{v}(M_{\mathfrak{q}_{i}})=\mathcal{O}_{F_{\mathfrak{q}_{i}}} $. This combined with $c_{\mathfrak{q}_{i}}\in \Delta_{\mathfrak{q}_{i}}F_{\mathfrak{q}_{i}}^{\times 2}$ shows that $FM_{\mathfrak{q}_{i}}\cong [1,-\Delta_{\mathfrak{q}_{i}}]$. So, by Theorem \ref{thm:scale-ADC}, $M^{(a_{i})}$ is non-primitive ADC for $i=1,2,\ldots$, as required.
 \end{proof}
 \begin{thm}[{\cite{beli_talk_2025}}] \label{thm:coM-Beli}
 	 Let $M$ be a ternary $\mathcal{O}_{F}$-lattice over an algebraic number field $F$. Let $a\in F^{\times}/F^{\times 2}$ and $a\not=1$. Then $a\in \co(M)$ if and only if $a_{\mathfrak{p}}\in \co(M_{\mathfrak{p}})$ for all $\mathfrak{p}\in \Omega_{F}$.
 \end{thm}
 \begin{proof}[Proof of Theorem \ref{thm:coM-global}]
 	By Theorem \ref{thm:coM-Beli}, $a\in \co(M)$ if and only if $a_{\mathfrak{p}}\in \co(M_{\mathfrak{p}}) $. 
 	
 	Suppose that $\mathfrak{p}$ is archimedean. If $\mathfrak{p}$ is complex, then $a_{\mathfrak{p}}\in \co(M_{\mathfrak{p}})$ holds trivially. If $\mathfrak{p}$ is real, then $a_{\mathfrak{p}}\in \co(M_{\mathfrak{p}})$ is equivalent to condition (i), from Remark \ref{re:co-1-Delta}(ii).
 	
 	For non-archimedean prime $\mathfrak{p}$, it suffices to determine the form of $M_{\mathfrak{p}}$ for applying Theorems \ref{thm:coM-nondyadic} and \ref{thm:coM-dyadic}. 
 	
 	Suppose that $\mathfrak{p}$ is non-dyadic. Assume that $FM_{\mathfrak{p}}$ is isotropic and $\mathfrak{v}(M_{\mathfrak{p}})=\mathfrak{p}^{k_{\mathfrak{p}}}\subseteq \mathcal{O}_{F_{\mathfrak{p}}}$. From Remark \ref{re:nondyadicACD-3}(ii) and Proposition \ref{prop:maximallattices-binary-ternary}(i), if $k_{\mathfrak{p}}$ is even, then $M_{\mathfrak{p}}\cong N_{1}^{3}(c_{\mathfrak{p}})$ or $H_{1,h_{\mathfrak{p}}}^{3}(c_{\mathfrak{p}})$ with $h_{\mathfrak{p}}\ge 1$ and $c_{\mathfrak{p}}\in \mathcal{U}_{\mathfrak{p}}$, and so $\co(M_{\mathfrak{p}})=\{1,\Delta_{\mathfrak{p}}\}$; if $k_{\mathfrak{p}}$ is odd, then $M_{\mathfrak{p}}\cong N_{1}^{3}(c_{\mathfrak{p}})$ or $H_{1,h_{\mathfrak{p}}}^{3}(c_{\mathfrak{p}})$ with $h_{\mathfrak{p}}\ge 1$ and $c_{\mathfrak{p}}\in \pi_{\mathfrak{p}} \mathcal{U}_{\mathfrak{p}}$. Hence $\co(M_{\mathfrak{p}})=\{1\}$, as desired.
 	
 	Suppose that $\mathfrak{p}$ is dyadic and $FM_{\mathfrak{p}}$ is isotropic. Assume that   $\mathfrak{v}(M_{\mathfrak{p}})=\mathfrak{p}^{k_{\mathfrak{p}}}\subseteq \mathfrak{p}^{2}$. From Remark \ref{re:dyadicACD-3}(ii), if   $k_{\mathfrak{p}}$ is even, then $M_{\mathfrak{p}}\cong H_{1,h_{\mathfrak{p}}}^{3}(c_{\mathfrak{p}})$ with $h_{\mathfrak{p}}\ge 1$ and $c_{\mathfrak{p}}\in \mathcal{U}_{\mathfrak{p}}$, and thus $\co(M_{\mathfrak{p}})=\{1,\Delta_{\mathfrak{p}}\}$; if $k_{\mathfrak{p}}$ is odd, then $M_{\mathfrak{p}}\cong H_{1,h_{\mathfrak{p}}}^{3}(c_{\mathfrak{p}})$ with $h_{\mathfrak{p}}\ge 1$ and $c_{\mathfrak{p}}\in \pi_{\mathfrak{p}} \mathcal{U}_{\mathfrak{p}}$, and thus $\co(M_{\mathfrak{p}})=\{1\}$.
 	
 	Assume that $  2\mathcal{O}_{F_{\mathfrak{p}}}\subseteq 2\mathfrak{s}(M_{\mathfrak{p}})\subseteq\mathcal{O}_{F_{\mathfrak{p}}} $. Then from Remark \ref{re:r-sM-relation}, $M_{\mathfrak{p}}\cong M_{1,r_{\mathfrak{p}}}^{3}(c_{\mathfrak{p}}) $ with $0\le r_{\mathfrak{p}}\le e_{\mathfrak{p}}$ and $c_{\mathfrak{p}}\in \mathcal{V}$. Hence $\co(M_{\mathfrak{p}})=\{1\}$.
 \end{proof}
 \begin{proof}[Proof of Theorem \ref{thm:coM-global-empty}]
 	Assume that there exists $\mathfrak{p}\in \Omega_{F}\backslash \infty_{F}$ such that $FM_{\mathfrak{p}}$ is anisotropic. If $\mathfrak{p}$ is non-dyadic, then from Remark \ref{re:nondyadicACD-3}(ii), $M_{\mathfrak{p}}=N_{2}^{3}(c)$ for some $c\in \mathcal{V}_{\mathfrak{p}}$; if $\mathfrak{p}$ is dyadic, then from Remark \ref{re:dyadicACD-3}(iii), $M_{\mathfrak{p}}=M_{2,r}^{3}(c)$ for some $r\in \{-1/2,0,\ldots,e\}$ and $c\in \mathcal{V}_{\mathfrak{p}}$. By Lemmas \ref{lem:co-nondyadic} and \ref{lem:co-dyadic}, we have $\co(M_{\mathfrak{p}})=\emptyset$ in both cases. Thus $ \co(M)=\emptyset$. 
 	
 	Now, $\spn(M)$ represents all $b\in \mathcal{O}_{F}$ for which $\gen(M)$ represents $b$. Since $M$ is indefinite, by  \cite[104:5]{omeara_quadratic_1963}, $M$ represents all these $b$. Therefore, $M$ is regular. So, by \eqref{equiv:uni-adc-regular}, if $M$ is locally ADC, then it is globally ADC.
 \end{proof}
  \begin{proof}[Proof of Theorem \ref{thm:thetaO+M}]
 	The theorem follows from  Lemma \ref{lem:thetaOM-NE-non-dyadic}(iii)(a) and \ref{lem:thetaOM-NE-dyadic}(i).
 \end{proof}
 \begin{lem}\label{lem:odd-class-number}
 	Let $M$ be a ternary $\mathcal{O}_{F}$-lattice over an algebraic number field $F$ with odd class number $h(F) $. If $\theta(O^{+}(M_{\mathfrak{p}}))\supseteq \mathcal{O}_{F_{\mathfrak{p}}}^{\times}F_{\mathfrak{p}}^{\times 2}$ for all $\mathfrak{p}\in \Omega_{F}\backslash \infty_{F}$, then $g(M)=1$.
 \end{lem}
 \begin{proof}
 	As $\rank\,M=3$, the proper spinor genus of $M$ agrees with its spinor genus. The proof is straightforward from the argument in \cite[Proposition 3.2]{xu_indefinite_2020}.
 \end{proof}
 \begin{proof}[Proof of Theorem \ref{thm:gM-hM=1}]
 	 By Theorem \ref{thm:thetaO+M}, we have $\theta(O^{+}(M_{\mathfrak{p}}))\supseteq \mathcal{O}_{F_{\mathfrak{p}}}^{\times }F_{\mathfrak{p}}^{\times 2}$ for all $\mathfrak{p}\in \Omega_{F}\backslash \infty_{F}$. Then the theorem follows by Lemma \ref{lem:odd-class-number} and \cite[104:5]{omeara_quadratic_1963}.
 \end{proof}
 \begin{proof}[Proof of Theorem \ref{thm:count-sol-23}]
 	(i) This follows from Theorem \ref{thm:locallyn-ADC-binary} and \cite[(4.4) in Remark 4.3]{He25}.
 	
 	(ii) This follows from Corollaries \ref{cor:count-nondyadic} and \ref{cor:count-dyadic}.
 \end{proof}
 
\section{Reformulation of the formulas for local densities and masses}\label{sec:mass}
  Throughout this section, we assume that $F$ is a non-archimedean local field or an algebraic number field, and let $L$ be an $\mathcal{O}_{F}$-lattice and $V=FL$. 

As in \cite{korner_classnum_1981,pfeuffer_binarer_1978}, we unify the invariants in non-dyadic and dyadic cases for better discussion. First, recall that every binary lattice over non-dyadic fields $F$ always has a Jordan splitting $\mathcal{O}_{F}x_{1}\perp \mathcal{O}_{F}x_{2}\cong \langle a_{1},a_{2}\rangle$ with $a_{1}=\ord(Q(x_{1}))$, $a_{2}=\ord(Q(x_{2}))$ and $\ord(a_{1})\le \ord(a_{2})$.
\begin{defn}\label{defn:non-dyadic}
Let $F$ be a non-dyadic local field. We say that a binary $\mathcal{O}_{F}$-lattice $L$ has a good BONG if $L\cong \langle a_{1},a_{2}\rangle$ with $\ord(a_{1})\le \ord(a_{2})$, denoted by $L\cong \prec a_{1},a_{2}\succ$. 
\end{defn}
 As dyadic cases, we define the function $d$ from $F^{\times}/F^{\times 2}$ to $\mathbb{N}\cup \{\infty\}$ by $d(c):=\ord(c^{-1}\mathfrak{d}(c))$ for $c\in F^{\times}$. From \cite[\S 63A]{omeara_quadratic_1963}, $d(F^{\times})=\{0,\infty\}$. Also, $d(c)=\infty$ if and only if $c\in F^{\times 2}$.

Now, assuming that $F$ is a non-archimedean local field and $L\cong \prec a_{1},a_{2}\succ $ relative to a good BONG, we define the invariants $a(L)$, $R_{1}(L)$, $R_{2}(L)$, $R(L)$ and $\alpha(L)$ of $L$ as follows,
\begin{align}
	\begin{split}\label{a-Ri}
		 a(L)&:= c^{-2}dL\in F^{\times}/\mathcal{O}_{F}^{\times 2},\quad\text{with}\quad \mathfrak{n}(L)=c\mathcal{O}_{F},\\
		 R_{1}(L)&:=\ord(a_{1}),\quad R_{2}(L)=\ord(a_{2}), \\
      \end{split}\\
		 \intertext{and}
	 \begin{split} \label{R-alpha}
	 	 R(L)&:= \ord(a(L))=\ord(\mathfrak{v}(L))-2\ord(\mathfrak{n}(L)),\\
	 \alpha(L)&:= \min\{R(L)/2+e,R(L)+d(-a(L))\}. 
	\end{split}
\end{align} 
\begin{prop}\label{prop:a-R-alpha}
	Suppose that $F$ is non-dyadic or dyadic, and $L\cong \prec a_{1},a_{2}\succ$ relative to a good BONG. Put $a=a(L)$,  $R_{i}=R_{i}(L)=\ord(a_{i})$, $R=R(L)$ and $\alpha=\alpha(L) $. We have
 \begin{enumerate}[itemindent=-0.5em,label=\rm (\roman*)]
	 \item $a=a_{2}/a_{1}$, $R=R_{2}-R_{1}$, $\alpha=\min\{(R_{2}-R_{1})/2+e,R_{2}-R_{1}+d(-a)\}$.
	
	\item $dV\in aF^{\times 2}$,  $ \ord(dV)\equiv R\pmod{2}$,  $ \ord(\mathfrak{d}(-dV)/dV)=d(-a)$ and
	\begin{align*}
		\ord(\mathfrak{d}(-dL))= \ord(\mathfrak{v}(L))+d(-a).
	\end{align*}
	
	 \item   $\ord(\mathfrak{n}(L))=R_{1}$,  $\ord(\mathfrak{v}(L))=R_{1}+R_{2}$, $\ord(\mathfrak{s}(L))=\min\{(R_{1}+R_{2})/2,R_{1}\}$, and if $F$ is dyadic, then $\ord(\mathfrak{w}(L))=\min\{R_{2}-R_{1}+e,R_{2}-R_{1}+\alpha\}$.
\end{enumerate}
\end{prop}
\begin{re}
	 When $F$ is dyadic and $L$ is binary, the invariants $R(L)$ and $\alpha(L)$ also agree with $R_{2}(L)-R_{1}(L)$ and $\alpha_{1}(L)$ introduced in \cite[\S2]{beli_Anew_2010}, respectively. Therefore, the properties stated in Section \ref{sec:BONGs} are applicable to $R(L)$ and $\alpha(L)$ in dyadic cases.
\end{re}
\begin{proof}
	(i) From definition and \cite[\S 3.1]{beli_integral_2003}, $\mathfrak{n}(L)=a_{1}\mathcal{O}_{F}$ and $dL=a_{1}a_{2}$. It follows that $a=a_{2}/a_{1}$ and thus the last two equalities are straightforward.
	
	(ii) Note that $dV=a_{1}a_{2}=a$ in $F^{\times 2}$, and thus $\ord(dV)$ and $R=\ord(a)$ have the same parity. The remaining equalities follow from the definition of the function $d$ and the property of quadratic defects \cite[\S 63A]{omeara_quadratic_1963}.
	
	(iii) When $F$ is non-dyadic, i.e., $e=0$, we have $R_{2}\ge R_{1}$ and thus $(R_{1}+R_{2})/2\ge R_{1}$. Hence $\mathfrak{n}(L)=\mathfrak{s}(L)=\mathfrak{p}^{R_{1}}$. Clearly, $\mathfrak{v}(L)=\mathfrak{p}^{R_{1}+R_{2}}$.
	
	When $F$ is dyadic, the equalities follow from the definition of BONG, \cite[Lemma 2.1 and Corollary 4.4(iv)]{beli_integral_2003} and \cite[Lemma 2.14]{beli_Anew_2010}.
\end{proof}

In the rest of this section, we assume that $F$ is an algebraic number field. For a finite extension $K/F$, we write $D_{K/F}$ for its discriminant and put $D_{F}:=D_{F/F}$ when $K=F$. Also, for any fractional ideal $\mathfrak{a}$ of $F$, we denote by $N(\mathfrak{a})$ the absolute norm of $\mathfrak{a}$. Define the mass $m(L)$ of $L$ by
\begin{align*}
	m(L):=\sum_{i=1}^{h(L)} \dfrac{1}{|O(L_{i})|},
\end{align*}
where $L_{i}$'s form a representative of the isometry classes in $\gen(L)$. Clearly, $h(L)=1$ if and only if $m(L)|O(L)|=1$. Hence one can refine the class number one condition by using the mass formula.

Following \cite{korner_classnum_1981}, we assume that $\dim V=2$ and let $F^{\prime}=F(\sqrt{-dV})$. Note that $V\not=F^{\prime}$ if and only if $-dV\in F^{\times 2}$. For $\mathfrak{p}\in \Omega_{F}\backslash \infty_{F}$,  we put $\chi(\mathfrak{p})=1$ if $-dV\not\in F^{\times 2}$, otherwise,
\begin{align*}
	\chi(\mathfrak{p})=
	\begin{cases}
		1 &\text{if  $\mathfrak{p}$ splits in $F^{\prime}$},  \\
		0 &\text{if  $\mathfrak{p}$ ramifies in $F^{\prime}$},  \\
		-1  &\text{if  $\mathfrak{p}$ is inert in $F^{\prime}$}.
	\end{cases}
\end{align*}
We also denote by $r(F)$ the number of the primes $\mathfrak{p}\in \Omega_{F}\backslash \infty_{F}$ with $\chi(\mathfrak{p})=0$.

As discussed in \cite[\S 1]{korner_classnum_1981}, the quadratic space $V$ can be viewed as an $F$-algebra that possesses an $F $-basis $1,v$, where  $1$ is the identity of $F$ and $v^{2}=-dV$, and $F$ is also a subalgebra of $V$ after embedding.
We denote the non-trivial automorphism $\sigma$ of $F$-algebra $V$ by $(a+bv)^{\sigma}=a-bv$. Then the norm map and the trace map from $V$ to $F$ are given by $Q(x)=xx^{\sigma}$ and $\tr(x)=x+x^{\sigma}$, $x\in V$. We define the \textit{maximal order} $\mathcal{O}_{V}$ of the algebra $V$ by the set $\{x\in V\mid Q(x)\in \mathcal{O}_{F}\;\text{and}\; \tr(x)\in \mathcal{O}_{F} \}$. Then
\begin{align*}
	V=\begin{cases}
		F^{\prime}   &\text{if $-dV\not\in F^{\times 2}$},\\
		F\oplus F    &\text{if $-dV\in F^{\times 2}$},
  	\end{cases}\quad \text{and}\quad	\mathcal{O}_{V}=\begin{cases}
  	\mathcal{O}_{F^{\prime}}   &\text{if $-dV\not\in F^{\times 2}$},\\
  	\mathcal{O}_{F}\oplus \mathcal{O}_{F}    &\text{if $-dV\in F^{\times 2}$}.
  	\end{cases}
\end{align*}
(cf. \cite[\S2]{korner_classnum_1981}). We say that an $\mathcal{O}_{F}$-lattice $M$ is an \textit{$\mathcal{O}_{F}$-order} if it is a subring of $\mathcal{O}_{V}$ with $1\in M$. For any $\mathcal{O}_{F}$-lattice $L$ on $V$, we define the set $M_{L}$ by
\begin{align*}
	 M_{L}:=\{x\in V\mid xL\subseteq L\}.
\end{align*}
In fact, it is an $\mathcal{O}_{F}$-order. And we say that $L$ is an \textit{ideal} of an $\mathcal{O}_{F}$-order $M$ if $M_{L}\supseteq M$, and a \textit{proper ideal} if $M_{L}=M$. Note that $D_{F^{\prime}/F}\mathcal{O}_{F}=\mathcal{O}_{F}$ when $-dV\in F^{\times 2}$. Hence $\mathfrak{v}(\mathcal{O}_{V})=4^{-1}D_{F^{\prime}/F}\mathcal{O}_{F}$ from \cite[(11)]{korner_classnum_1981}, regardless of whether $-dV\in F^{\times 2}$ or not. Applying \cite[Lemma 1]{korner_classnum_1981} with $M=M_{L}$, we have the equalities for the integral ideal $\mathfrak{c}(M_{L})$ as defined in \cite[(8)]{korner_classnum_1981}:
\begin{align*}
	\mathfrak{c}^{2}(M_{L})=4D_{F^{\prime}/F}^{-1}\mathfrak{v}(M)=4D_{F^{\prime}/F}^{-1}\mathfrak{v}(L) \mathfrak{n}(L)^{-2}.
\end{align*}
 Motivated by the equality, we define the quantity
 \begin{align}\label{cL}
 	\mathfrak{c}^{2}(L):= 4D_{F^{\prime}/F}^{-1}\mathfrak{v}(L) \mathfrak{n}(L)^{-2}.
 \end{align}
 
Based on \cite[Theorem 2]{korner_classnum_1981}, we present a mass formula in terms of invariants $a(L)$, $R(L)$ and $\alpha(L)$ (cf. \eqref{a-Ri} and \eqref{R-alpha}).
 \begin{thm}\label{thm:mass-1}
	 Suppose that $F$ is a totally real number field and $L$ is a positive definite binary $\mathcal{O}_{F}$-lattice. Then
	\begin{align*}
		m(L)=\dfrac{N(\mathfrak{c}(L))h(F^{\prime})}{2^{r(F)+1}[\mathcal{O}_{F^{\prime}}^{\times}:\mathcal{O}_{F}^{\times}]h(F)}\prod_{\mathfrak{p}\mid \mathfrak{c}(L)}	w_{\mathfrak{p}}(L),
	\end{align*}
	where   
	\begin{align}\label{uvw}
		w_{\mathfrak{p}}(L)=\dfrac{u_{\mathfrak{p}} (L)}{\upsilon_{\mathfrak{p}}(L)}\left( 1-\dfrac{\chi(\mathfrak{p})}{N(\mathfrak{p})} \right)
	\end{align}
	is given as follows: 
	\begin{align*}
		&u_{\mathfrak{p}}(L)=
		\begin{cases}
			2  &\text{if $\chi(\mathfrak{p})=0$, and either $2\mid R_{\mathfrak{p}}(L)$ and $\alpha_{\mathfrak{p}}(L)+d_{\mathfrak{p}}(-a_{\mathfrak{p}}(L))>2e_{\mathfrak{p}}$},  \\
			&\text{or $2\nmid R_{\mathfrak{p}}(L)$ and $R_{\mathfrak{p}}(L)>2e_{\mathfrak{p}}$},  \\
			1  &\text{otherwise},
		\end{cases}	\\
		&\psi_{\mathfrak{p}}(L)=\begin{cases}
			R_{\mathfrak{p}}(L)      &\text{if $R_{\mathfrak{p}}(L)$ is odd}, \\
			\alpha_{\mathfrak{p}}(L) &\text{if $R_{\mathfrak{p}}(L)$ is even},
		\end{cases}\quad\text{and}\quad
		\upsilon_{\mathfrak{p}}(L)=\begin{cases}
			N(\mathfrak{p})^{\lfloor \psi_{\mathfrak{p}}(L)/2 \rfloor} &\text{if $\psi_{\mathfrak{p}}(L)\le 2e_{\mathfrak{p}}$}, \\
			2N(\mathfrak{p})^{e_{\mathfrak{p}}}&\text{if $\psi_{\mathfrak{p}}(L)>2e_{\mathfrak{p}}$}.
		\end{cases} 
	\end{align*}
\end{thm}
\begin{proof}
Let $M=M_{L}$ be an $\mathcal{O}_{F}$-order associated with $L$. Put $E(M)=U(M)\cap V_{1} $ and $E(\mathcal{O}_{V})=U(\mathcal{O}_{V})\cap V_{1}$, where $U(M)$ (resp. $U(\mathcal{O}_{V})$) is the unit group of $M$ (resp. $\mathcal{O}_{V}$) and $V_{1}=\{x\in V\mid Q(x)=1\}$ (cf. \cite[\S 2]{korner_classnum_1981}).  

Since $F$ is totally real and $L$ is positive definite, by substituting the formulas
 \begin{align*}
 	\dfrac{h^{+}(\mathcal{O}_{V})}{[E(\mathcal{O}_{V}):E(M)]}=\dfrac{[E(M):1]h(F^{\prime})}{2^{r(F)}[\mathcal{O}_{F^{\prime}}^{\times}:\mathcal{O}_{F}^{\times}]h(F)}\quad\text{and}\quad h^{+}(L)=2[E(M):1]m(L)
 \end{align*}
  from \cite[Remarks (c) and (36)]{korner_classnum_1981} into \cite[Theorem 2]{korner_classnum_1981}, we obtain
  \begin{align*}
  	2[E(M):1]m(L)=\dfrac{[E(M):1]h(F^{\prime})N(\mathfrak{c}(M))}{2^{r(F)}[\mathcal{O}_{F^{\prime}}^{\times},\mathcal{O}_{F}^{\times}]h(F)}\prod_{\mathfrak{p}\mid \mathfrak{c}(M)}\dfrac{e(M,\mathfrak{p})}{a(M,\mathfrak{p})}\left(\ 1-\dfrac{\chi(\mathfrak{p})}{N(\mathfrak{p})}\right).
  \end{align*}
 Dividing both sides by $2[E(M):1]$ yields the desired mass formula under our notation, provided that  $e(M,\mathfrak{p})=u_{\mathfrak{p}}(L)$ and $a(M,\mathfrak{p})=v_{\mathfrak{p}}(L)$. 
   
   Consider the terms $\psi(M,\mathfrak{p})$, $a(M,\mathfrak{p})$ and $e(M,\mathfrak{p})$ defined in \cite[(32), (35) and Theorem 2]{korner_classnum_1981}. From \cite[Lemma 1]{korner_classnum_1981} and the definition of $R_{\mathfrak{p}}(L)$ and $a_{\mathfrak{p}}(L)$, we have
   \begin{align*}
   	\ord_{\mathfrak{p}}(\mathfrak{v}(M))=\ord_{\mathfrak{p}}(\mathfrak{v}(L)\mathfrak{n}(L)^{-2})=\ord_{\mathfrak{p}}(a_{\mathfrak{p}}(L))=R_{\mathfrak{p}}(L).
   \end{align*}
  Combining this with the third equality in Proposition \ref{prop:a-R-alpha}(ii), we have
   \begin{align*}
   	&\min\{\ord_{\mathfrak{p}}( \mathfrak{v}(M)\mathfrak{d}_{\mathfrak{p}}(-dV)/dV),\ord_{\mathfrak{p}}(4\mathfrak{v}(M))/2\}\\
   	=\;&\min\{R_{\mathfrak{p}}(L)+d_{\mathfrak{p}}(-a_{\mathfrak{p}}(L)), R_{\mathfrak{p}}(L)/2+e_{\mathfrak{p}}\}=\alpha_{\mathfrak{p}}(L).
   \end{align*}
   So $\psi(M,\mathfrak{p})=\psi_{\mathfrak{p}}(L)$ and thus $a(M,\mathfrak{p})=\upsilon_{\mathfrak{p}}(L)$. From the congruence in Proposition \ref{prop:a-R-alpha}(ii), one can easily rewrite the term $e(M,\mathfrak{p})$ as $u_{\mathfrak{p}}(L)$. 
\end{proof} 

We provide a local formula to compute the ideal $\mathfrak{c}(L)$ and build a bridge with the local density, GK invariant, and EGK invariant discussed in \cite[Appendix A]{cho_localdensity-2020}. 
\begin{prop}\label{prop:cL}
	Suppose that $F$ is an algebraic number field and $L$ is a binary $\mathcal{O}_{F}$-lattice. Let $\mathfrak{p}\in \Omega_{F}\backslash \infty_{F}$. Then 
	\begin{align}\label{cL}
	 \ord_{\mathfrak{p}}(\mathfrak{c}(L))=\dfrac{R_{\mathfrak{p}}(L)+S_{\mathfrak{p}}(L)}{2},
	\end{align}
    where 
    \begin{align*}
         S_{\mathfrak{p}}(L):=\begin{cases}
         	-1 &\text{if $R_{\mathfrak{p}}(L)$ is odd}, \\
          	2\min\{e_{\mathfrak{p}},\lfloor d_{\mathfrak{p}}(-a_{\mathfrak{p}}(L))/2\rfloor\}  &\text{if $R_{\mathfrak{p}}(L)$ is even}.
         \end{cases}
    \end{align*}
    \end{prop}
    \begin{proof}
    	From definition, $\ord_{\mathfrak{p}}(\mathfrak{v}(L)\mathfrak{n}(L)^{-2})=R_{\mathfrak{p}}(L) $. To determine $\ord_{\mathfrak{p}}(4D_{F^{\prime}/F}^{-1}\mathcal{O}_{F})$, we apply the method in the proof of \cite[Theorem 2(1)]{boylan-skoruppa_relative-quad-extension_2025} to compute the order of the discriminant ideal of $F^{\prime}/F$. So, if $\mathfrak{p}$ is non-dyadic, we have
    	\begin{align*}
    		\ord_{\mathfrak{p}}(4^{-1}D_{F^{\prime}/F}\mathcal{O}_{F})=\ord_{\mathfrak{p}}(dV)-2\lfloor \ord_{\mathfrak{p}}(dV)/2\rfloor=
    		\begin{cases}
    			1  &\text{if $\ord_{\mathfrak{p}}(dV)$ is odd}, \\
    			0   &\text{if $\ord_{\mathfrak{p}}(dV)$ is even};
    		\end{cases}
    	\end{align*}
    	if $\mathfrak{p}$ is dyadic, since $d_{\mathfrak{p}}(-dV)=d_{\mathfrak{p}}(-a_{\mathfrak{p}}(L))$ from Proposition \ref{prop:a-R-alpha}(ii), we have
    	\begin{align*}
    		\ord_{\mathfrak{p}}(4^{-1}D_{F^{\prime}/F}\mathcal{O}_{F})=
    		\begin{cases}
    			\ord_{\mathfrak{p}}(dV)-2\lfloor \ord_{\mathfrak{p}}(dV)/2\rfloor=1   &\text{if $\ord_{\mathfrak{p}}(dV)$ is odd}, \\
    			\ord_{\mathfrak{p}}(dV)-2\lfloor \ord_{\mathfrak{p}}(dV)/2\rfloor-2t=-2t  &\text{if $\ord_{\mathfrak{p}}(dV)$ is even}, \\
    		\end{cases}
    	\end{align*}
    	where $t=\min\{e_{\mathfrak{p}},\lfloor d_{\mathfrak{p}}(-a_{\mathfrak{p}}(L))/2\rfloor\} $. Also, when $\mathfrak{p}$ is non-dyadic and $\ord_{\mathfrak{p}}(dV)$ is even,  $t=e_{\mathfrak{p}}=0$. So unifying both cases and from $\ord_{\mathfrak{p}}(dV)\equiv R_{\mathfrak{p}}(L)\pmod{2}$, we have
    	\begin{equation}\label{4DEF}
    		\begin{aligned}
    			\ord_{\mathfrak{p}}(4^{-1}D_{F^{\prime}/F}\mathcal{O}_{F})&=\begin{cases}
    				1   &\text{if $R_{\mathfrak{p}}(L)$ is odd},\\
    				-2t  &\text{if $R_{\mathfrak{p}}(L)$ is even}.
    			\end{cases} \\
    			&=-S_{\mathfrak{p}}(L).
    		\end{aligned}        
    	\end{equation}
    	Therefore, $\ord_{\mathfrak{p}}(4D_{F^{\prime}/F}^{-1}\mathcal{O}_{F})=S_{\mathfrak{p}}(L)$. So
    	\begin{align*}
    			2\ord_{\mathfrak{p}}(\mathfrak{c}(L))=\ord_{\mathfrak{p}}(\mathfrak{v}(L)\mathfrak{n}(L)^{-2})+\ord_{\mathfrak{p}}(4D_{F^{\prime}/F}^{-1}\mathcal{O}_{F})= R_{\mathfrak{p}}(L)+S_{\mathfrak{p}}(L) ,
    	\end{align*}
    	as desired.
    \end{proof} 
 \begin{prop}\label{prop:fB-dB}
 	Suppose that $F$ is an algebraic number field and $L$ is a binary $\mathcal{O}_{F}$-lattice. Assume that $\mathfrak{p}$ is a dyadic prime and $L_{\mathfrak{p}}$ is primitive. 
 	
 	Let $B$ be the associated Gram matrix of $L_{\mathfrak{p}}$, $D_{B}=-4\det B$ and $K=F_{\mathfrak{p}}(\sqrt{D_{B}})$. Define $d_{B}=\ord_{\mathfrak{p}}(D_{K/F_{\mathfrak{p}}}\mathcal{O}_{F_{\mathfrak{p}}})$ and $f_{B}=(\ord_{\mathfrak{p}}(D_{B})-d_{B})/2$ (cf. \cite[p.\hskip 0.1cm 1259]{cho_localdensity-2020}). Then 
 	\begin{align*}
 		f_{B}=(R_{\mathfrak{p}}(L)+S_{\mathfrak{p}}(L))/2\quad\text{and}\quad
 		d_{B}=2e_{\mathfrak{p}}-S_{\mathfrak{p}}(L).
 	\end{align*}
 \end{prop}
 \begin{re}\label{re:fB-dB}
 	With Proposition \ref{prop:fB-dB}, one can express or compute the local density $\beta_{\mathfrak{p}}(L)$, $\gk(L_{\mathfrak{p}}\perp -L_{\mathfrak{p}})$ and $\egk(L_{\mathfrak{p}}\cap \pi_{\mathfrak{p}}^{i}L_{\mathfrak{p}}^{\#})^{\le 1}$ in terms of invariants $a(L_{\mathfrak{p}})$ and $R(L_{\mathfrak{p}})$, by applying \cite[Propositions A6, A8 and A11]{cho_localdensity-2020}. See Theorem \ref{thm:localinfor-ADC} and Example \ref{ex:GK-EGK} for details.
 \end{re}
 \begin{proof}
 	Note that $D_{B}=-dV_{\mathfrak{p}}$ in $F_{\mathfrak{p}}^{\times}/F_{\mathfrak{p}}^{\times 2}$. It follows from \cite[\S 2]{frolich_discriminants_1960} $d_{B}=\ord_{\mathfrak{p}}(D_{F^{\prime}/F}\mathcal{O}_{F})$. Hence the formula for $d_{B}$ follows from \eqref{4DEF}. Since $L_{\mathfrak{p}}$ is primitive, $\ord_{\mathfrak{p}}(\mathfrak{n}(L))=0$. Hence
 	\begin{align*}
 		2\ord_{\mathfrak{p}}(\mathfrak{c}(L))=\ord_{\mathfrak{p}}(\mathfrak{v}(L))-\ord_{\mathfrak{p}}(4^{-1}D_{F^{\prime}/F}\mathcal{O}_{F}) &=\ord_{\mathfrak{p}}(4^{-1}D_{B})-\ord_{\mathfrak{p}}(4^{-1}D_{F^{\prime}/F}\mathcal{O}_{F})\\
 		&=\ord_{\mathfrak{p}}(D_{B})-d_{B}=2f_{B}.
 	\end{align*}
 	Thus $f_{B}=\ord_{\mathfrak{p}}(\mathfrak{c}(L))$. So the formula for $f_{B}$ follows by Proposition \ref{prop:cL}.
 \end{proof}  

Following \cite[\S 1]{pfeuffer_Einklassige_1971},
we denote by $A_{\mathfrak{p}^{k}}^{0}(L,L)$ the number of linear representations $\sigma$ from $L$ to $L$ modulo $\mathfrak{p}^{k}L$ such that $Q(\sigma(x))\equiv Q(x)\pmod{2\mathfrak{p}^{k}}$ for all $x\in L$, i.e.,
\begin{align*}
	A_{\mathfrak{p}^{k}}^{0}(L,L):=|\{\sigma:L\to L/\mathfrak{p}^{k}L\mid Q(\sigma(x))\equiv Q(x)\pmod{2\mathfrak{p}^{k}}\;\text{for all $x\in L$}\}|.
\end{align*}
Set $	\gamma_{\mathfrak{p}}^{0}(L,L):=2^{-1}\lim_{k\to \infty}A_{\mathfrak{p}^{k}}^{0}(L,L)N(\mathfrak{p})^{-k(\ell(\ell-1)/2)}$ and define
\begin{align*}
 \gamma_{\mathfrak{p}}(L,L):=&N(\mathfrak{p})^{e_{\mathfrak{p}}\ell}\gamma^{0}_{\mathfrak{p}}(L,L),
\end{align*}
as the \textit{local density} of $L$. According to \cite[Hilfssatz 3]{pfeuffer_Einklassige_1971},  $\gamma_{\mathfrak{p}}(L,L)$ coincides with $d_{\mathfrak{p}}(L,L)$ as defined in \cite[Definition]{pfeuffer_Einklassige_1971}, but not with $\beta_{\mathfrak{p}}(L)$ as defined in \cite[\S 2.4]{cho_localdensity-2020}. In fact, when $\ell=2$, we have $\beta_{\mathfrak{p}}(L)=N(\mathfrak{p})^{e_{\mathfrak{p}}}\gamma_{\mathfrak{p}}(L,L)$.

Note that the proof of \cite[Hilfssatz 7]{pfeuffer_Einklassige_1971} is also applicable to $L^{(c)}$ with $\ord_{\mathfrak{p}}(c)\le 0$. Hence we have
\begin{align}\label{scaleL}
\gamma_{\mathfrak{p}}(L^{(c)},L^{(c)})=N(\mathfrak{p})^{\ord_{\mathfrak{p}}(c)(\ell(\ell+1)/2)}\gamma_{\mathfrak{p}}(L,L)
\end{align}
for $c\in F^{\times}$.
 
When $L$ is binary, the formulas for $\gamma_{\mathfrak{p}}(L,L)$ were derived by Siegel in \cite{siegel_III_1937} for non-dyadic cases, and by Pfeuffer in \cite{pfeuffer_Einklassige_1971,pfeuffer_binarer_1978} for dyadic cases. Based on invariants $\mathfrak{s}(L_{\mathfrak{p}})$, $\mathfrak{v}(L_{\mathfrak{p}})$, $a(L_{\mathfrak{p}})$, $R(L_{\mathfrak{p}})$ and $\alpha(L_{\mathfrak{p}})$, we provide a unified formula of $\gamma_{\mathfrak{p}}(L,L)$ for all $\mathfrak{p}\in \Omega_{F}\backslash \infty_{F}$, and the corresponding formula for $1/m(L)$.
\begin{thm}\label{thm:localdensity}
	Suppose that $F$ is non-dyadic or dyadic, and $L\cong \prec a_{1},a_{2}\succ$ is a binary $\mathcal{O}_{F}$-lattice relative to a good BONG. Put $a=a(L)$, $R=R(L)$ and $\alpha=\alpha(L)$. Then the local density formula of $L$ is given by
	\begin{align}\label{gamma}
		\gamma(L,L)=N(\mathfrak{
			s}(L))N(\mathfrak{v}(L))N(\mathfrak{p})^{e+\lfloor t/2\rfloor}r,
	\end{align}
	where the pair $(t,r)$ is given as the following formulas:
	 \begin{enumerate}[itemindent=-0.5em,label=\rm (\roman*)]
	 \item If $R>0$, then 
	\begin{align*}
		(t,r)=\begin{cases}
			(\alpha,1)    &\text{if  $R \le 2e$ and $  \alpha <R +d(-a )$}, \\
			(\alpha,2)    &\text{if  $R \le 2e$ and $  \alpha =R +d(-a )$}, \\
			(2e,2)    &\text{if $R >2e$}.
		\end{cases}
	\end{align*}
	\item  If $R\le 0$, then 
	\begin{align*}
		(t,r)=\begin{cases}
			(e-R/2,1)    &\text{if  $R>-2e$ and $d(-a)>e-R/2$}, \\
			(d(-a),2)    &\text{if $d(-a)\le e-R/2$}, \\
			(2e,1-\chi(\mathfrak{p})/N(\mathfrak{p}))    &\text{if $R=-2e$},
		\end{cases}
	\end{align*} 
\end{enumerate}
where the first two cases in (i) (resp. (ii)) are ignored, when $F$ is non-dyadic, i.e., $e=0$.
\end{thm}
\begin{re}\label{re:local-density}
	  By Proposition \ref{prop:a-R-alpha}(i) and (iii), we see that the values of $N(\mathfrak{s}(L))$, $N(\mathfrak{v}(L))$, $R(L)$, and $\alpha(L)$ depend only on $a(L)$, $R_{1}(L)$, and $R_{2}(L)$. Therefore, the local density of binary lattices over arbitrary non-archimedean local fields is completely determined by the invariants $a(L)$ and $R_{i}(L)$ for $i=1,2$.
\end{re}
\begin{proof}
 Let $  \mathfrak{s}(L)=c\mathcal{O}_{F}$. Write $L=L^{\prime (c)}$ with for some binary $\mathcal{O}_{F}$-lattice $L^{\prime}$ with $\mathfrak{s}(L^{\prime})=\mathcal{O}_{F}$. Then we have
 \begin{align*}
 \lhs=\gamma(L,L)=N(\mathfrak{p})^{3\ord(c)}\gamma(L^{\prime},L^{\prime})=N(\mathfrak{s}(L))^{3}\gamma(L^{\prime},L^{\prime}).
 \end{align*}
  Let $s=\ord(\mathfrak{v}(L^{\prime}))$. From \cite[\S 82J]{omeara_quadratic_1963}, we have $\mathfrak{p}^{s}=\mathfrak{v}(L)\mathfrak{s}(L)^{-2}$ and thus
  \begin{align}\label{Nps}
  	N(\mathfrak{p})^{s}=N(\mathfrak{v}(L))N(\mathfrak{s}(L))^{-2}.
  \end{align}
   Assume that $F$ is non-dyadic. From \cite[Hilfssatz 56]{siegel_III_1937}, we have
	\begin{align*}
		\gamma(L^{\prime},L^{\prime})=\begin{cases}
			1-\chi(\mathfrak{p})/N(\mathfrak{p}) &\text{if $L^{\prime}$ is unimodular},\\
			2N(\mathfrak{p})^{s} &\text{if $L^{\prime}$ is not unimodular}.\\
		\end{cases}
	\end{align*} 

	If $R=0=-2e$, then $(t,r)=(0,1-\chi(\mathfrak{p})/N(\mathfrak{p}))$. Since $L^{\prime} $ is unimodular, we have $s=0$. It follows from \eqref{Nps} that $N(\mathfrak{v}(L))=N(\mathfrak{s}(L))^{2}$. Hence  $\rhs= N(\mathfrak{s}(L))^{3}(1-\chi(\mathfrak{p})/N(\mathfrak{p}))=\lhs$. 
	 If $R>0=-2e$, then $(t,r)=(0,2)$. By \eqref{Nps}, we have $\rhs=2N(\mathfrak{v}(L))N(\mathfrak{s}(L))=N(\mathfrak{s}(L))^{3}(2N(\mathfrak{p})^{s})=\lhs$.
	
	Assume that $F$ is dyadic. From \cite[p.\hskip 0.1cm 105]{pfeuffer_binarer_1978}, we have
	 \begin{align*}
		\gamma(L^{\prime},L^{\prime})=N(\mathfrak{p})^{s+e+\lfloor t/2\rfloor}r,
	\end{align*}
	 with appropriate parameters $s$, $t$ and $r$. Note from \eqref{Nps} that $ N(\mathfrak{s}(L))^{3}N(\mathfrak{p})^{s}=N(\mathfrak{v}(L))N(\mathfrak{s}(L))$. So it remains to translate the parameters $t$ and $r$ to the invariants introduced above.
	
	\textbf{Case I: $R>0$}
	
	In this case,  $L=L^{\prime (c)}$ with $L^{\prime}=\prec 1,a\succ$ and $c=a_{1}$. From \cite[p.\hskip 0.1cm 104]{pfeuffer_binarer_1978}, we have $\mathfrak{p}^{t}=4\mathcal{O}_{F}+2\mathfrak{p}^{\lfloor s/2\rfloor}+\mathfrak{d}(-dL^{\prime})$, and $r=1$ or $2$, according as condition (a) or (b) holds:
	\begin{enumerate}
		\item[(a)] $\mathfrak{d}(-dL^{\prime})\subsetneq \mathfrak{p}^{t}\subseteq \mathfrak{p}^{s} $;
		
		\item[(b)] $\mathfrak{d}(-dL^{\prime})=\mathfrak{p}^{t}$ or $\mathfrak{p}^{s}\subsetneq \mathfrak{p}^{t}$.
	\end{enumerate}
We claim that
	\begin{align*}
		&s=R,\quad \ord(\mathfrak{d}(-dL^{\prime}))=R +d(-a),\quad\text{and}\quad t=\begin{cases}
			\alpha  &\text{if $R\le 2e$}, \\
			2e  &\text{if $R>2e$}.
		\end{cases}
	\end{align*}
		First, we have $R_{1}(L^{\prime})=0$ and $R_{2}(L^{\prime})=\ord(a)=R$. By Proposition \ref{prop:a-R-alpha}(iii) and (ii),$s=\ord(\mathfrak{v}(L^{\prime}))=R_{1}(L^{\prime})+R_{2}(L^{\prime})=R$, and so $\ord(\mathfrak{d}(-dL^{\prime}))=\ord(\mathfrak{v}(L^{\prime}))+d(-a(L^{\prime}))=R +d(-a)$. 
 	  Hence $  t =\min\{2e, \lfloor R/2\rfloor+e,R+d(-a)\}$.
 	  
 	   Furthermore, if $R\le 2e$ is even, then by Proposition \ref{prop:Rproperty}(i), $\alpha \le 2e$. Hence, by Proposition \ref{prop:a-R-alpha}(i),  $t=\min\{2e,R/2+e,R+d(-a)\}=\min\{2e,\alpha\}=\alpha$. If $R\le 2e$ is odd, then $R\le 2e-1$ and $d(-a)=0$. So $\lfloor R/2\rfloor +e\ge R+d(-a)=R$. Therefore, $t=\min\{2e,R\}=R$. So, by Proposition \ref{prop:Rproperty}(iii), $t=R=\alpha$. If $R>2e$, then $\lfloor R/2\rfloor+e\ge2e$ and $R+d(-a)\ge 2e$. So  $t=2e$.  The claim is proved.
	
	If $R \le 2e$, by the claim and Proposition \ref{prop:Rproperty}(iii), we have $s=R\le t=\alpha$. So $\mathfrak{p}^{t}\subseteq \mathfrak{p}^{s}$ holds trivially, and $\mathfrak{p}^{s}\subsetneq \mathfrak{p}^{t}$ fails. Therefore, (a) is equivalent to $R+d(-a)>\alpha $; (b) is equivalent to $R+d(-a)=\alpha$.
	
	If $R>2e$, by the claim and Proposition \ref{prop:Rproperty}(i), we have $t=2e<s=R$. Hence $\mathfrak{p}^{t} \subseteq \mathfrak{p}^{s} $ fails, so does (a). Clearly, (b) is equivalent to $R+d(-a)=2e$ or $R>2e$; however, the former condition would imply that $2e=R +d(-a)\ge R $, a contradiction. 
	
	\textbf{Case II: $R\le 0$}
	
	Write $a_{i}=\varepsilon_{i}\pi^{R_{i}}$ with $\varepsilon_{i}\in \mathcal{O}_{F}^{\times}$. By \cite[Corollary 3.4(iii)]{beli_integral_2003} and \cite[93:17]{omeara_quadratic_1963}, we have $L=L^{\prime (c)}$, where $c=\pi^{(R_{1}+R_{2})/2}$ and
	\begin{align*}
		L^{\prime}=A(\varepsilon_{1}\pi^{-R/2}, \varepsilon_{1}^{-1}\pi^{R/2}(1+\varepsilon_{1}\varepsilon_{2}))\cong \prec \varepsilon_{1}\pi^{-R/2}, \varepsilon_{2}\pi^{R/2}\succ ,
	\end{align*}
	with $R\in [-2e,0]^{E}$ (from \eqref{eq:BONGs} and Proposition \ref{prop:Rproperty}(iv)). As \cite[p.\hskip 0.1cm 104]{pfeuffer_binarer_1978}, put $\mathfrak{p}^{t}=\mathfrak{n}(L^{\prime})\mathfrak{w}(L^{\prime})$, and then $r=1$ or $2$ or $1-\chi(\mathfrak{p})/N(\mathfrak{p})$, according as condition (c) or (d) or (e) holds:
	\begin{enumerate}
		\item[(c)] $4\mathfrak{o}\subsetneq \mathfrak{p}^{t}$ and $\mathfrak{d}(-dL^{\prime})\subsetneq \mathfrak{p}^{t}$;  
		\item[(d)] $ 4\mathfrak{o}\subsetneq \mathfrak{p}^{t}= \mathfrak{d}(-dL^{\prime}) $;
		\item[(e)] $4\mathfrak{o}=\mathfrak{p}^{t}$.
	\end{enumerate} 
	We claim that
	\begin{align*}
  	\ord(\mathfrak{d}(-dL^{\prime}))=d(-a)\quad\text{and}\quad t=\min\{e-R/2,d(-a)\}.
	\end{align*}
	  Clearly, $ R_{1}(L^{\prime})=-R/2$ and $R_{2}(L^{\prime})=R/2$. By Proposition \ref{prop:a-R-alpha}(i), we see that  $R(L^{\prime})=R$,  $d(-a(L^{\prime}))=d(-\varepsilon_{1}\varepsilon_{2})=d(-a)$ and $\alpha(L^{\prime})=\alpha$. Since $R(L^{\prime})$ is even, by Proposition \ref{prop:a-R-alpha}(ii) and (iii), we further have 
	\begin{align*}
		\ord(\mathfrak{d}(-dL^{\prime}))=\ord(\mathfrak{v}(L^{\prime}))+d(-a(L^{\prime}))=R_{1}(L^{\prime})+R_{2}(L^{\prime})+d(-a)=d(-a).
	\end{align*}
	 By Proposition \ref{prop:a-R-alpha}(iii), $\ord(\mathfrak{n}(L^{\prime}))=R_{1}(L^{\prime})=-R/2$ and $\ord(\mathfrak{w}(L^{\prime}))=R(L^{\prime})+\alpha(L^{\prime})=-R/2+\alpha$. So
	 \begin{align*}
	 	t=\ord(\mathfrak{n}(L^{\prime})\mathfrak{w}(L^{\prime}))=-R/2+(-R/2+\alpha)=\min\{e-R/2,d(-a)\},
	 \end{align*} 
	 as claimed.
	 
	 Note that if $R=-2e$, then by \eqref{eq:BONGs}, $d(-a)\ge -R=2e$ and thus $t=2e$ from the claim.
	 
	If $d(-a)>e-R/2$, by the claim, we have $t=e-R/2<\ord(\mathfrak{d}(-dL^{\prime}))=d(-a)$. Thus (d) cannot happen. Clearly, (c) is equivalent to $2e>e-R/2$, i.e., $R>-2e$; and (e) is equivalent to $2e=e-R/2$, i.e., $R=-2e$ and  $d(-a)=\infty$.
	
	If $d(-a)\le e-R/2$, by the claim again, we have $t=d(-a)=\ord(\mathfrak{d}(-dL^{\prime}))$. Thus (c) cannot happen. Clearly, (d) is equivalent to $d(-a)<2e$; note that $e-R/2<2e$, otherwise, $R=-2e$, which implies $d(-a)\ge 2e$, a contradiction. So (d) is equivalent to $d(-a)\le e-R/2$. Clearly, (e) is equivalent to $2e=d(-a)\le e-R/2$, i.e., $R=-2e$ and $d(-a)=2e$.   
\end{proof}
 \begin{thm}\label{thm:mass-2}
	Suppose that $F$ is a totally real number field of degree $n$ and $L$ is a positive definite binary $\mathcal{O}_{F}$-lattice. Then
	\begin{align*}
		\dfrac{1}{m(L)}=\dfrac{\pi^{n}N(\mathfrak{s}(L)) }{\sqrt{D_{F}N(\mathfrak{v}(L))}} \prod_{\mathfrak{p}\in\Omega_{F}\backslash \infty_{F}} N(\mathfrak{p})^{\lfloor t_{\mathfrak{p}}(L)/2\rfloor}r_{\mathfrak{p}}(L),
	\end{align*}
	where the pair $(t_{\mathfrak{p}}(L),r_{\mathfrak{p}}(L))$ is given by the formulas (i) and (ii) in Theorem \ref{thm:localdensity}.
\end{thm}
\begin{proof}
	 When $L$ is binary and positive definite, from \cite{siegel_III_1937} and Theorem \ref{thm:localdensity}, we have
	\begin{align*}
	   \dfrac{1}{m(L)}=\dfrac{1}{\sqrt{D_{F}}}\prod_{\mathfrak{p}\in \Omega_{F}}\gamma_{\mathfrak{p}}(L,L),\quad  \prod_{\mathfrak{p}\in \infty_{F}}\gamma_{\mathfrak{p}}(L,L)=\dfrac{\pi^{n}}{2^{n}}\sqrt{N(\mathfrak{v}(L))^{-3}},\\
	\intertext{and}
		\prod_{\mathfrak{p}\in \Omega_{F}\backslash\infty_{F}}\gamma_{\mathfrak{p}}(L,L)=N(\mathfrak{s}(L))N(\mathfrak{v}(L))\prod_{\mathfrak{p}\in\Omega_{F}\backslash \infty_{F}} N(\mathfrak{p})^{e_{\mathfrak{p}}+\lfloor t_{\mathfrak{p}}(L)/2\rfloor}r_{\mathfrak{p}}(L).
	\end{align*}
	 Combining these and noting that $2^{n}=\prod_{\mathfrak{p}\in \Omega_{F}\backslash \infty_{F}}N(\mathfrak{p})^{e_{\mathfrak{p}}}$, we derive the desired mass formula. 
\end{proof}
\begin{thm}\label{thm:localinfor-ADC}
	Suppose that $F$ is an algebraic number field and $L$ is a binary ADC $\mathcal{O}_{F}$-lattice. Let $\mathfrak{p}\in \Omega_{F}\backslash \infty_{F}$, $\nu\in \{1,2\}$, $\varepsilon_{\mathfrak{p}}\in \mathcal{U}_{\mathfrak{p}}$ and $\delta_{\mathfrak{p}}\in \mathcal{U}_{\mathfrak{p}}\backslash\{1,\Delta_{\mathfrak{p}}\}$. Let $B$ be the Gram matrix of $L_{\mathfrak{p}}$. Then the above local quantities of $L$ are given by Tables \ref{tab:1}, \ref{tab:2}, \ref{tab:3}, \ref{tab:4} and \ref{tab:5}.
\end{thm}
\begin{re}
	In Table \ref{tab:4}, let $e_{\mathfrak{p}}\ge 1$ and $L_{\mathfrak{p}}^{\nu}=N_{\nu}^{2}(\Delta_{\mathfrak{p}})$ or $M_{\nu}^{2}(\Delta_{\mathfrak{p}})$ with $\nu\in \{1,2\}$. When $\nu=2$, $L_{\mathfrak{p}}^{\nu}=L_{\mathfrak{p}}^{2}$ is not primitive, we ignore the quantities $f_{B}$ and $d_{B}$. However, $L_{\mathfrak{p}}^{2}=L_{\mathfrak{p}}^{1 (\pi_{\mathfrak{p}})}$ and $L_{\mathfrak{p}}^{1}$ (or the associated matrix) is primitive. So from \eqref{scaleL}, we have $\beta(L_{\mathfrak{p}}^{2})=N(\mathfrak{p})^{3}\beta(L_{\mathfrak{p}}^{1})$. Hence one can apply the formula in \cite[Proposition A6]{cho_localdensity-2020} to compute $\beta(L_{\mathfrak{p}}^{1})$, thereby obtaining $\beta(L_{\mathfrak{p}}^{2})$. Alternatively, one can use the relation $\beta(L_{\mathfrak{p}})=N(\mathfrak{p})^{e_{\mathfrak{p}}}\gamma(L_{\mathfrak{p}},L_{\mathfrak{p}})$ (from definition), and then apply our formula \eqref{gamma} directly.
\end{re}
 \begin{proof}
 	From Theorems \ref{thm:nondyadic-ADC-binary} and \ref{thm:dyadic-binary}, if  $L_{\mathfrak{p}}$ is binary ADC, then it is either of the form $N_{\nu,r}^{2}(c)$, with  $\nu\in \{1,2\}$ and $c\in \mathcal{V}_{\mathfrak{p}}$ satisfying $(\nu,c)\not=(2,1)$, or of the form $M_{\nu}^{2}(\Delta_{\mathfrak{p}})$ with $\nu\in \{1,2\}$, whose structures have been determined. Thus, by brute-force checking, one can obtain the corresponding local quantities discussed above for binary ADC lattices.
 	
 	First, using Proposition \ref{prop:a-R-alpha}, one can obtain Table \ref{tab:1}, which lists the quantities related to invariants \eqref{a-Ri} and \eqref{R-alpha}. Based on these data, we derive Tables \ref{tab:2} and \ref{tab:3} using the formulas \eqref{uvw} and \eqref{gamma}.
 	
 	To obtain \ref{tab:5}, we only need to know the quantities $f_{B}$ and $d_{B}$ in Table \ref{tab:4} from \cite[Propositions A6, A8, and A11]{cho_localdensity-2020}. These two quantities can be computed using Proposition \ref{prop:cL} and Table \ref{tab:1}.
 \end{proof}

With the information in the tables of Theorem \ref{thm:localinfor-ADC}, we provide a series of examples, which indicate that the invariants $\gk(L_{\mathfrak{p}}\perp -L_{\mathfrak{p}})$ and $\egk(L_{\mathfrak{p}}\cap \pi_{\mathfrak{p}}^{i}L_{\mathfrak{p}}^{\#})^{\le 1}$ fail to determine  $\gamma_{\mathfrak{p}}(L_{\mathfrak{p}},L_{\mathfrak{p}})$  and  $\beta(L_{\mathfrak{p}})$ for each $e_{\mathfrak{p}}\ge 2$.
\begin{ex}\label{ex:GK-EGK}
	Let $e_{\mathfrak{p}}\ge 2$, $L_{\mathfrak{p}}^{\nu}=N_{\nu}^{2}(\delta_{\mathfrak{p}})$, with $\nu\in \{1,2\}$, $\delta_{\mathfrak{p}}\in \mathcal{U}_{\mathfrak{p}}$ and $d_{\mathfrak{p}}(\delta_{\mathfrak{p}})=2e_{\mathfrak{p}}-1$, and $L_{\mathfrak{p}}^{\prime}=M_{1}^{2}(\Delta_{\mathfrak{p}})$. Then from Table \ref{tab:5}, we have
	\begin{align*}
		&\gk(L_{\mathfrak{p}}^{\nu}\perp -L_{\mathfrak{p}}^{\nu})=\gk(L_{\mathfrak{p}}^{\prime}\perp -L_{\mathfrak{p}}^{\prime})=(0,1,1,2), \\ &\egk(L_{\mathfrak{p}}^{\nu}\cap \pi_{\mathfrak{p}}^{i}L_{\mathfrak{p}}^{\nu\#})^{\le 1}=\egk(L_{\mathfrak{p}}^{\prime}\cap \pi_{\mathfrak{p}}^{i}L_{\mathfrak{p}}^{\prime\#})^{\le 1}=\begin{cases}
			   \emptyset  &\text{if $e_{\mathfrak{p}}>2$},   \\
			   (1;1;1)   &\text{if $e_{\mathfrak{p}}=2$},
		\end{cases}
	\end{align*}
	for $i\in \jor(L_{\mathfrak{p}}^{\nu})=\jor(L_{\mathfrak{p}}^{\prime})=\{1-e_{\mathfrak{p}}\}$, but $\beta(L_{\mathfrak{p}}^{\nu})=2N(\mathfrak{p})^{2}\not=\beta(L_{\mathfrak{p}}^{\prime})=N(\mathfrak{p})^{2}$. 
	 
	 In fact, the different local densities can be explained by invariants \eqref{a-Ri} and \eqref{R-alpha}. Although $R(L_{\mathfrak{p}}^{\nu})=R(L_{\mathfrak{p}}^{\prime})=2-2e_{\mathfrak{p}}$, we have
	 \begin{align*}
	 	d_{\mathfrak{p}}(-a(L_{\mathfrak{p}}^{\nu}))=d_{\mathfrak{p}}(\delta_{\mathfrak{p}})=2e_{\mathfrak{p}}-1\not= d_{\mathfrak{p}}(-a(L_{\mathfrak{p}}^{\prime}))=d_{\mathfrak{p}}(\Delta_{\mathfrak{p}})=2e_{\mathfrak{p}}.
	 \end{align*}
	   Therefore, the first four quantities of $L_{\mathfrak{p}}^{\nu}$ and $L_{\mathfrak{p}}^{\prime}$ in Table \ref{tab:3} are identical except that $r(L_{\mathfrak{p}}^{\nu})=2$ and $r(L_{\mathfrak{p}}^{\prime})=1$. So it follows from our formula \eqref{gamma} that $\gamma(L_{\mathfrak{p}}^{\nu},L_{\mathfrak{p}}^{\nu})=2\gamma(L_{\mathfrak{p}}^{\prime},L_{\mathfrak{p}}^{\prime}) $ and thus $\beta(L_{\mathfrak{p}}^{\nu})=2\beta(L_{\mathfrak{p}}^{\prime})$. 
\end{ex}
\newpage
 	\begin{table}[h]
 	\begin{center}
 		\captionof{table}{The quantities $d_{\mathfrak{p}}(-a(L_{\mathfrak{p}}))$, $R(L_{\mathfrak{p}})$ and $\alpha(L_{\mathfrak{p}})$ related to  \eqref{a-Ri} and \eqref{R-alpha} for binary ADC lattices $L$.}
 		\label{tab:1}
 		\renewcommand\arraystretch{1.5}
 		\hskip -0.8cm \begin{tabular}{c|c|c|c|c}
 			\toprule[1.2pt]
 			$e_{\mathfrak{p}}$  &  $L_{\mathfrak{p}}$   &    $d_{\mathfrak{p}}(-a(L_{\mathfrak{p}}))$   & $R(L_{\mathfrak{p}})$   &  $\alpha(L_{\mathfrak{p}})$  \\
 			\hline
 			\multirow{3}*{\text{$0$}} & $N_{1}^{2}(1)$   &$ \infty $  & $0$  & $0$   \\
 			\cline{2-5}
 			& $N_{\nu}^{2}(\Delta_{\mathfrak{p}})$  &$ 2e_{\mathfrak{p}}  $  &$ 0 $ & $0$   \\
 			\cline{2-5}
 			& $N_{\nu}^{2}(\varepsilon_{\mathfrak{p}}\pi_{\mathfrak{p}})$  &$ 0 $  &$1 $ & $1/2$    \\
 			\hline
 			\multirow{5}*{\text{$\ge 1$}} & $N_{1}^{2}(1)$  &$\infty$ &$ -2e_{\mathfrak{p}} $  & $0$     \\
 			\cline{2-5}
 			& $N_{\nu}^{2}(\Delta_{\mathfrak{p}})$ &$2e_{\mathfrak{p}}$ &$ -2e_{\mathfrak{p}} $  &$ 0$    \\
 			\cline{2-5}
 			& $N_{\nu}^{2}(\delta_{\mathfrak{p}})$  &$d_{\mathfrak{p}}(\delta_{\mathfrak{p}})$ &$ 1-d_{\mathfrak{p}}(\delta_{\mathfrak{p}}) $  &$1$   \\
 			\cline{2-5}
 			& $N_{\nu}^{2}(\varepsilon_{\mathfrak{p}}\pi_{\mathfrak{p}})$ &$ 0 $  &$ 1$ & $1$   \\ 
 			\cline{2-5}
 			& $M_{\nu}^{2}(\Delta_{\mathfrak{p}})$  &$2e_{\mathfrak{p}}$  &$2-2e_{\mathfrak{p}} $   & $1$  \\ 
 			\bottomrule[1.2pt]
 		\end{tabular}
 	\end{center} 
 \end{table}
 \begin{table}[h]
 	\begin{center}
 		\captionof{table}{The quantities $\psi(L_{\mathfrak{p}})$, $u(L_{\mathfrak{p}})$, $v(L_{\mathfrak{p}})$ and $w(L_{\mathfrak{p}})$ related to  \eqref{uvw} for binary ADC lattices $L$.}
 		\label{tab:2}
 		\renewcommand\arraystretch{2}
 		\hskip -0.8cm \begin{tabular}{c|c|c|c|c|c}
 			\toprule[1.2pt]
 			$e_{\mathfrak{p}}$   &  $L_{\mathfrak{p}}$    &  $\psi(L_{\mathfrak{p}})$  &  $u(L_{\mathfrak{p}})$  & $\upsilon(L_{\mathfrak{p}})$     &   $w(L_{\mathfrak{p}})$  \\
 			
 			\hline
 			\multirow{3}*{\text{$0$}} & $N_{1}^{2}(1)$   &$ 0 $     & $ 2-\chi^{2}(\mathfrak{p})$ &$1$ & $ (2-\chi^{2}(\mathfrak{p}))(1-\dfrac{\chi(\mathfrak{p})}{N(\mathfrak{p})})$ \\
 			\cline{2-6}
 			& $N_{\nu}^{2}(\Delta_{\mathfrak{p}})$    & $0$  & $1$ &$ 1$  &$1-\dfrac{\chi(\mathfrak{p})}{N(\mathfrak{p})}$  \\
 			\cline{2-6}
 			& $N_{\nu}^{2}(\varepsilon_{\mathfrak{p}}\pi_{\mathfrak{p}})$    & $1$  & $ 2-\chi^{2}(\mathfrak{p})$ &$2 $ & $\dfrac{2-\chi^{2}(\mathfrak{p})}{2}(1-\dfrac{\chi(\mathfrak{p})}{N(\mathfrak{p})})$\\
 			\hline
 			\multirow{5}*{\text{$\ge 1$}} & $N_{1}^{2}(1)$     & $0$  & $2-\chi^{2}(\mathfrak{p})$ &$ 1$ & $(2-\chi^{2}(\mathfrak{p}))(1-\dfrac{\chi(\mathfrak{p})}{N(\mathfrak{p})})$ \\
 			\cline{2-6}
 			& $N_{\nu}^{2}(\Delta_{\mathfrak{p}})$   & $0$  & $1$ &$ 1$ & $1-\dfrac{\chi(\mathfrak{p})}{N(\mathfrak{p})}$\\
 			\cline{2-6}
 			& $N_{\nu}^{2}(\delta_{\mathfrak{p}})$    & $1$  & $1$ &$ 1$  & $1-\dfrac{\chi(\mathfrak{p})}{N(\mathfrak{p})}$  \\
 			\cline{2-6}
 			& $N_{\nu}^{2}(\varepsilon_{\mathfrak{p}}\pi_{\mathfrak{p}})$   & $1$  & $1$ &$1 $   & $  1-\dfrac{\chi(\mathfrak{p})}{N(\mathfrak{p})} $ \\ 
 			\cline{2-6}
 			& $M_{\nu}^{2}(\Delta_{\mathfrak{p}})$  & $1$  & $1$ &$ 1$   & $   1-\dfrac{\chi(\mathfrak{p})}{N(\mathfrak{p})} $  \\ 
 			\bottomrule[1.2pt]
 		\end{tabular}
 	\end{center}
 \end{table}
 \begin{table}[h]\small
 	\begin{center}
 		\captionof{table}{The quantities $N(\mathfrak{s}(L_{\mathfrak{p}}))$, $N(\mathfrak{v}(L_{\mathfrak{p}}))$, $t(L_{\mathfrak{p}})$, $r(L_{\mathfrak{p}})$ and $\gamma(L_{\mathfrak{p}},L_{\mathfrak{p}})$ related to  \eqref{gamma} for binary ADC lattices $L$.}
 		\label{tab:3}
 		\renewcommand\arraystretch{1.8}
 		\hskip -0.8cm \begin{tabular}{c|c|c|c|c|c|c}  
 			\toprule[1.2pt]
 			$e_{\mathfrak{p}}$ &$L_{\mathfrak{p}}$& $N(\mathfrak{s}(L_{\mathfrak{p}}))$ &$N(\mathfrak{v}(L_{\mathfrak{p}}))$ &$t(L_{\mathfrak{p}})$ &$r(L_{\mathfrak{p}})$ &$\gamma(L_{\mathfrak{p}},L_{\mathfrak{p}})$\\
 			\hline
 			\multirow{3}*{\text{$0$}} & $N_{1}^{2}(1)$   &$ 1 $  & $1$  & $ 0$  & $1-\dfrac{\chi(\mathfrak{p})}{N(\mathfrak{p})} $ &$1-\dfrac{\chi(\mathfrak{p})}{N(\mathfrak{p})}$  \\
 			\cline{2-7}
 			& $N_{\nu}^{2}(\Delta_{\mathfrak{p}})$  &$ 1 $  &$ 1 $ & $ 0$  & $1-\dfrac{\chi(\mathfrak{p})}{N(\mathfrak{p})}$ &$ 1-\dfrac{\chi(\mathfrak{p})}{N(\mathfrak{p})}$    \\
 			\cline{2-7}
 			& $N_{\nu}^{2}(\varepsilon_{\mathfrak{p}}\pi_{\mathfrak{p}})$  &$ 1 $  &$ N(\mathfrak{p}) $ & $0$  & $2$ &$ 2N(\mathfrak{p})$  \\
 			\hline
 			\multirow{7}*{\text{$\ge 1$}} & $N_{1}^{2}(1)$   &$ N(\mathfrak{p})^{-e_{\mathfrak{p}}} $  & $N(\mathfrak{p})^{-2e_{\mathfrak{p}}}$  & $2e_{\mathfrak{p}}$  & $1-\dfrac{\chi(\mathfrak{p})}{N(\mathfrak{p})}$ &$ N(\mathfrak{p})^{-e_{\mathfrak{p}}}(1-\dfrac{\chi(\mathfrak{p})}{N(\mathfrak{p})})$  \\
 			\cline{2-7}
 			& $N_{1}^{2}(\Delta_{\mathfrak{p}})$  &$ N(\mathfrak{p})^{-e_{\mathfrak{p}}} $  &$ N(\mathfrak{p})^{-2e_{\mathfrak{p}}} $ & $2e_{\mathfrak{p}}$  & $1-\dfrac{\chi(\mathfrak{p})}{N(\mathfrak{p})}$ &$N(\mathfrak{p})^{-e_{\mathfrak{p}}} (1-\dfrac{\chi(\mathfrak{p})}{N(\mathfrak{p})})$  \\
 			\cline{2-7}
 			& $N_{2}^{2}(\Delta_{\mathfrak{p}})$  &$ N(\mathfrak{p})^{1-e_{\mathfrak{p}}} $  &$ N(\mathfrak{p})^{2-2e_{\mathfrak{p}}}$ & $2e_{\mathfrak{p}}$  & $1-\dfrac{\chi(\mathfrak{p})}{N(\mathfrak{p})}$ &$ N(\mathfrak{p})^{3-e_{\mathfrak{p}}}(1-\dfrac{\chi(\mathfrak{p})}{N(\mathfrak{p})})$  \\
 			\cline{2-7}
 			& $N_{\nu}^{2}(\delta_{\mathfrak{p}})$   &$ N(\mathfrak{p})^{\frac{1-d_{\mathfrak{p}}(\delta_{\mathfrak{p}})}{2}}$  &$N(\mathfrak{p})^{1-d_{\mathfrak{p}}(\delta_{\mathfrak{p}})}$ & $d_{\mathfrak{p}}(\delta_{\mathfrak{p}})$  & $2$ &$ 2N(\mathfrak{p})^{e_{\mathfrak{p}}+1-d_{\mathfrak{p}}(\delta_{\mathfrak{p}})}$   \\
 			\cline{2-7}
 			& $N_{\nu}^{2}(\varepsilon_{\mathfrak{p}}\pi_{\mathfrak{p}})$ &$ 1 $  &$N(\mathfrak{p})$ & $1$  & $2$ &$ 2N(\mathfrak{p})^{e_{\mathfrak{p}}+1} $    \\ 
 			\cline{2-7}
 			& $M_{1}^{2}(\Delta_{\mathfrak{p}})$ &$  N(\mathfrak{p})^{1-e_{\mathfrak{p}}} $  &$ N(\mathfrak{p})^{2-2e_{\mathfrak{p}}} $ & $2e_{\mathfrak{p}}-1$  & $1$ &$ N(\mathfrak{p})^{2-e_{\mathfrak{p}}}$   \\ 
 			\cline{2-7}
 			& $M_{2}^{2}(\Delta_{\mathfrak{p}})$ &$ N(\mathfrak{p})^{2-e_{\mathfrak{p}}} $  &$ N(\mathfrak{p})^{4-2e_{\mathfrak{p}}} $ & $2e_{\mathfrak{p}}-1$  & $1$ &$ N(\mathfrak{p})^{5-e_{\mathfrak{p}}}$   \\ 
 			\bottomrule[1.2pt]
 		\end{tabular}
 	\end{center}
 \end{table}
 \begin{table}[h]
 	\begin{center}
 		\captionof{table}{The quantities $f_{B}$, $d_{B}$ and $\beta(L_{\mathfrak{p}})$ defined in \cite[Appendix A]{cho_localdensity-2020} for binary ADC lattices $L$.}
 		\label{tab:4}
 		\renewcommand\arraystretch{1.8}
 		\hskip -0.8cm \begin{tabular}{c|c|c|c|c}
 			\toprule[1.2pt]
 			$e_{\mathfrak{p}}$ &$L_{\mathfrak{p}}$& $f_{B}$ &$d_{B}$   &$\beta(L_{\mathfrak{p}})$\\
 			\hline
 			\multirow{7}*{\text{$\ge 1$}} & $N_{1}^{2}(1)$   &$ 0 $  & $0$    &$  1-\dfrac{\chi(\mathfrak{p})}{N(\mathfrak{p})}$  \\
 			\cline{2-5}
 			& $N_{1}^{2}(\Delta_{\mathfrak{p}})$  &$ 0 $  &$ 0 $  &$ 1-\dfrac{\chi(\mathfrak{p})}{N(\mathfrak{p})} $  \\
 			\cline{2-5}
 			& $N_{2}^{2}(\Delta_{\mathfrak{p}})$  &$   $  &$  $   &$ N(\mathfrak{p})^{3}(1-\dfrac{\chi(\mathfrak{p})}{N(\mathfrak{p})})$  \\
 			\cline{2-5}
 			& $N_{\nu}^{2}(\delta_{\mathfrak{p}})$   &$  0$  &$ 2e_{\mathfrak{p}}+1-d_{\mathfrak{p}}(\delta_{\mathfrak{p}})$  &$ 2N(\mathfrak{p})^{2e_{\mathfrak{p}}+1-d_{\mathfrak{p}}(\delta_{\mathfrak{p}})}$   \\
 			\cline{2-5}
 			& $N_{\nu}^{2}(\varepsilon_{\mathfrak{p}}\pi_{\mathfrak{p}})$ &$ 0 $  &$2e_{\mathfrak{p}}+1$   &$ 2N(\mathfrak{p})^{2e_{\mathfrak{p}}+1} $    \\ 
 			\cline{2-5}
 			& $M_{1}^{2}(\Delta_{\mathfrak{p}})$ &$  1 $  &$ 0 $   &$ N(\mathfrak{p})^{2}$   \\ 
 			\cline{2-5}
 			& $M_{2}^{2}(\Delta_{\mathfrak{p}})$ &$   $  &$   $   &$ N(\mathfrak{p})^{5}$   \\ 
 			\bottomrule[1.2pt]
 		\end{tabular}
 	\end{center}
 \end{table}
 \begin{table}[h]\small
 	\begin{center}
 		\captionof{table}{The quantities $\gk(L_{\mathfrak{p}}\perp -L_{\mathfrak{p}})$ and $\egk(L_{\mathfrak{p}}\cap \pi_{\mathfrak{p}}^{i}L_{\mathfrak{p}}^{\#})^{\le 1}$ with $i\in \jor(L_{\mathfrak{p}})$ defined in \cite[Appendix A]{cho_localdensity-2020} for binary ADC lattices $L$.}
 		\label{tab:5}
 		\renewcommand\arraystretch{1.8}
 		\hskip -0.8cm \begin{tabular}{c|c|c|c|c}
 			\toprule[1.2pt]
 			$e_{\mathfrak{p}}$ &$L_{\mathfrak{p}}$  &$\gk(L_{\mathfrak{p}}\perp -L_{\mathfrak{p}})$ &$\egk(L_{\mathfrak{p}}\cap\pi_{\mathfrak{p}}^{i} L_{\mathfrak{p}}^{\#})^{\le 1}$ &$\jor(L_{\mathfrak{p}})$\\
 			\hline
 			\multirow{5}*{\text{$\ge 1$}} & $N_{1}^{2}(1)$      & $(0,0,0,0)$  & $\begin{cases}
 				\emptyset &\text{if $e_{\mathfrak{p}}>1$}, \\
 				(2;1;-1)    &\text{if $e_{\mathfrak{p}}=1$}.
 			\end{cases}$ &$ \{-e_{\mathfrak{p}}\}$  \\
 			\cline{2-5}
 			& $N_{1}^{2}(\Delta_{\mathfrak{p}})$    & $(0,0,0,0)$  & $\begin{cases}
 				\emptyset &\text{if $e_{\mathfrak{p}}>1$}, \\
 				(2;1;-1)    &\text{if $e_{\mathfrak{p}}=1$}.
 			\end{cases}$ &$\{-e_{\mathfrak{p}}\} $  \\
 			\cline{2-5}
 			& $N_{\nu}^{2}(\delta_{\mathfrak{p}})$     & $(0,1,2e_{\mathfrak{p}}-d_{\mathfrak{p}}(\delta_{\mathfrak{p}}),2e_{\mathfrak{p}}+1-d_{\mathfrak{p}}(\delta_{\mathfrak{p}}))$  & $\begin{cases}
 				\emptyset  &\text{if $d_{\mathfrak{p}}(\delta_{\mathfrak{p}}) >3$}, \\
 				(1;1;1)    &\text{if $d_{\mathfrak{p}}(\delta_{\mathfrak{p}}) =3$}, \\
 				(1,1;0,1;1,0)    &\text{if $d_{\mathfrak{p}}(\delta_{\mathfrak{p}}) =1$}.
 			\end{cases}$ &$ \{\frac{1-d_{\mathfrak{p}}(\delta_{\mathfrak{p}})}{2}\}$   \\
 			\cline{2-5}
 			& $N_{\nu}^{2}(\varepsilon_{\mathfrak{p}}\pi_{\mathfrak{p}})$   & $(0,1,2e_{\mathfrak{p}},2e_{\mathfrak{p}}+1)$  & $(1,1;0,1;1,0)$ &$\{0,1\} $    \\ 
 			\cline{2-5}
 			& $M_{1}^{2}(\Delta_{\mathfrak{p}})$   & $(0,1,1,2)$  & $\begin{cases}
 				\emptyset &\text{if $e_{\mathfrak{p}}>2$}, \\
 				(1;1;1)    &\text{if $e_{\mathfrak{p}}=2$},\\
 				(1;0;1)    &\text{if $e_{\mathfrak{p}}=1$}.
 			\end{cases}$ &$\{1-e_{\mathfrak{p}}\}$   \\ 
 			\bottomrule[1.2pt]
 		\end{tabular}
 	\end{center}
 \end{table}
 
\section*{Acknowledgments} 
I would like to thank Prof. Constantin-Nicolae Beli for sharing his slides on the spinor exception of ternary quadratic forms, and thank Prof. Sungmun Cho for indicating recent progress on local density formulas for binary quadratic forms. This work was supported by a grant from the National Natural Science Foundation of China (Project No. 12301013).

\end{document}